\documentclass[11pt,a4paper]{amsart}
\usepackage{tikz, float}
\usepackage{comment}
\usepackage{caption}
\usepackage{amsmath,amsthm,amsfonts,graphicx,amssymb,amscd,dsfont,euscript,enumerate,verbatim,calc,mathtools,listings}
\usepackage[hidelinks]{hyperref}
\usepackage[nameinlink, noabbrev]{cleveref}
\newtheorem{thm}{Theorem}[section]

\newtheorem{cor}[thm]{Corollary}
\newtheorem{lem}[thm]{Lemma}
\newtheorem{clm}[thm]{Claim}

\newtheorem{prop}[thm]{Proposition}
\newtheorem{rmk}{Remark}

\theoremstyle{definition}
\newtheorem{defn}[thm]{Definition}
\newtheorem{example}{Example}

\newcommand{\pslnr}{\mathrm{PSL}_n(\mathbb{R})}
\newcommand{\pslnc}{\mathrm{PSL}_n(\mathbb{C})}
\newcommand{\slnc}{\mathrm{SL}_n(\mathbb{C})}
\newcommand{\slnr}{\mathrm{SL}_n(\mathbb{R})}
\newcommand{\sun}{\mathrm{SU}(n)}
\newcommand{\son}{\mathrm{SO}(n)}

\newcommand{\psltc}{\mathrm{PSL}_3(\mathbb{C})}
\newcommand{\psltr}{\mathrm{PSL}_3(\mathbb{R})}

\newcommand{\pslr}{\mathrm{PSL}_2(\mathbb{R})}
\newcommand{\pslc}{\mathrm{PSL}_2(\mathbb{C})}

\newtheoremstyle{named}%
{}{}{\itshape}{}{\bfseries}{.}{.5em}{\thmnote{#3}}
\theoremstyle{named}

\allowdisplaybreaks[1]
\numberwithin{equation}{section}
\begin{document}
\title{Dominating surface-group representations via Fock-Goncharov coordinates}
\author{Pabitra Barman}
\author{Subhojoy Gupta}
\address{Department of Mathematical Sciences, Indian Institute of Science Education and Research (IISER) Mohali, Knowledge City, Sector 81, S.A.S. Nagar (Mohali) 140306, Punjab, India}
\email{pabitrab@iisermohali.ac.in}
\address{Department of Mathematics, Indian Institute of Science, Bangalore 560012, India}
\email{subhojoy@iisc.ac.in}

\begin{abstract}
   Let $S$ be a punctured surface of  negative Euler characteristic. We show that given a generic representation $\rho:\pi_1(S) \rightarrow \pslnc$, there exists a positive representation $\rho_0:\pi_1(S) \rightarrow \pslnr$ that dominates $\rho$ in the Hilbert length spectrum as well as in the translation length spectrum, for the translation length in the symmetric space  $\mathbb{X}_n= \pslnc/\mathrm{PSU}(n)$.  Moreover, the $\rho_0$-lengths of peripheral curves remain unchanged. The dominating representation $\rho_0$ is explicitly described via Fock-Goncharov coordinates. Our methods are linear-algebraic, and involve weight matrices of weighted planar networks. 
\end{abstract}

\maketitle

\section{Introduction}
For a closed oriented surface $S$ of negative Euler characteristic, the Teichm\"{u}ller space $\mathcal{T}(S)$ can be identified with a component of the space of Fuchsian (i.e.\ discrete and faithful)  representations from the surface-group $\pi_1(S)$ to $\pslr$, up to conjugation. An analogue of Teichm\"{u}ller space in the space of representations from its  surface-group into $\pslnr$ for $n>2$, is the component comprising Hitchin representations, the study of which goes back to the  work of Hitchin (\cite{Hitchin}) and Labourie (\cite{Labourie}), and is the central object of higher Teichm\"{u}ller theory (see \cite{Wienhard} for a survey). In \cite{FG06}, Fock-Goncharov introduced an alternative perspective, and characterized such representations in terms of a notion of positivity. Indeed, when $S$ is a punctured surface of negative Euler characteristic, they introduced certain birational coordinates on the space of $\pslnc$-representations of the surface-group when equipped with framings (see \Cref{FGparam}). The positive real points of this space correspond to \textit{positive} representations, the analogue of Hitchin representations, which are also discrete and faithful. This article is concerned with these positive representations, and their \textit{complex} deformations  into the space of punctured surface-group representations into $\pslnc$, and we shall use the Fock-Goncharov coordinates to prove  results about domination in the length spectrum. 

\vspace{.05in}

The notion of domination first arose in the context of closed surface-group representations into $\pslr$ and $\pslc$: 

\begin{defn}[Domination for $\pslc$-representations] Given two representations $\rho_1, \rho_2 :\pi_1(S) \rightarrow \pslc,\ \rho_2$ is said to \emph{dominate} $\rho_1$  \textit{in the translation length spectrum} if there exists $\lambda\le 1$ such that \[ \ell_{\rho_1}(\gamma) \le \lambda \cdot \ell_{\rho_2}(\gamma) \] for all $\gamma \in \pi_1(S)$, where $\ell_{\rho_i}(\gamma)$ denotes the translation length of $\rho_i(\gamma)$ in $\mathbb{H}^3$. The domination is said to be \emph{strict} if $\lambda<1$. 
\end{defn}


In \cite{GKW}  Gu\'eritaud-Kassel-Wolff showed that for a closed surface $S_g$ of genus $g\geq 2$ and a non-Fuchsian representation $\rho : \pi_1(S_g) \rightarrow \pslr$, there exists a Fuchsian representation $j: \pi_1(S_g) \rightarrow \pslr$ that strictly dominates $\rho$; this had a direct application in the construction of compact AdS 3-manifolds. In the proof, they have used the idea of ``folded" hyperbolic structures on surfaces. Around the same time, the work of Deroin-Tholozan \cite{DT} proved that for a representation $\rho : \pi_1(S_g) \rightarrow \pslc$, there exists a Fuchsian representation $j : \pi_1(S_g) \rightarrow \pslr$ that dominates $\rho$. Moreover, the domination is strict unless $\rho$ is itself Fuchsian; we note that it was known from work of Thurston in \cite{Thurston86} that  the hypothesis of $\rho$ being non-Fuchsian is necessary.  Their proof involved harmonic maps, and worked for representations into the isometry group of any $\mathrm{CAT}(-1)$ Hadamard manifold. In \cite{DMSV}, this result has been generalized to the case when the target-group is the isometry group of $\text{CAT}(-1)$ spaces. 

For the case of a punctured surface, the following domination result was proved in \cite{GSu}, and was implicit in \cite{Sagman}:

\begin{thm}\cite{GSu}\label{GSdom}
	Let $S$ be a punctured surface of negative Euler characteristic. For any $\rho:\pi_1(S) \to \pslc$ which is not Fuchsian, there exists a Fuchsian representation $j: \pi_1(S) \to \pslc$ that strictly dominates $\rho$, such that the lengths of the peripheral curves of $S$ are unchanged, i.e.\ the $j$ lies in the same relative representation variety. 
\end{thm}

The proof in \cite{GSu} is mostly geometric and it involves straightening of $\rho$-equivariant pleated planes in $\mathbb{H}^3$ for the generic (non-degenerate) case. In all these results, it was crucial that the target was negatively curved.

\vspace{.05in}

In this article, we generalize \Cref{GSdom} for representations into $\pslnc$ where $n>2$, and establish domination results with respect to \textit{two} notions of length:

\begin{itemize}
    
\item First, given a representation $\rho:\pi_1(S) \to \pslnc$, the \textit{Hilbert length} of a closed curve $\gamma$ is:
\begin{equation}\label{rlen}
l_{\rho}(\gamma):=\ln \Bigg| \frac{\lambda_n}{\lambda_1} \Bigg|,
\end{equation}
where $\lambda_n$ (respectively $\lambda_1$) is the largest (respectively smallest) eigenvalue in modulus of $\rho(\gamma)$. Note that in the case when $n=2$, and $\rho$ is a Fuchsian representation, this is the hyperbolic length of $\gamma$ (up to a factor which is a universal constant).  Moreover, in the case when $n=3$, for a representation into $\mathrm{PSL}_3(\mathbb{R})$ which is the holonomy of a convex $\mathbb{RP}^2$-structure on $S$, this is precisely the length of $\gamma$ in the \textit{Hilbert metric} on $S$.

\item Second, the \emph{translation length} of $\gamma$ in the symmetric space $ \mathbb{X}_n=\mathrm{PSL}_n(\mathbb{C})/\mathrm{PSU}(n)$, given by
\begin{equation}\label{sym_len}
\ell_{\rho}(\gamma)= \sqrt{(\log |\lambda_1(A)|)^2+\cdots + (\log |\lambda_n(A)|)^2}
\end{equation} where $\lambda_i$ for $1\leq i\leq n$ are the eigenvalues of $\rho(\gamma)$. We defer a discussion of this to \Cref{symspace}.

\end{itemize}

\vspace{.1in}

\noindent For both notions, one can define:

\begin{defn}[Domination for $\pslnc$-representations] Given two representations $\rho_1, \rho_2 :\pi_1(S) \rightarrow \pslnc,\ \rho_2$ is said to \emph{dominate} $\rho_1$  in Hilbert length spectrum (respectively, the translation length spectrum) if  $$\ell_{\rho_1}(\gamma) \le \ell_{\rho_2}(\gamma) \text{ for all } \gamma \in \pi_1(S),$$ where $\ell_{\rho_i}(\gamma)$ denotes the Hilbert length (respectively, the translation length). 
\end{defn}

\vspace{.1in}

\noindent We shall prove:

\begin{thm}\label{mainthm}
	Let $S$ be an oriented surface  of genus $g\geq 0$ having $k\geq 1$ punctures of negative Euler characteristic. For  a generic representation $\rho:\pi_1(S) \to \pslnc$,  there exists a positive representation $\rho_0: \pi_1(S) \to \pslnr$ that dominates $\rho$ in both the Hilbert length spectrum, and the translation length spectrum. Moreover, the lengths of peripheral curves (if any) remain unchanged. 
\end{thm} 

\vspace{.05in} 

\noindent This was conjectured in \cite[Conjecture 1.2]{GSu}. 

\vspace{.1in} 


\noindent \textbf{Remarks.}
\textit{
(i) We note that in \Cref{mainthm} we do not have a statement about \textit{strict} domination, since except the case when $n=2$, one positive representation \textit{can} dominate another. Indeed, for $n>2$,  it follows from the work in \cite{Zhang} that given a positive representation $\rho_0$ as in the statement of the theorem, there is always another positive representation $\rho_1$ 
that strictly dominates $\rho_0$ in the Hilbert length spectrum, keeping lengths of peripheral  curves fixed. }

\vspace{.05in}

 \noindent \textit{(ii) In our proof we shall use coordinates on the moduli space of \textit{framed} representations provided by Fock-Goncharov in \cite{FG06}. These are defined with respect to a choice of an ideal triangulation on  $S$, in terms of parameters that are certain generalizations of the complex cross-ratio that we shall provide an account of in  Section 2.   However, these coordinates are defined only on an open dense set, which necessitates the word ``generic" in the statement of our result. 
}
\vspace{.05in}

\noindent \textit{(iii) A representation which has all Fock-Goncharov coordinates real and positive is a \textit{positive representation} (see \Cref{puncHit}); for a closed surface the analogue of this condition gives precisely a Hitchin representation (see, for example \cite[Theorem 1.15]{FG06}. See \cite{CZZ} for a slightly stronger notion of a ``cusped Hitchin representation" which is, in addition, type-preserving. From the statement that the peripheral lengths are unchanged, it follows that our result above also holds if ``Hitchin" is replaced by ``cusped Hitchin". }

 \vspace{.1in}

 When $S$ is a closed surface, the analogue of Fock-Goncharov coordinates for the space of Hitchin representations was developed by Bonahon-Dreyer in \cite{BD14, BD17}; here, the coordinates are all real and positive. It seems to be unknown if a complexification of their coordinates parametrizes an open dense set in the space of $\pslnc$-representations of the closed surface-group if $n>2$. However, in a sequel, we shall use the work in this paper to prove a domination result for a generic \textit{Borel-Anosov} representation of a closed surface-group into $\pslnc$, and we expect that it also holds, more generally, for a generic ``$\lambda$-Borel Anosov" representation as introduced recently in \cite{Tengren2023}.

\vspace{.05in}

The dominating positive representation $\rho_0$ in \Cref{mainthm} is described explicitly, since it is the representation obtained by replacing each Fock-Goncharov coordinate  of $\rho$ (after a suitable framing) with its modulus. The bulk of our proof of domination relies on constructing a  weighted planar network (see \Cref{wpnwm}) for each curve $\gamma \in \pi_1(S)$, with weights of each edge an appropriate function of the Fock-Goncharov coordinates, such that $\rho(\gamma)$ is its associated weight matrix (see \Cref{wmatrix}). This relies on the relation of the monodromy of a closed curve on $S$ with the Fock-Goncharov coordinates of the ideal triangles it encounters along its trajectory; we briefly recall this aspect of Fock-Goncharov's work (the theory of ``snakes" and their moves) in Section 2. The rest of the proof relies on fairly standard techniques in linear algebra. 

\vspace{.05in}

Thus, we introduce a new algebraic technique of proving domination results, that we expect will be useful in various contexts. In particular, in a sequel we shall apply the results of this paper towards proving inequalites of the entropy of surface-group representations. For the case of representations into $\pslc$, such an application of a domination result was noted by Deroin-Tholozan in \cite{DT} - see Corollary D of their paper, where they observe that their main theorem implies a famous result of Bowen. In this context, we mention that it is also possible to define a length function for each $1\leq i\leq n$  by defining $\ell^i_\rho(\gamma) = \ln (\lvert \lambda_{i+1}\vert/\lvert \lambda_{i}\rvert) $  where $\lvert \lambda_1\rvert \leq \lvert \lambda_2\rvert \leq \cdots \lvert \lambda_n\rvert $ are the eigenvalues of $\rho(\gamma)$ (\textit{c.f.} \cite[\S2.6]{BD14}); however, we have observed in computer experiments that the domination we prove does \textit{not} hold for these individually. 

\vspace{.05in}

Our proof also introduces a notion of a \textit{bending deformation} of a positive representation that generalizes the bending of $\pslc$-representations (that takes, for example, a Fuchsian representation to a quasi-Fuchsian one). Indeed, the dominating Hitchin representation $\rho_0$ arises by such a deformation applied to $\rho$, and the set of all such deformations defines a complex torus in the $\pslnc$-representation variety that we shall call the \textit{bending fiber} (see \Cref{defn:bf}).  The following is an immediate corollary of the arguments (see \Cref{benFib}):

\begin{cor}
    In a bending fiber, the (unique) positive representation dominates all other representations. 
\end{cor}

Note that this is in similar vein to a conjectural domination statement in the $\pslnc$-representation variety of a closed surface-group when identified with the moduli space of Higgs bundles. Namely, it is conjectured that in each fiber of the Hitchin fibration, the Hitchin representation dominates every other representation  -- see \cite[\S5]{DT} and for finer versions of the conjecture, see \cite[Conjecture 1.12]{DaiLi1,DaiLi2}. It would be interesting to relate the two viewpoints.

\vspace{.05in}

It would also be interesting to prove such a domination result for surface-group representations into other (pairs of) Lie groups. In particular, we wonder if one can adapt our methods to prove a relative version of the domination result in \cite[Theorem B]{Tholozan} between $3$-Fuchsian representations and positive (cusped Hitchin) representations into $\psltr$.

\vspace{.1in}
\subsection*{Acknowledgements}  This represents part of the PhD thesis work of PB, written under SG's supervision. The authors would like to thank the anonymous referee,  Beatrice Pozzetti and Krishnendu Gongopadhyay for their comments on earlier versions. SG is grateful to Tengren Zhang for helpful conversations and Francois Labourie for his encouraging comments. PB acknowledges the support of the CSSP project fund of Dr. Srijan Sarkar.  This work was supported by the Department of Science and Technology, Govt.of India grant no. CRG/2022/001822, and by the DST FIST program - 2021 [TPN - 700661].

\section{Preliminaries} \label{prelims}
In this section, we will discuss some of the preliminaries that are essential for this article. In particular, we provide an account of Fock-Goncharov coordinates and their relation to the monodromy of curves, based on their paper \cite{FG06}. We shall borrow from other expository accounts in \cite{Palesi}, \cite[Appendix A]{GMN}, \cite{Douglas} and \cite{BD14}. In addition, we shall provide a brief account of the theory of weighted planar networks and some relevant facts from linear algebra, that will be useful later. 

\subsection{Tuples of flags and their invariants}\label{CF}
Recall that a complete {flag} $\mathcal{F}$ in $\mathbb{C}^n$ is a collection of subspaces $ \{0\}=F_0 \subset F_1 \subset \cdots \subset F_n=\mathbb{C}^n$ such that $dim(F_i)=i$. Throughout, we denote the space of all complete flags in $\mathbb{C}^n$ by $\mathcal{F}(\mathbb{C}^n)$. This subsection summarizes some of the basic definitions we need to describe the Fock-Goncharov coordinates. 

\begin{defn}[Compatible with basis]
	A flag $A \in \mathcal{F}(\mathbb{C}^n)$ is called compatible with a given ordered basis $\mathcal{B}=\{v_1,v_2,\cdots,v_n\}$ of $\mathbb{C}^n$ if $A_p=\langle v_1,v_2,\cdots,v_p \rangle$ for all $1\le p \le n$.
\end{defn}

\begin{defn}[General position for a pair]
	A pair of flags $(A,B) \in \mathcal{F}(\mathbb{C}^n)^2$ is said to be in general position if $A_p \cap B_q = 0$ whenever $p+q=n$. For such a pair, there is an ordered basis $\mathcal{B}_1=\{v_1,v_2,\cdots,v_n\}$ such that $A$ is compatible with $\mathcal{B}_1$ and $B$ is compatible with the opposite ordered basis $\mathcal{B}_2=\{v_n,v_{n-1},\cdots,v_1\}$. 
 \end{defn}

 \begin{defn}[General position for a triple]
 A triple of flags $(A,B,C)\in \mathcal{F}(\mathbb{C}^n)^3$ is said to be in general position if the sum \[ A_p+B_q+C_r= A_p \oplus B_q \oplus C_r=\mathbb{C}^n \] is direct whenever $p+q+r=n$. Equivalently, $A_p\cap B_q\cap C_r=0$ for all $p,q,r\ge 0$ such that $p+q+r=n$.
\end{defn}

\begin{defn}[Triple ratio]\label{defn:tr} 
	The $pqr$-\emph{triple ratio} of a triple of flags $(A,B,C)$ in general position is defined as \[ T_{pqr}(A,B,C)\coloneqq \frac{a_{(p-1)}\wedge b_{(q+1)}\wedge c_{(r)} }{a_{(p+1)}\wedge b_{(q-1)}\wedge c_{(r)}} \frac{a_{(p)}\wedge b_{(q-1)}\wedge c_{(r+1)} }{a_{(p)}\wedge b_{(q+1)}\wedge c_{(r-1)}} \frac{a_{(p+1)}\wedge b_{(q)}\wedge c_{(r-1)} }{a_{(p-1)}\wedge b_{(q)}\wedge c_{(r+1)}} \] for all positive integers $p,q,r$ satisfying $p+q+r=n$ where $a_{p'} \in \Lambda^{p'}(A_{p'}),\ b_{q'} \in \Lambda^{q'}(B_{q'}),\ c_{r'}\in \Lambda^{r'}(C_{r'})$ are some arbitrarily chosen non-zero elements.
\end{defn}
\begin{rmk}
	It is not hard to show that the triple ratio does not depend on such choices of non-zero elements. Moreover, the above definition is equivalent to considering the triple ratio of the three flags $\Bar{A}, \Bar{B}, \Bar{C}$ induced by $A, B$ and $C$ respectively in the three-dimensional quotient \[ \frac{\mathbb{C}^n}{A_{p-1}\oplus B_{q-1}\oplus C_{r-1}}. \]
\end{rmk}

\begin{example}
	For three flags $F_1,F_2,F_3$ in general position in $\mathbb{C}^3$, let $F_i=(v_i,f_i)$ where $v_i$ is a nonzero element in the one dimensional subspace of $F_i$ and $f_i \in (\mathbb{C}^3)^*$ (the dual space) satisfying  $f_i(v_i) =0$. Then the triple ratio  \[ T_{111}(F_1,F_2,F_3)=\frac{f_1 (v_2) f_2 (v_3) f_3 (v_1)}{f_1 (v_3) f_2 (v_1) f_3 (v_2)}. \]
\end{example} 

The basic result is that the above triple ratio is a complete invariant of triples of flags in general position, up to composition by elements of $\pslnc$: 

\begin{thm}\cite[\S9]{FG06} \label{FGmain}
	Two triples of flags $(A,B,C)$ and $(A',B',C')$ in general position have the same triple ratios in $\mathbb{C}^*$ if and only if there exists an element $\phi \in \pslnc$ such that $(\phi (A),\phi (B),\phi (C)) =(A',B',C')$. Moreover,  for positive integers $p,q,r$ satisfying $p+q+r=n$ and for any set of non-zero complex numbers $X_{pqr} \in \mathbb{C}^*$, there exists a unique (up to composition with an element of $\pslnc$) triple of flags $(A,B,C)$ such that \[ T_{pqr}(A,B,C)=X_{pqr}. \] 
\end{thm}

The above definitions have analogues for the case of \textit{quadruples} of flags:

\begin{defn}[General position for a quadruple]
 A quadruple of flags $(A,B,C,D)\in \mathcal{F}(\mathbb{C}^n)^4$ is said to be in general position if the sum \[ A_p+B_q+C_r+ D_s= A_p \oplus B_q \oplus C_r \oplus D_s=\mathbb{C}^n \] is direct whenever $p+q+r+s=n$. Equivalently, $A_p\cap B_q\cap C_r \cap D_s=0$ for all integers $p,q,r\ge 0$ such that $p+q+r=n$.
\end{defn}

\begin{defn}[Double ratio]
	Let $(A,B,C,C')$ be a quadruple of flags in general position. Then for $1\le r < n$, the \emph{$r$-th double ratio} of $(A,B,C,C')$ is defined as \[ 
	D_r(A,B,C,C')\coloneqq -\frac{a_r \wedge b_{n-r-1} \wedge c_1}{a_r \wedge b_{n-r-1} \wedge c'_1}\  \frac{a_{r-1} \wedge b_{n-r} \wedge c'_1}{a_{r-1} \wedge b_{n-r} \wedge c_1}, \] where $a_{p'} \in \Lambda^{p'}(A_{p'}),\ b_{q'} \in \Lambda^{q'}(B_{q'}),\ c_1\in \Lambda^1(C_1),\ c'_1\in \Lambda^1(C'_1)$ are some arbitrarily chosen non-zero elements.
\end{defn}


\begin{example}
	Recall that flags in $\mathbb{C}^2$ yield points in $\mathbb{CP}^1$ on projectivizing. In this case, a quadruple in general position is equivalent to asserting that the corresponding points in $\mathbb{CP}^1$ are pairwise distinct. The above double ratio reduces to the classical cross ratio of four distinct points $z_1,z_2,z_3,z_4\in \mathbb{C}$, defined by \[ (z_1,z_2;z_3,z_4)=\frac{(z_1-z_3)(z_2-z_4)}{(z_1-z_4)(z_2-z_3)}. \] 
\end{example}

The next result shows that together with the triple ratio, the double ratio provides an invariant for quadruples of flags in general position:

\begin{thm}\cite[\S9]{FG06} \label{FGquad}
	A quadruple of flags $(A,B,C,C')$ in general position is uniquely determined (upto the action of $\pslnc$) from $D_r(A,B,C,C')$ for all $1\le r < n$ and all the triple ratios of $(A,B,C)$ and $(A,B,C')$.
\end{thm}

\subsection{Framed representations and Fock-Goncharov coordinates}
Throughout this section, $S$ will denote a punctured surface of genus $g\geq 0$ and $k\geq 1$ punctures, having negative Euler characteristic. We shall discuss the case of a closed surface in Section 5.

\begin{defn}\label{FR}
	A \emph{framed representation} $\hat{\rho}$ of $\pi_1(S)$ into $\pslnc$ is a pair $(\rho, \beta)$ consisting of the representation $\rho: \pi_1(S) \rightarrow \pslnc $ together with a \textit{framing},  a $\rho$-equivariant map $\beta: F_{\infty} \rightarrow \mathcal{F}(\mathbb{C}^n)$, where the Farey set $F_\infty$ is the set of points on the ideal boundary of the universal cover, corresponding to the lifts of the punctures.
\end{defn}

\begin{defn} Given a punctured surface $S$ as above, the moduli space of framed $\pslnc$-representations is 
     \[\widehat{\chi}_n(S) := \{\text{framed representations } (\rho,\beta) \text{ of } \pi_1(S) \text{ into } \pslnc \}/\pslnc\] 
where the quotient is given by $(\rho, \beta) \sim (A\rho A^{-1}, A \cdot \beta)$ for $A \in \pslnc$.
\end{defn}

Fock-Goncharov provided coordinates on this moduli space, and proved:  
\begin{thm}\cite{FG06}\label{FGparam}
	There is a birational isomorphism \[ \Psi_\tau:\widehat{\chi}_n(S) \mapsto (\mathbb{C}^*)^{(2g+k-2)(n^2-1)} \]
 for any ideal triangulation $\tau$ of $S$.
\end{thm}

We shall now explain this result, and how their coordinates and the map $\Psi_\tau$ above are defined using the invariants of tuples of flags introduced in the previous subsection. The particular case for $\psltc$  has been discussed in \cite{FG07} and \cite{CTT}.  Note that these coordinates will be defined for an open dense subset of $\widehat{\chi}_n(S)$ as we shall assume that various triples and quadruples of flags are all ``generic", i.e.\ in general position  -- see \S2.3 for a discussion. 

\vspace{.05in}



In what follows, an $n$-triangulation is a sub-triangulation of a triangle $T$ as shown on \Cref{ftrian}, where each vertex corresponds to a triple of non-negative integers $(i,j,k)$ such that $i+j+k=n$. We can clearly notice that the downward facing small triangles (shaded in \Cref{ftrian}) in an $n$-triangulation are in a one-one correspondence with the vertices of a $(n-2)$-triangulation. 

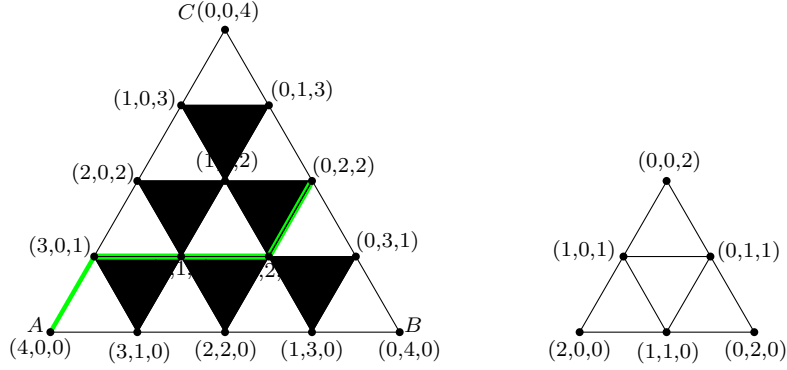
\begin{figure}[ht]
\begin{tikzpicture}
\clip(-0.57,-0.56) rectangle (9.85,4.5);
\fill[fill=black,fill opacity=0.25] (2.89,3) -- (1.73,3) -- (2.31,2) -- cycle;
\fill[fill=black,fill opacity=0.25] (2.31,2) -- (1.15,2) -- (1.73,1) -- cycle;
\fill[fill=black,fill opacity=0.25] (3.46,2) -- (2.31,2) -- (2.89,1) -- cycle;
\fill[fill=black,fill opacity=0.25] (1.73,1) -- (0.58,1) -- (1.15,0) -- cycle;
\fill[fill=black,fill opacity=0.25] (1.73,1) -- (2.89,1) -- (2.31,0) -- cycle;
\fill[fill=black,fill opacity=0.25] (2.89,1) -- (3.46,0) -- (4.04,1) -- cycle;
\draw[smooth,samples=100,domain=0.0:2.3094010767585034] plot(\x,{sqrt(3)*(\x)});
\draw[smooth,samples=100,domain=2.3094:4.6188] plot(\x,{0-sqrt(3)*((\x)-8/sqrt(3))});
\draw (1.73,3)-- (2.89,3);
\draw (3.46,2)-- (1.15,2);
\draw (0.58,1)-- (4.04,1);
\draw[smooth,samples=100,domain=1.73205:3.4641016151377544] plot(\x,{0-sqrt(3)*(\x)+6});
\draw[smooth,samples=100,domain=1.15:2.3094010767585034] plot(\x,{0-sqrt(3)*(\x)+4});
\draw[smooth,samples=100,domain=0.58:1.1547005383792517] plot(\x,{0-sqrt(3)*(\x)+2});
\draw (0,0)-- (4.62,0);
\draw[green, line width=0.6mm] (0,0)-- (0.58,1);
\draw[green, line width=0.8mm] (0.58,1)-- (2.89,1);
\draw[green, line width=0.8mm] (2.89,1)-- (3.46,2);
\draw (2.89,3)-- (1.15,0);
\draw (3.46,2)-- (2.31,0);
\draw (4.04,1)-- (3.46,0);
\draw (2.89,3)-- (1.73,3);
\draw (1.73,3)-- (2.31,2);
\draw (2.31,2)-- (2.89,3);
\draw (2.31,2)-- (1.15,2);
\draw (1.15,2)-- (1.73,1);
\draw (1.73,1)-- (2.31,2);
\draw (3.46,2)-- (2.31,2);
\draw (2.31,2)-- (2.89,1);
\draw (2.89,1)-- (3.46,2);
\draw (1.73,1)-- (0.58,1);
\draw (0.58,1)-- (1.15,0);
\draw (1.15,0)-- (1.73,1);
\draw (1.73,1)-- (2.89,1);
\draw (2.89,1)-- (2.31,0);
\draw (2.31,0)-- (1.73,1);
\draw[smooth,samples=100,domain=7.0:8.1547] plot(\x,{sqrt(3)*((\x)-7)});
\draw[smooth,samples=100,domain=8.1547:9.3094] plot(\x,{0-sqrt(3)*((\x)-1)+4+6*sqrt(3)});
\draw (2.89,1)-- (3.46,0);
\draw (3.46,0)-- (4.04,1);
\draw (4.04,1)-- (2.89,1);
\draw (7,0)-- (9.31,0);
\draw (7.58,1)-- (8.73,1);
\draw (7.58,1)-- (8.15,0);
\draw (8.73,1)-- (8.15,0);
\begin{scriptsize}
\fill [color=black] (0,0) circle (1.5pt);
\draw[color=black] (-0.13,-0.22) node {(4,0,0)};
\draw[color=black] (-0.2,0.1) node {$A$};
\fill [color=black] (2.31,4) circle (1.5pt);
\draw[color=black] (2.35,4.24) node {(0,0,4)};
\draw[color=black] (1.8,4.24) node {$C$};
\fill [color=black] (4.62,0) circle (1.5pt);
\draw[color=black] (4.75,-0.24) node {(0,4,0)};
\draw[color=black] (4.8,0.1) node {$B$};
\fill [color=black] (1.73,3) circle (1.5pt);
\draw[color=black] (1.25,3.07) node {(1,0,3)};
\fill [color=black] (2.89,3) circle (1.5pt);
\draw[color=black] (3.33,3.19) node {(0,1,3)};
\fill [color=black] (1.15,2) circle (1.5pt);
\draw[color=black] (0.7,2.09) node {(2,0,2)};
\fill [color=black] (3.46,2) circle (1.5pt);
\draw[color=black] (3.88,2.19) node {(0,2,2)};
\fill [color=black] (0.58,1) circle (1.5pt);
\draw[color=black] (0.13,1.13) node {(3,0,1)};
\fill [color=black] (4.04,1) circle (1.5pt);
\draw[color=black] (4.46,1.19) node {(0,3,1)};
\fill [color=black] (1.15,0) circle (1.5pt);
\draw[color=black] (1.18,-0.27) node {(3,1,0)};
\fill [color=black] (2.31,0) circle (1.5pt);
\draw[color=black] (2.32,-0.25) node {(2,2,0)};
\fill [color=black] (3.46,0) circle (1.5pt);
\draw[color=black] (3.46,-0.25) node {(1,3,0)};
\fill [color=black] (1.73,1) circle (1.5pt);
\draw[color=black] (1.73,0.82) node {(2,1,1)};
\fill [color=black] (2.31,2) circle (1.5pt);
\draw[color=black] (2.34,2.26) node {(1,1,2)};
\fill [color=black] (2.89,1) circle (1.5pt);
\draw[color=black] (2.91,0.8) node {(1,2,1)};
\fill [color=black] (7.58,1) circle (1.5pt);
\draw[color=black] (7.06,1.08) node {(1,0,1)};
\fill [color=black] (8.73,1) circle (1.5pt);
\draw[color=black] (9.24,1.06) node {(0,1,1)};
\fill [color=black] (7,0) circle (1.5pt);
\draw[color=black] (7,-0.25) node {(2,0,0)};
\fill [color=black] (9.31,0) circle (1.5pt);
\draw[color=black] (9.33,-0.25) node {(0,2,0)};
\fill [color=black] (8.15,0) circle (1.5pt);
\draw[color=black] (8.17,-0.27) node {(1,1,0)};
\fill [color=black] (8.15,2) circle (1.5pt);
\draw[color=black] (8.19,2.26) node {(0,0,2)};
\end{scriptsize}
\end{tikzpicture}
	\caption{Snake (green) in a 4-triangulation (left), one-one correspondence of the downward facing small (shaded) triangles with 2-triangulation (right).}
	\label{ftrian}
\end{figure}

To define $\Psi_\tau((\rho,\beta)),\ \frac{(n-2)(n-1)}{2}$ triangle invariants and $(n-1)$ edge invariants are associated with each of the ideal triangles and edges, respectively, of a fixed ideal triangulation $\tau$. In the following two paragraphs, we shall discuss how these invariants are associated.

To assign the triangle invariants, consider the $(n-1)$-triangulation in each ideal triangle of $\tau$.  An invariant is associated for each of the downward-facing small triangles inside the $(n-1)$-triangulation; these are the \textit{triangle invariants}, defined as follows. Let $A,B,C$ be a generic triple of flags in $\mathcal{F}(\mathbb{C}^n)$ associated counterclockwise to the ideal vertices of a triangle $T\in \tau$ via the framing $\beta$. Then the triangle invariant $X_{i,j,k}$ (where $i+j+k=n-1-2=n-3$) for $T$ is defined to be the $(i+1,j+1,k+1)$-triple ratio of the three flags $(A,B,C)$. 

To assign the edge invariants, let $(A,B,C,C')$ be a quadruple of flags associated (via the framing) to the vertices of the ideal quadrilateral formed by two adjacent ideal triangles of $\tau$, in counterclockwise order,  where $A$ and $C$ are the flags associated with the (endpoints of the) common edge $E$.  The $(n-1)$ edge-invariants for $E$ are defined as the $r$-th double ratio of $(A,B,C,C^\prime)$ for $1\leq r\leq n-1$.  Notice that $E$ consists of $(n-1)$ edges of the $(n-1)$-triangulation; the edge-invariants can be thought of as corresponding to these small edges.  

\begin{rmk}
	A quick dimension count yields the following: on any ideal triangulation on $S$, there are $2(2g+k-2)$ triangles and $3(2g+k-2)$ edges. Each triangle has $\frac{(n-2)(n-1)}{2}$ triangle invariants and each edge has $(n-1)$ edge invariants associated. Thus, the total number of (non-zero complex) parameters is $(2g+k-2)(n-1)(n+1)$ which explains the dimension of the target space in Theorem \ref{FGparam}.
\end{rmk}

\begin{example}
	For $n=2$, these define Thurston's shear-bend coordinates for the $\pslc$-representa-tion variety. Recall that flags in $\mathbb{C}^2$ can be identified with the points in $\mathbb{CP}^1$; thus, a framing is a  $\rho$-equivariant map $\beta: F_{\infty} \rightarrow \mathbb{CP}^1$. Since, triple of points on $\mathbb{CP}^1$ are unique (upto the action of $\pslc$), there are no triangle invariants in this case.  There is one edge invariant  associated with each edge of the triangulation which is the complex cross-ratio of the four ideal points that are the vertices of the two adjacent ideal triangles. 
\end{example}

We shall now give an idea of how to define $\Psi_\tau^{-1}$ in \Cref{FGparam}, i.e. how to define a framed representation $(\rho,\beta)$ from given coordinates. 

We shall first determine the framing $\beta$ as follows. Consider the lift of the ideal triangulation in the universal cover of $S$ and choose an ideal triangle $\Delta_0$. From the (given) triangle invariants of that triangle, we obtain a triple of flags associated with the three ideal vertices of $\Delta_0$ from \Cref{FGmain}, which are well defined up to post-composition by an element of $\pslnc$. Now consider an adjacent triangle $\Delta_1$ of $\Delta_0$ with a common edge $e_0$. Using \Cref{FGquad}, the triangle invariants of $\Delta_0$ and $\Delta_1$ and the edge invariants of $e_0$ determine the flag at the other vertex of $\Delta_1$. Continuing this process, we can assign flags uniquely to each vertex of the ideal triangulation, and thus obtain the map $\beta$.

\vspace{.05in} 

To obtain the representation $\rho$, we introduce what we shall call the \textit{monodromy graph} on $S$. For the fixed ideal triangulation $\tau$, the monodromy graph $\Gamma_\tau$ is the embedded graph on $S$ comprising an edge transverse to each edge of $\tau $ and a triple of edges inside each triangle that connect the endpoints of the above edges to form a smaller triangle (see \Cref{monod}). In other words, we start with the dual graph of $\tau$ and then blow up each vertex to a small triangle completely inside each ideal triangle to get $\Gamma_{\tau}$.

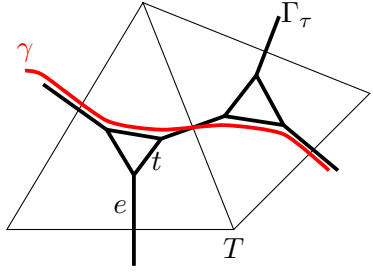
\begin{figure}[ht]
\centering
\begin{tikzpicture}[scale=0.6]
\draw (3,5) -- (0,0) -- (5,0) -- (3,5) -- (8,3) -- (5,0);
\draw [line width=0.5mm,   black] (3.4,2) -- (2.2,2.2) -- (2.8,1.2) -- (3.4,2) -- (4.8,2.5) --(6.1,2.3) -- (5.5,3.4) -- (4.8,2.5);
\draw [line width=0.5mm,   black] (0.8,3.2) -- (2.2,2.2);
\draw [line width=0.5mm,   black] (2.8,-0.8) -- (2.8,1.2);
\draw [line width=0.5mm,   black] (6.1,2.3) -- (7.3,1.3);
\draw [line width=0.5mm,   black] (5.5,3.4) -- (6,4.7);
\draw [line width=0.5mm,  red] plot [smooth, tension=0.4] coordinates { (0.4,3.5) (0.8,3.4) (2.2,2.4) (3.4,2.2) (4.8,2.3) (6.1,2.1) (7.1, 1.3)};
\node [red] at (0.4,3.9) {$\gamma$};
\node at (5,-0.4) {$T$};
\node at (6.4,4.74) {$\Gamma_\tau$};
\node at (3.3,1.5) {$t$};
\node at (2.5,0.5) {$e$};
\end{tikzpicture}
\setlength{\belowcaptionskip}{-8pt}
\caption{Part of the monodromy graph (shown in bold).}
\label{monod}
\end{figure}
\noindent The edges transverse to the edges of $\tau$ are called $e$-edges and the edges of the small triangles are called $t$-edges. With every oriented edge of $\Gamma_\tau$, an element of $\pslnc$ shall be assigned, which we denote by $E(e)$ and $T(t)$, where $e$ and $t$ represent the tuples of edge invariants and triangle invariants respectively. Briefly, the matrix $T(t)$ is the transformation that maps the flags associated with the vertices of an edge of the ideal triangle to those of the next edge (here the edges of an ideal triangle are considered to be ordered counterclockwise). The matrix $E(e)$ swaps the flags associated with the edge of $\tau$  that is transverse to $e$, and maps the line (i.e., the one-dimensional subspace of the flag) associated with the opposite vertex of an adjacent triangle (to $e$) onto the line associated with the opposite vertex of the other adjacent triangle. A given curve $\gamma$ on $S$ can be homotoped in a unique way to a concatenation of (directed) edges of $\Gamma_\tau$ such that the $e$-edges and the $t$-edges appear consecutively. To compute the monodromy of $\gamma$,  namely $\rho(\gamma)$, let us consider \Cref{monComp}.
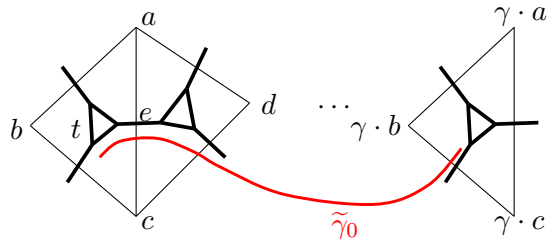
\begin{figure}[tph]
	\centering
	\begin{tikzpicture}
	\draw (-0.4,1.2)-- (1,2.5);
	\draw (-0.4,1.2)-- (1,0);
	\draw (1,2.5)-- (2.5,1.5);
	\draw (2.5,1.5)-- (1,0);
	\draw (4.6,1.2)-- (6,2.5);
	\draw (6,2.5)-- (6,0);
	\draw (4.6,1.2)-- (6,0);
	\draw (1,2.5)-- (1,0);
	\draw (0.93,2.83) node[anchor=north west] {$a$};
	\draw (-.8,1.43) node[anchor=north west] {$b$};
	\draw (0.93,0.14) node[anchor=north west] {$c$};
	\draw (5.6,2.91) node[anchor=north west] {$\gamma \cdot a$};
	\draw (3.7,1.48) node[anchor=north west] {$\gamma \cdot b$};
	\draw (5.6,0.14) node[anchor=north west] {$\gamma \cdot c$};
	\draw (2.52,1.75) node[anchor=north west] {$d$};
	\draw (3.26,1.66) node[anchor=north west] {$\cdots$};
	\draw [line width=1.4pt] (0.02,1.98)-- (0.39,1.49);
	\draw [line width=1.4pt] (0.39,1.49)-- (0.42,0.95);
	\draw [line width=1.4pt] (0.42,0.95)-- (0.75,1.22);
	\draw [line width=1.4pt] (0.75,1.22)-- (0.39,1.49);
	\draw [line width=1.4pt] (0.42,0.95)-- (0.09,0.45);
	\draw [line width=1.4pt] (0.75,1.22)-- (1.32,1.24);
	\draw [line width=1.4pt] (1.32,1.24)-- (1.67,1.69);
	\draw [line width=1.4pt] (1.67,1.69)-- (1.77,1.16);
	\draw [line width=1.4pt] (1.77,1.16)-- (1.32,1.24);
	\draw [line width=1.4pt] (1.77,1.16)-- (2.18,0.78);
	\draw [line width=1.4pt] (5.02,1.98)-- (5.39,1.49);
	\draw [line width=1.4pt] (5.75,1.22)-- (5.39,1.49);
	\draw [line width=1.4pt] (5.39,1.49)-- (5.42,0.95);
	\draw [line width=1.4pt] (5.42,0.95)-- (5.75,1.22);
	\draw [line width=1.4pt] (5.42,0.95)-- (5.09,0.45);
	\draw [line width=1.4pt] (5.75,1.22)-- (6.32,1.24);
	\draw [line width=1.4pt] (1.67,1.69)-- (1.93,2.23);
	\draw (0,1.41) node[anchor=north west] {$t$};
	\draw (.9,1.56) node[anchor=north west] {$e$};
	\draw [line width=1.1pt,  red] plot [smooth, tension=.6] coordinates { (0.52, 0.78) (0.78, 0.99) (1.29, 1.03) (1.77, 0.85) (2.12, 0.65) (2.96, 0.29) (4.17, 0.16) (4.86, 0.41) (5.3, 0.9)};
	\draw [color=red](3.45,0.2) node[anchor=north west] {$\widetilde\gamma_0$};
	\end{tikzpicture}
	\caption[Computing monodromy of a curve]{Computing monodromy of a curve $\gamma$.}
	\label{monComp}
\end{figure}
Let $\widetilde{\gamma_0}$ be an arc on $\widetilde{\gamma}$ that descends onto $\gamma$. Let $\widetilde{\gamma_0}$ start from a triangle with ideal vertices $a,b,c$. Then the endpoint of $\widetilde{\gamma_0}$ lies within the triangle with ideal vertices $\gamma \cdot a, \gamma \cdot b, \gamma \cdot a$. Let $A,B,C$ be the flags associated with the vertices $a,b,c$. We want $\rho(\gamma)$ such that $(A,B,C)=(\beta(a),\beta(b),\beta(c) ) \xmapsto{\rho(\gamma)} (\beta(\gamma \cdot a), \beta(\gamma \cdot b),\beta(\gamma \cdot c) )$, as $\beta$ must be $\rho$-equivariant. Let $\widetilde{\gamma_0} \sim e_1 * t_1*\cdots * e_k * t_k$. Then we can see that the matrix $T(t_k)^{\delta_k}E(e_k)\cdots T(t_1)^{\delta_1}E(e_1)$ maps the flags $A,B,C$ onto the flags $\beta(\gamma \cdot a), \beta(\gamma \cdot b),\beta(\gamma \cdot c)$ respectively, where the signs $\delta_i\in \{\pm 1 \}$ are determined by the order in which the curve crosses the triangles. Hence we obtain 
\begin{equation}\label{rhogamma}
\rho(\gamma)=T(t_k)^{\delta_k}E(e_k)\cdots T(t_1)^{\delta_1}E(e_1).
\end{equation}
where note that the constituent matrices depend on the corresponding triangle and edge invariants respectively. 

In \Cref{mmat} we shall refine the above matrix product into a product of simpler matrices.

\subsection{The genericity condition}
The condition that the representation $\rho$ is ``generic" in the statement of the main result \Cref{mainthm}  comes from the fact that we require that there exists a framing $\beta$ such that the framed representation $\hat{\rho} = (\rho,\beta)$ has well-defined Fock-Goncharov coordinates with respect to some ideal triangulation $\tau$.

The following lemma justifies calling this a ``generic" condition:

\begin{lem}\label{gencon}

    The set of representations $\rho$ in the $\pslnc$-representation variety (for a punctured surface-group) that have a framing such that the resulting framed representation has well-defined Fock-Goncharov coordinates, is open and dense.
	
\end{lem}
\begin{proof}
        It suffices to show that the subset of framed representations in the moduli space of \textit{framed} representations with well-defined Fock-Goncharov coordinates is open and dense, since the projection map to the $\pslnc$-representation variety (that forgets the framing) is continuous and an open map.
        
        Let $\tilde{\tau}$ be the lift of the ideal triangulation to the universal cover. 
	Now the following two conditions are sufficient for $(\rho,\beta)$ to have the Fock-Goncharov coordinates: 
	\begin{itemize}
		\item The triple of flags assigned by $\beta$ to the ideal vertices of the lift of any triangle in $\tilde{\tau}$ are in general position.
  
		\item The quadruples of flags associated with the ideal vertices of the ideal quadrilateral formed by two adjacent triangles of $\tilde{\tau}$ are in general position.
	\end{itemize}

    By the $\rho$-equivariance of $\beta$, the above reduces to a \textit{finite} check, involving finitely many ideal triangles that form the fundamental domain of the $\pi_1(S)$-action, together with the adjacent fundamental domains. 
    
    Note that flags \textit{not} being in a general position involves certains coincidences of the lower-dimensional subspaces of those flags; thus, the space of triples (or quadruples) of such flags has positive codimension in the space of \textit{all} triples (or quadruples). It follows that the set of framed representations which has the Fock-Goncharov coordinates with respect to an ideal triangulation $\tau$ is  already open and dense in the moduli space of framed representations.
\end{proof}

\begin{rmk}
   It is also true that the open dense subset of the $\pslnc$-representation variety is \textit{not} the whole space -- for example, a representation whose image is an order-$2$ group in $\pslnc$ will be in the complement that subset.  For $n=2$, there is a characterization of this open and dense set in terms of the notion of \textit{non-degenerate} framed representations (see \cite[Definition 4.3]{AB}, and also \cite[Corollary 4.6]{GG}). 
\end{rmk}

\subsection{Hitchin and positive representations}\label{HitchinR}

For a closed surface $S$ of genus $g\geq 2$, Hitchin representations were identified in \cite{Hitchin} as the connected component of the $\pslnr$-representation variety \[ \mathcal{X}_n(S,\mathbb{R})\coloneqq Hom(\pi_1(S),\pslnr)/\pslnr \]
that contains the \emph{$n$-Fuchsian} representations, where $\rho=\iota_n \circ \rho_2$, where $\rho_2:\pi_1(S) \rightarrow \pslr$ is a Fuchsian representation and $\iota_n$ is the unique irreducible representation $\iota_n:\pslr\rightarrow \pslnr$. 

\vspace{.05in}

    Hitchin showed that like Teichm\"{u}ller space $\mathcal{T}(S)$, this component is diffeomorphic to a ball, of $\mathbb{R}^{(6g-6)(n^2-1)}$. Moreover, in \cite{Labourie}, Labourie showed that the Hitchin representations are discrete and faithful. In fact, he established a dynamical property called Anosov that shows the existence of a $\rho$-equivariant continuous map $\beta: \partial_{\infty}\widetilde{S} \rightarrow \mathcal{F}(\mathbb{C}^n)$. This shall be crucial in Section 5 to define a framing when we handle the case of a closed surface. 

\vspace{.05in}

More importantly for us, the work of Fock-Goncharov \cite{FG06} introduced the notion of \textit{positive (framed) representations}, namely those that have real and positive Fock-Goncharov coordinates, as described in the previous subsection. 

\vspace{.05in}

\noindent For the case when $S$ is a punctured surface of negative Euler characteristic, we shall define:

\begin{defn}[Positive representation]\label{puncHit}
A representation $\rho : \pi_1(S) \rightarrow \pslnr$ is a \emph{Hitchin representation} if there exists a framing $\beta$ such that the framed representation $\hat{\rho} = (\rho,\beta)$ has well-defined Fock-Goncharov coordinates with respect to some ideal triangulation $\tau$, and the coordinates are all real and positive. 
\end{defn}

\begin{rmk}
    In the punctured surface case, the set of Hitchin representations forms an open subset of the $\pslnr$-representation variety but not a component. Indeed,  since the surface-group is a free group, the $\pslnr$-representation variety is in this case connected. 
\end{rmk}

\noindent The notion of a Hitchin representation as a positive representation was introduced by Labourie-McShane in \cite{LabMc} for the case when the boundary holonomies are purely loxodromic, and was generalized by Canary-Zhang-Zimmer in \cite{CZZ} to also allow for boundary components with unipotent holonomy.   

\vspace{.05in}

\noindent We also introduce the following notion:

  \begin{defn}[Bending fiber]\label{defn:bf}
     Let   ${\rho}_1, {\rho}_2:\pi_1(S) \to \pslnc$ be two representations that can be equipped with framings such that the resulting framed representations have well-defined Fock-Goncharov coordinates with respect to the same ideal triangulation $\tau$. Then they are said to be \emph{in the same bending fiber} if the absolute values of each of these Fock-Goncharov coordinates are equal, in which case they are said to be obtained by a \textit{bending deformation} from each other. 
  \end{defn}

\subsection{Monodromy via Fock-Goncharov coordinates}\label{mmat} In the rest of this subsection we shall describe a further decomposition of the $T$ and $E$ matrices in \eqref{rhogamma} as a finer product of simpler matrices due to Fock-Goncharov (see \cite[Chapter 9]{FG06}) which they had needed for proving some positivity results, and we shall need in Section 3.

To obtain these matrices, Fock-Goncharov introduced the notion of a \textit{snake} and their moves -- in this subsection we aim to give a self-contained account of this, as at places our conventions (and therefore computations) will differ from those of Fock-Goncharov (compare, for instance, our \Cref{snkmv} with \cite[Fig 9.14]{FG06}).

\begin{defn}[Snake] A \emph{snake} is an oriented path on the edges of an $(n-1)$-triangulation starting from a vertex of the main triangle up to its opposite edge consisting of exactly $n-1$ edges of the small triangles.  (See for example \Cref{ftrian}, where the snake is highlighted in green.) 
\end{defn}

A snake provides a \textit{projective basis} of $\mathbb{C}^n$, namely a basis of $\mathbb{C}^n$ up to multiplication by a common factor, as follows:

First, we associate the one-dimensional subspace $V^{a,b,c}=A_{n-a}\cap B_{n-b}\cap C_{n-c}$ (where $a+b+c=n-1$) with each vertex of the $(n-1)$-triangulation. Choose a non-zero vector $v_1$ from the line associated with the initial vertex of the snake (also called the \textit{snake head}). Then there is a unique vector $v_2$ in the line associated with the counterclockwise next vertex of the adjacent unshaded small triangle such that $v_1+v_2$ belongs to the line associated to the remaining vertex of that unshaded triangle (see \Cref{prjBasis}).

\begin{figure}[ht]
\centering
\begin{tikzpicture}
\draw (0,0)-- (0.58,1)-- (1.15,0)-- (0,0);
\draw (-0.2,0) node {$v_1$};
\draw (0.58,1.2) node {$v_2$};
\draw (1.8,0) node {$v_1+v_2$};
\end{tikzpicture}
\setlength{\belowcaptionskip}{-8pt}
\caption{Assigning vectors to a small triangle.}
\label{prjBasis}
\end{figure}
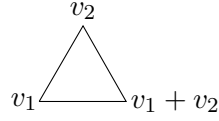

Passing to the next adjacent small triangle, we can iteratively obtain a vector for each vertex on the snake, which forms a basis of $\mathbb{C}^n$. For a different choice of the initial vector $v_1$ on the line, all elements of this basis are multiplied by a common factor. Thus, we in fact have a \textit{projective} basis uniquely determined by the snake. 

The $T(t)$ and $E(e)$ matrices can now be thought of as the change of basis matrix to change from one projective basis to another, when the snakes correspond to two (oriented) sides of an ideal triangle. We shall first illustrate this for the matrix $T(t)$, in the case where $n=3$, in the following example. 

\vspace{.05in}

\begin{example}\label{n3M}
In this example, $n=3$ and we shall refer to \Cref{flgcng} and \Cref{snkcng}. Let $A, B, C$ be the flags associated with the vertices of an ideal triangle. With a slight abuse of notation we shall call the vertices of the triangle by the flags associated with it. Let us fix a vector $a$ from the one dimensional subspace of $A$. We immediately obtain the vectors at the other vertices such that
\begin{equation}\label{vReln}
    a+p=q,\quad p+c=r_1,\quad q+r_2=b.
\end{equation}
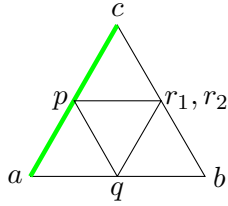
\begin{figure}[ht]
	\begin{tikzpicture}
	\draw (0,0)-- (1.15,2)-- (2.31,0)-- (0,0);
	\draw[green, line width=0.6mm] (0,0)-- (1.15,2);
	\draw (1.15,0)-- (0.58,1)-- (1.73,1)-- (1.15,0);
	\draw (-0.2,0) node {$a$} (0.4,1) node {$p$} (1.15,2.2) node {$c$} (2.5,0) node {$b$} (1.15,-0.2) node {$q$} (2.2,1) node {$r_1,r_2$};
	\end{tikzpicture}
	\setlength{\belowcaptionskip}{-8pt}
	\caption{Basis transformation for $n=3$.}
	\label{flgcng}
\end{figure}
Without loss of generality, let us assume $A=(v_1,l_1), B=(v_2,l_2)$ and $ C=(v_3,l_3)$, where \[ v_1=\begin{pmatrix}
    0\\0\\1
\end{pmatrix}, v_2=\begin{pmatrix}
    1\\0\\1
\end{pmatrix}, v_3=\begin{pmatrix}
    0\\1\\1
\end{pmatrix} \] and $l_1,l_2,l_3$ are the lines such that \[ l_1(x,y,z)=Xx+y=0,~ l_2(x,y,z)=x-z=0 \text{ and } l_3(x,y,z)=y-z=0 \] where $X \in \mathbb{C}$ coincides with the triple ratio of $A,B,C$ (\textit{c.f.} \Cref{defn:tr}) since \[ T_{111}(A,B,C)=\frac{l_1 (v_2) l_2 (v_3) l_3 (v_1)}{l_1 (v_3) l_2 (v_1) l_3 (v_2)}=\frac{X\cdot (-1)\cdot (-1)}{1\cdot (-1)\cdot (-1)}=X. \] Let us choose $a=v_1$. We already have $\langle b\rangle = \langle v_2 \rangle,~ \langle c \rangle= \langle v_3 \rangle$. From \Cref{flgcng}, we have \[ \langle p \rangle = l_1\cap l_3= \bigg\langle \begin{pmatrix}
    1\\-X\\-X
\end{pmatrix} \bigg\rangle, \langle r_1 \rangle=\langle r_2 \rangle=l_2\cap l_3= \bigg\langle \begin{pmatrix}
    1\\1\\1
\end{pmatrix} \bigg\rangle \text{ and } \langle q \rangle= l_1\cap l_2 = \bigg\langle\begin{pmatrix}
    1\\-X\\1
\end{pmatrix} \bigg\rangle. \] We see that, the vectors \[ p=\begin{pmatrix}
    \frac{1}{1+X}\\ -\frac{X}{1+X}\\ -\frac{X}{1+X}
\end{pmatrix}, q=\begin{pmatrix}
    \frac{1}{1+X}\\ -\frac{X}{1+X}\\ \frac{1}{1+X}
\end{pmatrix}, c= \begin{pmatrix}
    0\\1\\1
\end{pmatrix}, r_1=\begin{pmatrix}
    \frac{1}{1+X}\\ \frac{1}{1+X}\\ \frac{1}{1+X}
\end{pmatrix}, r_2=\begin{pmatrix}
    \frac{X}{1+X}\\ \frac{X}{1+X}\\ \frac{X}{1+X}
\end{pmatrix}, b=\begin{pmatrix}
    1\\0\\1
\end{pmatrix} \] satisfy \cref{vReln}. Hence we get, $r_2=Xr_1$, where $X$ is the triple ratio of $A,B$ and $C$.

Now $\{a,p,c\}, \{a,p,r_1\},\{a,q,r_2\}$ and $\{a,q,b\}$ are the ordered bases corresponding to the snakes in \Cref{snkcng} respectively. 
\noindent Let $T_1, T_2, T_3\in \psltc$ be the change of basis matrices from the snake $AC$ to $AB$ as in \Cref{snkcng}.

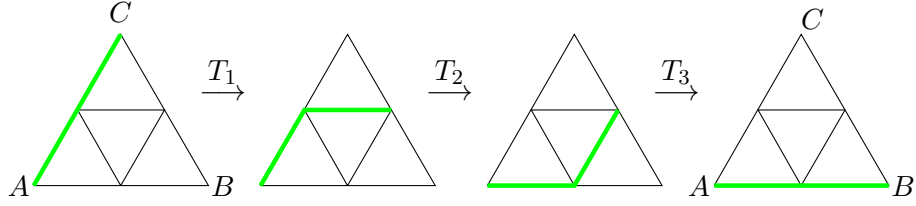
\begin{figure}[ht]
    \begin{tikzpicture}
    \draw (0,0)-- (1.15,2)-- (2.31,0)-- (0,0);
    \draw (1.15,0)-- (0.58,1)-- (1.73,1)-- (1.15,0);
    \draw[green, line width=0.6mm] (0,0)-- (1.15,2);
    \draw (2.5,1.2) node {$\longrightarrow$} (2.5,1.5) node {$T_1$};
    \draw (-0.2,0) node {$A$} (1.15,2.3) node {$C$} (2.5,0) node {$B$};
    \draw (3,0)-- (4.15,2)-- (5.31,0)-- (3,0);
    \draw (4.15,0)-- (3.58,1)-- (4.73,1)-- (4.15,0);
    \draw[green, line width=0.6mm] (3,0)-- (3.58,1)-- (4.73,1);
    \draw (5.5,1.2) node {$\longrightarrow$} (5.5,1.5) node {$T_2$};
    \draw (6,0)-- (7.15,2)-- (8.31,0)-- (6,0);
    \draw (7.15,0)-- (6.58,1)-- (7.73,1)-- (7.15,0);
    \draw[green, line width=0.6mm] (6,0)-- (7.15,0)-- (7.73,1);
    \draw (8.5,1.2) node {$\longrightarrow$} (8.5,1.5) node {$T_3$};
    \draw (9,0)-- (10.15,2)-- (11.31,0)-- (9,0);
    \draw (10.15,0)-- (9.58,1)-- (10.73,1)-- (10.15,0);
    \draw[green, line width=0.6mm] (9,0)-- (11.31,0);
    \draw (8.8,0) node {$A$} (10.3,2.2) node {$C$} (11.5,0) node {$B$};
    \end{tikzpicture}
    \setlength{\belowcaptionskip}{-8pt}
    \caption{Snake move for $n=3$.}
    \label{snkcng}
\end{figure}
 \noindent We observe that 
 \[ \begin{pmatrix}
1 & 0 & 0\\
0 & 1&0\\
0&1&1     
 \end{pmatrix}
\begin{pmatrix}
    a\\p\\c
\end{pmatrix}=\begin{pmatrix}
    a\\p\\r_1
\end{pmatrix}.
 \] Hence,
\[ T_1= \begin{pmatrix}
1 & 0 & 0\\
0 & 1&0\\
0&1&1
\end{pmatrix}.\]
Also, \[ \begin{pmatrix}
1&0&0\\
0&1&0\\
0&0&X
\end{pmatrix}
\begin{pmatrix}
1&0&0\\
1&1&0\\
0&0&1
\end{pmatrix}
\begin{pmatrix}
    a\\p\\r_1
\end{pmatrix}=\begin{pmatrix}
    a\\q\\r_2
\end{pmatrix} \text{ and } \begin{pmatrix}
1 & 0 & 0\\
0 & 1&0\\
0&1&1
\end{pmatrix}\begin{pmatrix}
    a\\q\\r_2
\end{pmatrix}=\begin{pmatrix}
    a\\q\\b
\end{pmatrix}
\]
which implies that
\[
T_2=\begin{pmatrix}
1&0&0\\
0&1&0\\
0&0&X
\end{pmatrix}
\begin{pmatrix}
1&0&0\\
1&1&0\\
0&0&1
\end{pmatrix},\ 
T_3=\begin{pmatrix}
1 & 0 & 0\\
0 & 1&0\\
0&1&1
\end{pmatrix}. \]
So we compute that the change of basis matrix from snake $AC$ to $AB$ is 
\begin{equation*}
\begin{split}
M_{AC\rightarrow AB}=T_3T_2T_1 &= \begin{pmatrix}
1 & 0 & 0\\
0 & 1&0\\
0&1&1
\end{pmatrix}
\begin{pmatrix}
1&0&0\\
0&1&0\\
0&0&X
\end{pmatrix}
\begin{pmatrix}
1&0&0\\
1&1&0\\
0&0&1
\end{pmatrix}
\begin{pmatrix}
1 & 0 & 0\\
0 & 1&0\\
0&1&1
\end{pmatrix} \\
&=\begin{pmatrix}
1&0&0\\
1&1&0\\
1&1+X&X
\end{pmatrix}.
\end{split}
\end{equation*}
To determine a projective basis corresponding to the snake $CA$, let us fix the vector $c$ at $C$. From \cref{vReln}, we have $c+(-r_1)=-p \text{ and } (-p)+q=a$. So $\{c,-p,a\}$ is a projective basis corresponding to the snake $CA$. So we have \[ M_{CA\rightarrow AC}=\begin{pmatrix}
0&0&1\\
0&-1&0\\
1&0&0
\end{pmatrix}. \]
Hence we finally obtain 
\begin{equation}\label{Tt3}
\begin{split}
T(t)=M_{CA\rightarrow AB}=M_{AC\rightarrow AB}\cdot M_{CA \rightarrow AC} &=\begin{pmatrix}
1&0&0\\
1&1&0\\
1&1+X&X
\end{pmatrix}\begin{pmatrix}
0&0&1\\
0&-1&0\\
1&0&0
\end{pmatrix} \\
&=\begin{pmatrix}
0&0&1\\
0&-1&1\\
X&-1-X&1
\end{pmatrix}.
\end{split}
\end{equation}
This concludes the example. 
\end{example}
\begin{rmk}
    As the above example demonstrates, we follow the convention in \cite{FG06}, namely, we write the ordered basis corresponding to a snake as a column vector and left-multiply it with the change-of-basis matrix. In a different convention, if we write the ordered basis as a row vector and perform a right-multiplication with the change-of-basis matrix, we would obtain matrices that differ by a transpose. In that case, the order of the matrix product in \cref{rhogamma} would be reversed.
\end{rmk}
\vspace{.1in} 
\noindent We shall now generalize the above example to obtain an product expansion (for general $n$) of 
\begin{equation}\label{Tt}
    T(t)=M(t)_{AC\rightarrow AB}\cdot M(t)_{CA \rightarrow AC} = M(t) \cdot S
\end{equation}
where for the second equality $M(t)_{CA \rightarrow AC} =S $, a constant matrix defined by 
\begin{equation}\label{S}
S=\begin{pmatrix}
0 & \cdots & \cdots & 0 & 1 \\
\vdots & & \reflectbox{$\ddots$} & -1 & 0\\
\vdots & \reflectbox{$\ddots$} & \reflectbox{$\ddots$} & \reflectbox{$\ddots$} & \vdots \\
0 & \reflectbox{$\ddots$} & \reflectbox{$\ddots$} & & \vdots \\
(-1)^{n+1} & 0 & \cdots & \cdots & 0
\end{pmatrix}_{n\times n}.
\end{equation} 
and $M(t)_{AC\rightarrow AB}$ is abbreviated to $M(t)$. 

\vspace{.05in} 

\noindent Let us briefly discuss the iterative process to derive the expansion for $M(t)$. The initial snake $AC$ moves to the final snake $AB$ through the $(n-1)$ steps as shown in \Cref{snkmv}, where in each step, the intermediate snake moves happen in the steps shown in \Cref{intmv}.

\begin{figure}[ht] 
   \begin{tikzpicture}[scale=0.5]
   \draw (0,0)-- (2.31,4)-- (4.62,0)-- (0,0);
   \draw (1.15,0)--(2.89,3) --(1.73,3) --(3.46,0) --(4.04,1)--(0.58,1) --(1.15,0);
   \draw (1.15,2)-- (3.46,2)-- (2.31,0)-- (1.15,2);
   \draw[green, line width=0.4mm] (0,0)-- (2.31,4);
   \draw (5,2.5) node {$\longrightarrow$} (5,3) node {$T_1$};
   \draw (6,0)-- (8.31,4)-- (10.62,0)-- (6,0);
   \draw (7.15,0)--(8.89,3) --(7.73,3) --(9.46,0) --(10.04,1)--(6.58,1) --(7.15,0);
   \draw (7.15,2)-- (9.46,2)-- (8.31,0)-- (7.15,2);
   \draw[green, line width=0.4mm] (6,0)-- (7.73,3)-- (8.89,3);
   \draw (11,2.5) node {$\longrightarrow$} (11,3) node {$T_2$};
   \draw (12,0)-- (14.31,4)-- (16.62,0)-- (12,0);
   \draw (13.15,0)--(14.89,3) --(13.73,3) --(15.46,0) --(16.04,1)--(12.58,1) --(13.15,0);
   \draw (13.15,2)-- (15.46,2)-- (14.31,0)-- (13.15,2);
   \draw[green, line width=0.4mm] (12,0)-- (13.15,2)-- (15.46,2);
   \draw[dashed,->] (16.2,2.5) -- (17.7,2.5);
   \draw (18,0)-- (20.31,4)-- (22.62,0)-- (18,0);
   \draw (19.15,0)--(20.89,3) --(19.73,3) --(21.46,0) --(22.04,1)--(18.58,1) --(19.15,0);
   \draw (19.15,2)-- (21.46,2)-- (20.31,0)-- (19.15,2);
   \draw[green, line width=0.4mm] (18,0)-- (18.58,1)-- (22.04,1);
   \draw (23,2.5) node {$\longrightarrow$} (23,3.2) node {$T_{n-1}$};
   \draw (24,0)-- (26.31,4)-- (28.62,0)-- (24,0);
   \draw (25.15,0)--(26.89,3) --(25.73,3) --(27.46,0) --(28.04,1)--(24.58,1) --(25.15,0);
   \draw (25.15,2)-- (27.46,2)-- (26.31,0)-- (25.15,2);
   \draw[green, line width=0.4mm] (24,0)-- (28.62,0);
   \end{tikzpicture}
   \caption{Snake moves in $n$ steps.}
   \label{snkmv}	
\end{figure}
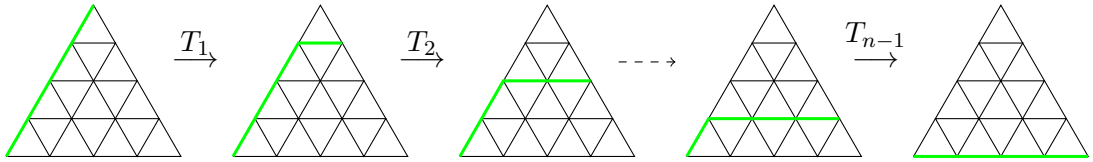

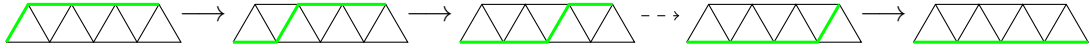
\begin{figure}[htp]
	\begin{tikzpicture}[scale=0.5]
	\draw (0,0)-- (0.58,1)-- (4.04,1)-- (4.62,0)-- (0,0);
	\draw (0.58,1)-- (1.15,0)-- (1.73,1)-- (2.31,0)-- (2.89,1)-- (3.46,0)-- (4.04,1);
	\draw[green, line width=0.4mm] (0,0)-- (0.58,1)-- (4.04,1);
	\draw (5.2,0.7) node {$\longrightarrow$};
	\draw (6,0)-- (6.58,1)-- (10.04,1)-- (10.62,0)-- (6,0);
	\draw (6.58,1)-- (7.15,0)-- (7.73,1)-- (8.31,0)-- (8.89,1)-- (9.46,0)-- (10.04,1);
	\draw[green, line width=0.4mm] (6,0)-- (7.15,0)-- (7.73,1)-- (10.04,1);
	\draw (11.2,0.7) node {$\longrightarrow$};
	\draw (12,0)-- (12.58,1)-- (16.04,1)-- (16.62,0)-- (12,0);
	\draw (12.58,1)-- (13.15,0)-- (13.73,1)-- (14.31,0)-- (14.89,1)-- (15.46,0)-- (16.04,1);
	\draw[green, line width=0.4mm] (12,0)-- (14.31,0)-- (14.89,1)-- (16.04,1);
	\draw[dashed,->] (16.8,0.7) -- (17.8,0.7);
	\draw (18,0)-- (18.58,1)-- (22.04,1)-- (22.62,0)-- (18,0);
	\draw (18.58,1)-- (19.15,0)-- (19.73,1)-- (20.31,0)-- (20.89,1)-- (21.46,0)-- (22.04,1);
	\draw[green, line width=0.4mm] (18,0)-- (21.46,0)-- (22.04,1);
	\draw (23.2,0.7) node {$\longrightarrow$};
	\draw (24,0)-- (24.58,1)-- (28.04,1)-- (28.62,0)-- (24,0);
	\draw (24.58,1)-- (25.15,0)-- (25.73,1)-- (26.31,0)-- (26.89,1)-- (27.46,0)-- (28.04,1);
	\draw[green, line width=0.4mm] (24,0)-- (28.62,0);
	\end{tikzpicture}
	\caption{Intermediate snake moves in each step.}
	\label{intmv}
\end{figure}
Notice that all the snake moves consist of two kinds of elementary snake moves, namely, Type \textit{I} and Type \textit{II}, as shown in \Cref{elemSnake}.
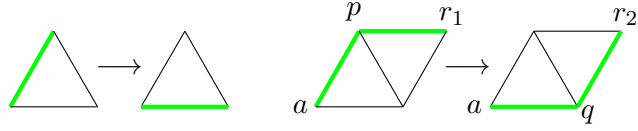
\begin{figure}[ht]
	\centering
	\begin{tikzpicture}[scale=0.5]
	\draw [line width=1.6pt,color=green] (0,0)-- (1.15,2);
	\draw [line width=1.6pt,color=green] (1.15,2)-- (3.46,2);
	\draw (3.46,2)-- (2.31,0);
	\draw (2.31,0)-- (0,0);
	\draw (1.15,2)-- (2.31,0);
	\draw (5.77,2)-- (4.62,0);
	\draw (5.77,2)-- (8.08,2);
	\draw [line width=1.6pt,color=green] (6.93,0)-- (4.62,0);
	\draw (5.77,2)-- (6.93,0);
	\draw [line width=1.6pt,color=green] (8.08,2)-- (6.93,0);
	\draw (-4.62,0)-- (-3.46,2);
	\draw (-3.46,2)-- (-2.31,0);
	\draw [line width=1.6pt,color=green] (-2.31,0)-- (-4.62,0);
	\draw [line width=1.6pt,color=green] (-8.08,0)-- (-6.93,2);
	\draw (-6.93,2)-- (-5.77,0);
	\draw (-5.77,0)-- (-8.08,0);
	\draw (-5.2,1) node {$\longrightarrow$};
	\draw (4,1) node {$\longrightarrow$};
        \draw (-0.4,0) node {$a$};
        \draw (1.01,2.52) node {$p$};
        \draw (7.2,-0.2) node {$q$};
        \draw (3.6,2.4) node {$r_1$};
        \draw (4.2,0) node {$a$};
        \draw (8.2,2.39) node {$r_2$};
	\end{tikzpicture}
	\caption[Elementary snake moves]{Elementary snake moves of Type \textit{I} (left) and Type \textit{II} (right).}
	\label{elemSnake}
\end{figure}
Also, notice that elementary snake moves of Type \textit{I} occur only at the end of each intermediate snake moves and in this snake move, only the last basis vector is changed. Let
\begin{equation}\label{F_i}
F_i \coloneqq \begin{pmatrix}
1 & 0 & & \cdots &  & 0 \\
0 & \ddots   \\
\vdots & \ddots & 1 & 0 & & \vdots \\
& & 1 & 1 & \\
\vdots & & & \ddots & \ddots & 0 \\
0 & \cdots & & & 0 & 1 
\end{pmatrix}_{n\times n}
\
\begin{matrix}\text{(i.e., 1 on the diagonals and the only}\\ \text{ non-diagonal 1 at $(i+1,i)$-th position).} \end{matrix}
\end{equation} 
Clearly, $F_{n-1}$ is the change of basis matrix corresponding to the elementary snake move of Type \textit{I}. So we get, $T_1=F_{n-1}$. Let
\begin{equation}\label{H_i}
H_i(x) \coloneqq diag(1,1,\cdots,1,\underbrace{x,\cdots, x}_{i}).
\end{equation}
Let us consider the Type \textit{II} move shown in \Cref{elemSnake} and let $X_{a,b,c}$ be the triangle invariant associated with the downward facing small triangle for $a+b+c=n$. Similar to \cref{vReln}, we get \[ a+p=q, \quad r_2=r_1X_{a,b,c}. \]
 Now notice that, $H_i(X_{a,b,c})F_{n-i-1}$ for some $i$ is the change of basis matrix for the elementary snake move of Type \textit{II}, similar to the $T_2$ matrix in \Cref{n3M}. In \Cref{snkmv}, notice that the snake-move for each $T_i$ matrix for $1<i$ consists of $(i-1)$ elementary moves of Type \textit{II} followed by one elementary Type \textit{I} move. Adjusting the indices properly, we obtain 
$$T_2=F_{n-1}H_1(X_{0,0,n-3})F_{n-2}$$
$$\cdots$$
\[ T_{n-1}=F_{n-1}H_1(X_{n-3,0,0})F_{n-2}H_2(X_{n-4,1,0})F_{n-3}\cdots H_{n-2}(X_{0,n-3,0})F_1. \] 
So we finally obtain 
\begin{equation}\label{Mt}
M(t)=T_{n-1}\cdots T_2T_1= \prod_{j=1}^{n-1}\bigg[F_{n-1}\prod_{i=1}^{n-j-1}\bigg(H_i(X_{i-1,n-i-j-1,j-1})F_{n-i-1}\bigg)\bigg]
\end{equation}

Now to derive an expansion of the matrix $E(e)$, we shall first handle the base case for $n=2$. Let $A,B,C,D$ be four flags associated counterclockwise with the vertices of two adjacent triangles as in \Cref{edgmat}.

\begin{figure}[ht]
	\begin{center}
		\begin{tikzpicture}[scale=0.4]
		\draw (0,0)-- (2,2)-- (4,0)-- (2,-2)-- (0,0)-- (4,0);
		\draw[green, line width=0.6mm] (0,0)-- (4,0);
		\draw (-0.4,0) node {$A$} (2,-2.4) node {$B$} (4.4,0) node {$C$} (2,2.4) node {$D$};
		\end{tikzpicture}
		\caption{Four flags around an edge.}
		\label{edgmat}
	\end{center}
\end{figure}
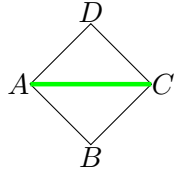

 Fix a vector $a$ from the one-dimensional subspace of $A$. Then we obtain unique vectors $c,d$ in the one-dimensional subspaces of $C$ and $D$ respectively such that $a+d=c$. Moreover, we obtain unique vectors $c',b$ in the one dimensional subspaces of $C$ and $B$ respectively such that $a+c'=b$. Let $(AC)^+$ and $(AC)^-$ denote the snakes $AC$ that come from the triangles $\Delta ACD$ and $\Delta ABC$ respectively. Then $\{a, c\}$ and $\{a,c'\}$ are the projective bases corresponding to the snakes $(AC)^+$ and $(AC)^-$ respectively. We can easily check that $c'=Z\cdot c$, where $Z$ is the cross-ratio of the projectivization of the lines associated to $A,B,C$ and $D$ in $\mathbb{CP}^1$. Hence we get
\[  
M_{(AC)^+ \rightarrow (AC)^-}=
\begin{pmatrix}
1 & 0 \\
0 & Z
\end{pmatrix}\]
and hence 
\[ E(e)=M_{(CA)^+ \rightarrow (AC)^-}=M_{(AC)^+ \rightarrow (AC)^-}\cdot M_{(CA)^+ \rightarrow (AC)^+}=
\begin{pmatrix}
1 & 0\\
0& Z
\end{pmatrix}\cdot
\begin{pmatrix}
0 & 1\\
-1 & 0
\end{pmatrix}=
\begin{pmatrix}
0 & 1\\
-Z & 0
\end{pmatrix}.  \] 
We obtain the general form 
\begin{equation}
M_{(AC)^+\rightarrow (AC)^-}= diag(1,z_{n-2},z_{n-3}z_{n-2},\cdots,z_0z_1\cdots z_{n-2}),
\end{equation}
by induction on $n$, where $z_{n-2},z_{n-3},\cdots,z_0$ are the coordinates associated to the edge $CA$. For more details, readers are referred to \S9 of \cite{FG06} and \S2 of \cite{Douglas}. As before, we have
\begin{equation}\label{Ee}
E(e)=M_{(AC)^+\rightarrow (AC)^-}\cdot S = diag(1,z_{n-2},z_{n-3}z_{n-2},\cdots,z_0z_1\cdots z_{n-2}) \cdot S.
\end{equation}
In particular, for $n=3, M_{(AC)^+\rightarrow (AC)^-}= diag(1,y,xy)$ where $e= \{x,y\}$ are the edge-invariants of the edge and 
\begin{equation}\label{Ee3}
E(e)= \begin{pmatrix}
1 & 0 & 0 \\
0 & y & 0 \\
0 & 0 & xy
\end{pmatrix}\cdot
\begin{pmatrix}
0&0&1\\
0&-1&0\\
1&0&0
\end{pmatrix}=
\begin{pmatrix}
0&0&1\\
0&-y&0\\
xy&0&0
\end{pmatrix}.
\end{equation}

\vspace{.1in} 

We shall crucially  use \eqref{Tt}, \eqref{Ee} and the decomposition \eqref{Mt} later in this article.

\subsection{Weighted planar network and weight matrix}\label{wpnwm}
In this subsection we shall introduce the basic theory of weighted planar networks; this arises in the context of  totally positive matrices that we first define (see, for example, \cite{FZ}), and forms an important tool in the proof of our main result.

\begin{defn}\label{totpos}
	A real $n\times n$ matrix $M$ is said to be \emph{totally positive (respectively, non-negative)} if all its minors are positive (respectively, non-negative).
\end{defn}

\begin{defn}[Weighted planar network]
	A {(weighted) planar network} $(\Gamma,\omega)$ of order $n$ is an acyclic directed planar graph $\Gamma$ in which $2n$ boundary vertices are distinguished as $n$ sources and $n$ sinks and each of the edges $e$ are assigned scalar weights $\omega(e)$.  See  \Cref{network} for an example -- note that in the figures we shall assume the edges are directed from left to right; in particular, A slanted edge of type `` \tikz{\draw (0,0) -- (-0.3,0.3);} " is directed in the downward ($\searrow$) direction and a slanted edge of type `` \tikz{\draw (0,0) -- (0.3,0.3);} " is directed in the upward ($\nearrow$) direction. 
\end{defn}
 
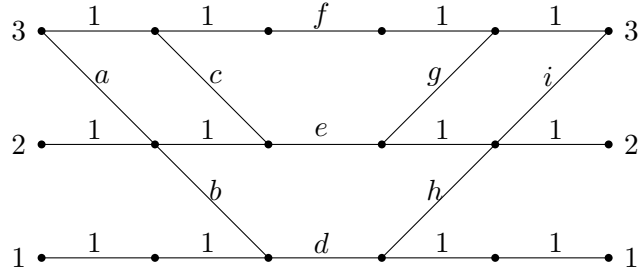
\begin{figure}[ht]
\begin{center}
	\begin{tikzpicture}
	\fill (0,0) circle (1.5pt) (1.5,0) circle (1.5pt) (3,0) circle (1.5pt) (4.5,0) circle (1.5pt) (6,0) circle (1.5pt) (7.5,0) circle (1.5pt) (0,1.5) circle (1.5pt) (1.5,1.5) circle (1.5pt) (3,1.5) circle (1.5pt) (4.5,1.5) circle (1.5pt) (6,1.5) circle (1.5pt) (7.5,1.5) circle (1.5pt) (0,3) circle (1.5pt) (1.5,3) circle (1.5pt) (3,3) circle (1.5pt) (4.5,3) circle (1.5pt) (6,3) circle (1.5pt) (7.5,3) circle (1.5pt);
	\draw (0,0) -- (7.5,0);
	\draw (0,1.5) -- (7.5,1.5);
	\draw (0,3) -- (7.5,3);
	\draw (0,3) -- (3,0);
	\draw (4.5,0) -- (7.5,3);
	\draw (1.5,3) -- (3,1.5);
	\draw (4.5,1.5) -- (6,3);
	\draw (-.3,0) node {$1$} (7.8,0) node {$1$} (-.3,1.5) node {$2$} (7.8,1.5) node {$2$} (-.3,3) node {$3$} (7.8,3) node {$3$} (.7,.2) node {$1$} (2.2,.2) node {$1$} (3.7,.2) node {$d$} (5.3,.2) node {$1$} (6.8,.2) node {$1$} (2.3,.9) node {$b$} (5.2,.9) node {$h$} (.7,1.7) node {$1$} (2.2,1.7) node {$1$} (3.7,1.7) node {$e$} (5.3,1.7) node {$1$} (6.8,1.7) node {$1$} (.8,2.4) node {$a$} (2.3,2.4) node {$c$} (5.2,2.4) node {$g$} (6.7,2.4) node {$i$} (.7,3.2) node {$1$} (2.2,3.2) node {$1$} (3.7,3.2) node {$f$} (5.3,3.2) node {$1$} (6.8,3.2) node {$1$};
	\end{tikzpicture}
	\caption{A weighted planar network.}
	\label{network}
 \end{center}
\end{figure}

 
 \begin{defn}[Weights of paths, weight matrix]\label{wmatrix}
 	 The weight of a directed path in $\Gamma$ is defined to be the product of the weights of its edges. The {weight matrix} $\mathcal{W}(\Gamma;\omega)$ corresponding to a planar network $(\Gamma,\omega)$ of order $n$ is an $n \times n$ matrix whose $(i,j)$-entry is the sum of weights of all paths from the source $i$ to the sink $j$. We simply denote it by $\mathcal{W}(w)$ when the planar network $\Gamma$ is known and fixed.
 \end{defn} 
 \begin{example}
 	 The weight matrix of the network in \Cref{network} is the following
 	 \[ \begin{pmatrix}
 	 d & dh & dhi \\
 	 bd & bdh+e & bdhi+eg+ei \\
 	 abd & abdh+ae+ce & abdhi+e(a+c)(g+i)+f
 	 \end{pmatrix}.\] 
 \end{example}

 One can in fact, show that \textit{any} $3\times 3$ totally-nonnegative matrix is of the above form (see \cite{FZ}). 

 \vspace{.05in}

\noindent As mentioned earlier, weighted planar networks arise in the context of totally non-negative matrices via the following basic result. We were led to considering weighted planar networks because of this, and the fact (proved by Fock-Goncharov) that for a Hitchin representation $\rho: \pi_1(S)\rightarrow \pslnr,\ \rho(\gamma)$ is totally positive (non-negative) for any $\gamma\in \pi_1(S)$ is non-peripheral (resp. peripheral).

\begin{thm}\label{bren}\cite[Theorem 3.1]{Brenti}
	Every totally nonnegative matrix is the weight matrix of a planar network.
\end{thm}

\noindent We now mention some additional properties of weight matrices.  First, we define:

\begin{defn}[Concatentation] If $(\Gamma_1,\omega_1)$ and $(\Gamma_2,\omega_2)$ are two planar networks of of order $n$ (i.e.\ with $n$ sources and $n$ sinks). If we connect all the sinks of $\Gamma_1$ with the respective sources of $\Gamma_2$, we obtain a new planar network $(\Gamma,\omega)$ of order $n$, whose sources are those of $(\Gamma_1,\omega_1)$ and sinks are those of $(\Gamma_2,\omega_2)$. This is called the \emph{concatenation} of two planar networks. (See  \Cref{concat} for an example.) 
\end{defn}

\begin{figure}[ht]
	\begin{tikzpicture}
	\draw (4,4)-- (6,2);
	\draw (7,2)-- (9,4);
	\draw (4,4)-- (9,4);
	\draw (4,2)-- (9,2);
	\draw (4,3)-- (9,3);
	\draw (5,4)-- (6,3);
	\draw (7,3)-- (8,4);
	\draw (11,3)-- (16,3);
	\draw (16,4)-- (11,4);
	\draw (11,2)-- (16,2);
	\draw (11,2)-- (13,4);
	\draw (12,2)-- (13,3);
	\draw (16,2)-- (14,4);
	\draw (15,2)-- (14,3);
	\draw (5,0.4)-- (15,0.4);
	\draw (5,-0.6)-- (15,-0.6);
	\draw (5,-1.6)-- (15,-1.6);
	\draw (5,0.4)-- (7,-1.6);
	\draw (6,0.4)-- (7,-0.6);
	\draw (8,-0.6)-- (9,0.4);
	\draw (8,-1.6)-- (10,0.4);
	\draw (10,-1.6)-- (12,0.4);
	\draw (11,-1.6)-- (12,-0.6);
	\draw (13,-0.6)-- (14,-1.6);
	\draw (15,-1.6)-- (13,0.4);
	\begin{scriptsize}
	\fill [color=black] (4,2) circle (1.5pt);
	\draw[color=black] (3.82,2.04) node {$1$};
	\fill [color=black] (4,3) circle (1.5pt);
	\draw[color=black] (3.78,3.04) node {$2$};
	\fill [color=black] (4,4) circle (1.5pt);
	\draw[color=black] (3.78,4) node {3};
	\fill [color=black] (6,2) circle (1.5pt);
	\fill [color=black] (9,4) circle (1.5pt);
	\draw[color=black] (9.22,4.04) node {3};
	\fill [color=black] (9,2) circle (1.5pt);
	\draw[color=black] (9.2,2.08) node {1};
	\fill [color=black] (7,2) circle (1.5pt);
	\draw[color=black] (6.56,1.56) node {(a)};
	\fill [color=black] (9,3) circle (1.5pt);
	\draw[color=black] (9.2,3.06) node {2};
	\fill [color=black] (5,4) circle (1.5pt);
	\fill [color=black] (6,3) circle (1.5pt);
	\fill [color=black] (7,3) circle (1.5pt);
	\fill [color=black] (8,4) circle (1.5pt);
	\fill [color=black] (11,4) circle (1.5pt);
	\draw[color=black] (10.72,4.08) node {3};
	\fill [color=black] (11,3) circle (1.5pt);
	\draw[color=black] (10.76,3.06) node {2};
	\fill [color=black] (11,2) circle (1.5pt);
	\draw[color=black] (10.82,2.06) node {1};
	\fill [color=black] (13,4) circle (1.5pt);
	\fill [color=black] (16,4) circle (1.5pt);
	\draw[color=black] (16.2,4.06) node {$3$};
	\fill [color=black] (16,2) circle (1.5pt);
	\draw[color=black] (16.2,2.06) node {1};
	\fill [color=black] (16,3) circle (1.5pt);
	\draw[color=black] (16.22,3.06) node {2};
	\draw[color=black] (13.58,1.54) node {(b)};
	\fill [color=black] (14,4) circle (1.5pt);
	\fill [color=black] (12,2) circle (1.5pt);
	\fill [color=black] (13,3) circle (1.5pt);
	\fill [color=black] (14,3) circle (1.5pt);
	\fill [color=black] (15,2) circle (1.5pt);
	\fill [color=black] (5,0.4) circle (1.5pt);
	\draw[color=black] (4.72,0.46) node {3};
	\fill [color=black] (15,0.4) circle (1.5pt);
	\draw[color=black] (15.28,0.44) node {3};
	\fill [color=black] (5,-0.6) circle (1.5pt);
	\draw[color=black] (4.72,-0.54) node {2};
	\fill [color=black] (15,-0.6) circle (1.5pt);
	\draw[color=black] (15.26,-0.54) node {2};
	\fill [color=black] (5,-1.6) circle (1.5pt);
	\draw[color=black] (4.76,-1.52) node {1};
	\fill [color=black] (15,-1.6) circle (1.5pt);
	\draw[color=black] (15.26,-1.56) node {1};
	\draw[color=black] (9.98,-2.06) node {(c)};
	\fill [color=black] (7,-1.6) circle (1.5pt);
	\fill [color=black] (8,-1.6) circle (1.5pt);
	\fill [color=black] (10,0.4) circle (1.5pt);
	\fill [color=black] (10,-1.6) circle (1.5pt);
	\fill [color=black] (12,0.4) circle (1.5pt);
	\fill [color=black] (13,0.4) circle (1.5pt);
	\fill [color=black] (6,0.4) circle (1.5pt);
	\fill [color=black] (7,-0.6) circle (1.5pt);
	\fill [color=black] (8,-0.6) circle (1.5pt);
	\fill [color=black] (9,0.4) circle (1.5pt);
	\fill [color=black] (11,-1.6) circle (1.5pt);
	\fill [color=black] (12,-0.6) circle (1.5pt);
	\fill [color=black] (13,-0.6) circle (1.5pt);
	\fill [color=black] (14,-1.6) circle (1.5pt);
	\end{scriptsize}
	\end{tikzpicture}
	\caption{The planar network (c) is the concatenation of (a) and (b).}
	\label{concat}
\end{figure}
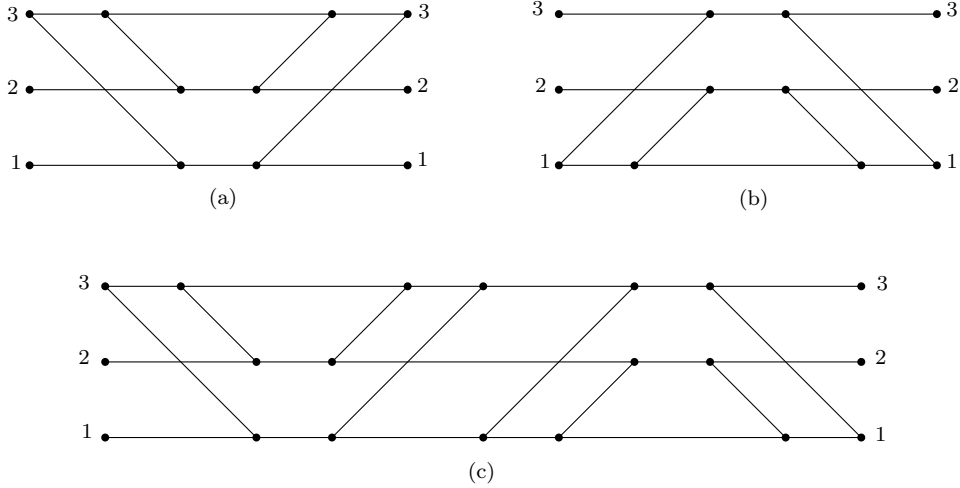

\begin{prop}\label{prp:1}
	Let $(\Gamma_1,\omega_1)$ and $(\Gamma_2,\omega_2)$ be two planar networks of order $n$ and $A$ and $B$ be their weight matrices respectively. Then the weight matrix of the concatenated planar network $(\Gamma,\omega)$ is $AB$.
\end{prop}
\begin{proof}
	Let $A=(a_{ij})_{n \times n},\ B=(b_{ij})_{n \times n}$ and $d_1,\ d_2, \cdots ,d_n$ be the concatenated vertices of $(\Gamma,\omega)$. Let $C=(c_{ij})_{n\times n}$ be the weight matrix of $(\Gamma,\omega)$. Clearly the sources and sinks of $(\Gamma,\omega)$ are the sources of $(\Gamma_1,\omega_1)$ and sinks of $(\Gamma_2,\omega_2)$ respectively. So $c_{ij}$ is the sum of the weights of all the paths from the $i$-th source of $(\Gamma_1,\omega_1)$ to the $j$-th sink of $(\Gamma_2,\omega_2)$ in $(\Gamma,\omega)$. Now every path from source $i$ to sink $j$ in $(\Gamma,\omega)$ has to pass through exactly one of these $d_i$'s. Let $p$ be such a path that passes through $d_k$ for some $k$. Then $a_{ik}$ is the sum of weights of all paths from the source $i$ to the sink $k$ in $(\Gamma_1,\omega_1)$. Let $P^1_{ik}$ be the set (possibly empty) of all such paths. Similarly, let $P^2_{kj}$ be the set of all the paths from source $k$ to sink $j$ in $(\Gamma_2,\omega_2)$. Each path in $(\Gamma,\omega)$ from source $i$ to sink $j$ that passes through $d_k$ comes by concatenating some path in $P^1_{ik}$ with some path in $P^2_{kj}$ and vice versa. So we get, \[ c_{ij}= \sum_{k=1}^{n}a_{ik}b_{kj}. \] Hence $C=AB$.   
\end{proof}

\vspace{.05in} 

The following \Cref{linds}, known as Lindstr\"{o}m's Lemma, shall play a crucial role later.  It shows how to compute minors of the weight matrix of a planar network directly from the planar network.

\begin{defn}[Vertex-disjoint family] 
	Let $I,J\subseteq \{1,\cdots,n \}$ such that $|I|=|J|$. A \emph{vertex-disjoint family} of paths from sources $I$ to sinks $J$ is a collection of disjoint paths on the planar network (that no two of them even have a common vertex), each of which starts from some sources in $I$ and goes upto some sinks in $J$. Also, for each $i\in I$, the family must contain a path that starts from the source $i$. The weight of a vertex-disjoint family of paths is the product of weights of all the paths in the family.
\end{defn} 
\begin{example}
	Let us consider the planar network in \Cref{network}. Let $I=\{1,2\}$ and $J=\{2,3\}$. There is only one path from source 1 to sink 2 with weight $dh$, let us call it by $p_1$. There are 3 paths from source 2 to sink 3 with weights $bdhi,ei,eg$. Let us call them by $p_2,p_3,p_4$ respectively. Notice that both $p_2$ and $p_3$ intersects with $p_1$, but $p_4$ does not. Hence, $\{ p_1,p_4 \}$ is the only family of vertex disjoint paths from sources $I$ to sinks $J$ with weight $degh$.
\end{example}
Let $A$ be a $n\times n$ matrix and $I,J\subseteq \{1,\cdots,n \}$ such that $|I|=|J|$. Then by $\Delta_{I,J}$, we denote the minor of $A$ which is the determinant of the submatrix of $A$ consisting of rows indexed by $I$ and columns indexed by $J$.
\begin{lem}\cite{Lind}\label{linds}
	A minor $\Delta_{I,J}$ of the weight matrix of a planar
	network is equal to the sum of weights of all collections of vertex-disjoint paths from the sources indexed by $I$ to the sinks indexed by $J$.
\end{lem}
\begin{proof}
    See \cite[Lemma 7]{CN} or \cite[Lemma 1]{FZ} for the proof.
\end{proof}
\begin{cor}\label{rnngw}
	If all the weights of a planar network are real and nonnegative (positive), then its weight matrix is totally nonnegative (resp. positive).
\end{cor}

\subsection{Some linear algebra}\label{cwf}
We recall some basic linear algebra facts that we shall need later. In what follows, for a vector $x=(x_1,\cdots,x_n)\in \mathbb{R}^n$, we say $x>0$ (or $x\ge 0$) if $x_i>0$ (resp. $x_i\ge 0$) for all $i$. Moreover, a \textit{positive} (respectively, non-negative) matrix is one whose entries are all positive  (respectively, non-negative). Note that a totally-positive matrix (defined in the previous subsection) is also a positive matrix.

Recall also that the \emph{spectral radius} of a matrix $A$ is the largest eigenvalue in modulus, i.e.  $\sigma(A):= \max \{|\lambda| : \lambda \text{ is an eigenvalue of } A  \}.$

\vspace{.05in}

The following collects some basic facts about the spectral radius of positive matrices. For a proof, see for example \cite[\S 8]{Meyer}.

\begin{thm}[Perron-Frobenius]\label{PerFro} 
	Let $A$ be a positive square matrix. Then the following statements are true: 
	\begin{enumerate}
		\item $\sigma(A)>0.$
		
		\item $\sigma(A)$ is also an eigenvalue of $A$ and it is called the Perron-Frobenius eigenvalue of $A$.
		
		\item\label{posev} There exists an eigenvector $x>0$ such that $Ax=\sigma(A)x$.
\end{enumerate}
  \end{thm}

We shall also need the following:

\begin{lem}[The Collatz-Wielandt formula] For a positive matrix $A$ we have \[ \sigma(A)=\max_{x\in \mathcal{X}}f(x) \]
		where \[ f(x)=\min_{\substack{1\le i\le n \\ x_i \neq 0}}\frac{(Ax)_i}{x_i} \text{ and } \mathcal{X}=\{x\ge 0, x\neq 0  \}. \]
\end{lem}
  
\begin{proof}
	Let $\alpha=f(x)$ for some $x \in \mathcal{X}$. Since $f(x)\le \frac{(Ax)_i}{x_i}$ for all $i$, we have $0\le \alpha x \le Ax$. Let $p^T$ be a left-side eigenvector of $A$ corresponding to $\sigma(A)$. Then \[ \alpha x \le Ax \implies \alpha p^T x \le (p^T A)x = \sigma(A)p^T x \implies \alpha \le \sigma(A).  \]
	So we get $f(x)\le \sigma(A)$ for all $x\in \mathcal{X}$. Hence  $\max_{x\in \mathcal{X}}f(x) \le \sigma(A)$. Now let $q$ be an eigenvector corresponding to $\sigma(A)$. Note that $q>0$ because of \ref{posev}. Then \[ f(q)=\min_{1\le i \le n}\frac{(Aq)_i}{q_i}=\min_{1\le i \le n}\frac{(\sigma(A)q)_i}{q_i}=\sigma(A). \]
\end{proof}

\begin{rmk}
    The above formula also holds for a non-negative matrix $A$ by approximating it by positive matrices -- see \cite[\S8.3]{Meyer}.
\end{rmk}

\noindent As a consequence of the previous lemma we have the following elementary lemma which shall be crucial to our proof of domination :

\begin{lem}{\label{spctrl}}
	Let $(\Gamma, \omega)$ be a weighted planar network of order $n$ with complex weights and weight matrix $\mathcal{W}$. Let $(\Gamma, \lvert \omega\rvert)$ be the same planar network, but with each weight replaced by its modulus, and let  $\mathcal{W}^\prime$ be its weight matrix. Then the spectral radii of the two weight matrices satisfy $\sigma(\mathcal{W}) \le \sigma(\mathcal{W}^\prime)$. 
\end{lem}
\begin{proof}
	Let  the matrices $\mathcal{W} =(w_{ij})$ and $\mathcal{W}^\prime=(\textbf{w}_{ij})$ where $1\leq i,j \leq n$, where recall that the $ij$-th entry is a sum of weights of paths from source $i$ to sink $j$. Since the weight of a path is the product of weights, any such weight is replaced by its modulus when all weights are replaced by their moduli. Hence by the triangle inequality, we have $|w_{ij}| \le \textbf{w}_{ij}$. Let $\lambda$ be an eigenvalue of $\mathcal{W}$ with some eigenvector $v=(v_i)_{1\leq i\leq n}$, such that  $\mathcal{W} v=\lambda v$. Now for each $i$ we have 
	\[ |\lambda||v_i|=|\lambda v_i|=|\sum_j w_{ij}v_j| \le \sum_j |w_{ij}||v_j|\le \sum_j \mathbf{w}_{ij}|v_j|. \]
	Let $\mathbf{v}=(\mathbf{v}_i)=(|v_i|)$ be the vector with i-th entry $|v_i|$. Then notice that the previous computation shows \[ |\lambda| \le \frac{(\mathcal{W}^\prime \mathbf{v})_i}{\mathbf{v}_i} \] for all $i$. Thus \[ |\lambda| \le \min_i \frac{(\mathcal{W}^\prime \mathbf{v})_i}{\mathbf{v}_i} \le \max_{x\in \mathcal{X}} \min_i \frac{(\mathcal{W}^\prime x)_i}{x_i} = \sigma(\mathcal{W}^\prime) \]
	by the Collatz-Wielandt formula. Since we choose $\lambda$ to be an arbitrary eigenvalue of $\mathcal{W}$, we get $\sigma(\mathcal{W}) \le \sigma(\mathcal{W}^\prime)$ as desired.
\end{proof}

\begin{rmk}
    For an alternative proof of this, using the characterization of spectral radius as a limit $\sigma(A) = \lim\limits_{k\to \infty} \lVert A^k\rVert^{1/k}$, see \cite[Theorem 8.1.18]{HJ}.
\end{rmk}

\section{Domination in the Hilbert length spectrum}\label{hilp}

In this section, we shall prove \Cref{mainthm} for the Hilbert length spectrum. We shall first prove it when $n=3$, i.e., for representations into $\psltc$, and subsequently generalize to any $n\geq 3$. 

\subsection{Domination for $\psltc$-representations}\label{neq3}
Let $\rho: \pi_1(S) \rightarrow \psltc$ be a given generic representation, where recall that $S$ is a \textit{punctured} surface throughout. Let us fix an ideal triangulation $\tau$ on $S$. Since $\rho$ is generic, we obtain a framing $\beta$  such that the Fock-Goncharov coordinates for the framed representation $(\rho, \beta)$ are well-defined, as discussed in Section 2.3. Let  $\gamma \in \pi_1(S)$ realized by a closed curve on $S$ and let $\widetilde{\gamma}$ be its lift to the universal cover. Let $\widetilde{\gamma_0}$ be an arc on $\widetilde{\gamma}$ that descends onto $\gamma$ which is a concatenation $\widetilde{\gamma_0} \sim e_1 * t_1*\cdots * e_k * t_k$, where $e_i$ and $t_i$ are edges of the monodromy graph (\textit{c.f.} \Cref{monod}). Recall that, from \cref{rhogamma}, we have
\begin{equation}\label{eq:1}
\rho(\gamma)=T(t_k)^{\delta_k}E(e_k)\cdots T(t_1)^{\delta_1}E(e_1), 
\end{equation}
where $ \delta_i \in \{\pm 1\}$. In \cref{Tt} and \cref{Ee}, we have already seen that, $T$ and $E$ matrices are of the form
 \[  E(x,y)= \sqrt[3]{\frac{1}{xy^2}} \begin{pmatrix}
	0 & 0 & 1 \\
	0 & -y & 0 \\
	x y & 0 & 0
	 \end{pmatrix}, \hspace*{0.5in}
	 T(z)= \frac{1}{\sqrt[3]{z}} \begin{pmatrix}
	 0 & 0 & 1 \\
	 0 & -1 & 1 \\
	 z & -z-1 & 1
	 \end{pmatrix} \]
where the multiplicative factors ensure that the matrices are normalized to have determinant $1$. (Throughout, we shall use this normalization to be able to think of a matrix in $\pslnc$ as a matrix in $\slnc$.) 

\noindent We shall refer to the matrices
\begin{equation}
	 T(z)E(x,y)= \sqrt[3]{\frac{1}{xy^2z}}
	 \begin{pmatrix}
	 x y &0 & 0 \\
	 x y & y & 0 \\
	 x y & y(1+z) & z
	 \end{pmatrix}\label{eq:2}
 \end{equation}
 \begin{equation}
	 T(z)^{-1}E(x,y)=\sqrt[3]{\frac{z}{xy^2}} 
	 \begin{pmatrix}
	 \frac{x y}{z} & \frac{y(1+z)}{z} & 1 \\
	 0 & y & 1 \\
	 0 & 0 & 1
	 \end{pmatrix} \label{eq:3}
 \end{equation}
	 
 as the \emph{building blocks} of $\rho(\gamma)$. We can easily check that they are the weight matrices of the planar networks in \Cref{tEtiE}. (In the figure the weights $p=\sqrt[3]{\frac{z}{xy^2}}$ and $q=\sqrt[3]{\frac{1}{xy^2z}}$.)
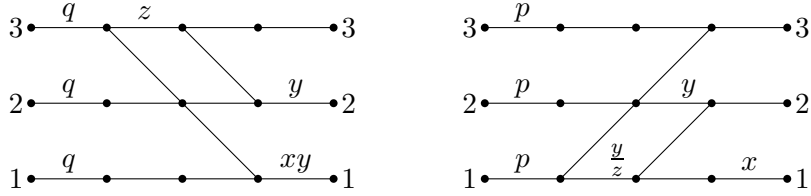
\begin{figure}[ht]
	\centering
	\begin{tikzpicture}
	\fill (0,0) circle (1.5pt) (1,0) circle (1.5pt) (2,0) circle (1.5pt) (3,0) circle (1.5pt) (4,0) circle (1.5pt) (0,1) circle (1.5pt) (1,1) circle (1.5pt) (2,1) circle (1.5pt) (3,1) circle (1.5pt) (4,1) circle (1.5pt) (0,2) circle (1.5pt) (1,2) circle (1.5pt) (2,2) circle (1.5pt) (3,2) circle (1.5pt) (4,2) circle (1.5pt);
	\draw (0,0) -- (4,0);
	\draw (0,1) -- (4,1);
	\draw (0,2) -- (4,2);
	\draw (1,2) -- (3,0);
	\draw (2,2) -- (3,1);
	\draw (-.2,0) node {$1$} (4.2,0) node {$1$} (-.2,1) node {$2$} (4.2,1) node {$2$}(-.2,2) node {$3$} (4.2,2) node {$3$} (.5,0.2) node {$q$} (.5,1.2) node {$q$} (.5,2.2) node {$q$} (1.5,2.2) node {$z$} (3.5,1.2) node {$y$} (3.5,.2) node {$xy$};
	\fill (6,0) circle (1.5pt) (7,0) circle (1.5pt) (8,0) circle (1.5pt) (9,0) circle (1.5pt) (10,0) circle (1.5pt) (6,1) circle (1.5pt) (7,1) circle (1.5pt) (8,1) circle (1.5pt) (9,1) circle (1.5pt) (10,1) circle (1.5pt) (6,2) circle (1.5pt) (7,2) circle (1.5pt) (8,2) circle (1.5pt) (9,2) circle (1.5pt) (10,2) circle (1.5pt);
	\draw (6,0) -- (10,0);
	\draw (6,1) -- (10,1);
	\draw (6,2) -- (10,2);
	\draw (7,0) -- (9,2);
	\draw (8,0) -- (9,1);
	\draw (5.8,0) node {$1$} (10.2,0) node {$1$} (5.8,1) node {$2$} (10.2,1) node {$2$} (5.8,2) node {$3$} (10.2,2) node {$3$} (6.5,0.2) node {$p$} (6.5,1.2) node {$p$} (6.5,2.2) node {$p$} (7.75,0.25) node {$\frac{y}{z}$} (9.5,0.2) node {$x$} (8.7,1.2) node {$y$};
	\end{tikzpicture}
	\caption{Planar networks with weight matrices $T(z)E(x,y)$ (left) and $T(z)^{-1}E(x,y)$ (right). Weights of the unassigned edges are 1.}
	\label{tEtiE}
\end{figure}

Concatenating these planar networks as dictated by \eqref{eq:1},  we obtain a planar network with weight matrix $\rho(\gamma)$ by \Cref{prp:1}. Note that if all the parameters are real and positive, then \Cref{rnngw} assures that the building blocks, and indeed the entire product, are totally nonnegative matrices.  In fact, we observe: 

\begin{lem}
 	The matrix $\rho(\gamma)$ is either triangular or totally positive when all the parameters are real and positive.
\end{lem}
\begin{proof}
	If the building block matrices of $\rho(\gamma)$ are all of the form $T(t)E(e)$ (or $T(t)^{-1}E(e)$), $\rho(\gamma)$ is lower (respectively upper) triangular. If not, then $\rho(\gamma)$ must contain a segment of the form $T(t_1)E(e_1)T(t_2)^{-1}E(e_2)$ or $T(t_1)^{-1}E(e_1)T(t_2)E(e_2)$. \Cref{etet_1} and \Cref{et_1et} represent the respective planar networks of these segments.
	\begin{figure}[tph]
		\centering
		\begin{tikzpicture}
		\filldraw (0,0) circle (1.5pt) (1,0) circle (1.5pt) (2,0) circle (1.5pt) (3,0) circle (1.5pt) (4,0) circle (1.5pt) (5,0) circle (1.5pt) (6,0) circle (1.5pt) (7,0) circle (1.5pt) (0,1) circle (1.5pt) (1,1) circle (1.5pt) (2,1) circle (1.5pt) (3,1) circle (1.5pt) (4,1) circle (1.5pt) (5,1) circle (1.5pt) (6,1) circle (1.5pt) (7,1) circle (1.5pt) (0,2) circle (1.5pt) (1,2) circle (1.5pt) (2,2) circle (1.5pt) (3,2) circle (1.5pt) (4,2) circle (1.5pt) (5,2) circle (1.5pt) (6,2) circle (1.5pt) (7,2) circle (1.5pt);
	 	\draw (0,0) -- (7,0);
	 	\draw (0,1) -- (7,1);
	 	\draw (0,2) -- (7,2);
	 	\draw (1,2) -- (3,0);
	 	\draw (2,2) -- (3,1);
	 	\draw (4,0) -- (6,2);
	 	\draw (5,0) -- (6,1);
 		\end{tikzpicture}
 		\caption{Planar network of $TET^{-1}E$.}
 		\label{etet_1}
 	\end{figure}
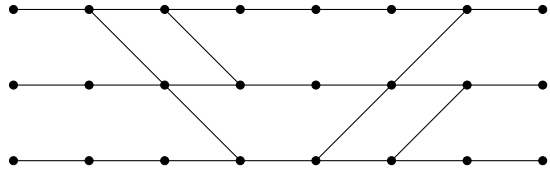 
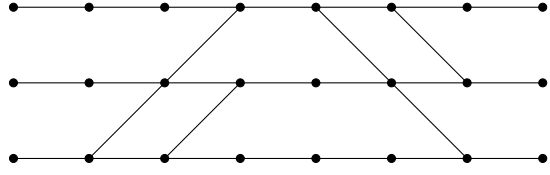
\begin{figure}[tph]
	\centering
	\begin{tikzpicture}
	\filldraw (0,0) circle (1.5pt) (1,0) circle (1.5pt) (2,0) circle (1.5pt) (3,0) circle (1.5pt) (4,0) circle (1.5pt) (5,0) circle (1.5pt) (6,0) circle (1.5pt) (7,0) circle (1.5pt) (0,1) circle (1.5pt) (1,1) circle (1.5pt) (2,1) circle (1.5pt) (3,1) circle (1.5pt) (4,1) circle (1.5pt) (5,1) circle (1.5pt) (6,1) circle (1.5pt) (7,1) circle (1.5pt) (0,2) circle (1.5pt) (1,2) circle (1.5pt) (2,2) circle (1.5pt) (3,2) circle (1.5pt) (4,2) circle (1.5pt) (5,2) circle (1.5pt) (6,2) circle (1.5pt) (7,2) circle (1.5pt);
	\draw (0,0) -- (7,0);
	\draw (0,1) -- (7,1);
	\draw (0,2) -- (7,2);
	\draw (1,0) -- (3,2);
	\draw (2,0) -- (3,1);
	\draw (4,2) -- (6,0);
	\draw (5,2) -- (6,1);
	\end{tikzpicture}
	\caption{Planar network of $T^{-1}ETE$.}
	\label{et_1et}
\end{figure}	 	
Let $I,J \subseteq \{1,2,3\}$ such that $|I|=|J|$. In the planar network with weight matrix $\rho(\gamma)$, the segment corresponding to $T(t_1)E(e_1)T(t_2)^{-1}E(e_2)$ (or $T(t_1)^{-1}E(e_1)T(t_2)E(e_2)$) clearly contains some vertex disjoint family of paths from sources $I$ to sinks $J$. For the rest of the planar network, we extend these paths through the horizontal paths in the planar network. So we get vertex disjoint paths from sources $I$ to sinks $J$ in the whole planar network. Hence, from the Lindstr\"om Lemma \ref{linds}, we can conclude that $\rho(\gamma)$ is totally positive when all the parameters are real and positive. 
\end{proof}

\vspace{.05in}

\noindent We are now in a position to define the dominating positive representation $\rho_0:\pi_1(S) \to \psltr$.

\vspace{.05in}

\begin{defn}[Defining $\rho_0$]\label{defn:rho0} Recall that the given representation $\rho:\pi_1(S) \to \psltc$ has a framing such that the resulting framed representation has well-defined Fock-Goncharov coordinates with respect to the ideal triangulation $\tau$. Then  $\rho_0:\pi_1(S) \to \psltr$ is the representation obtained by replacing each coordinate by its modulus; since each Fock-Goncharov coordinate is now real and positive, such a representation is positive (see \Cref{puncHit}). 
\end{defn}

\noindent The following then establishes the special case of Theorem \ref{mainthm} that we were aiming to prove in this subsection:

\begin{prop}\label{prop3} Let $\rho: \pi_1(S) \to \psltc$ be a (generic) representation as in the beginning of this subsection. Then the positive representation $\rho_0$ defined above dominates $\rho$ in the Hilbert length spectrum. Moreover, the lengths of the peripheral curves are unchanged. 
\end{prop}

\begin{proof}
	 Let  $\gamma \in \pi_1(S)$. Let us recall from \cref{rlen} that the Hilbert length of $\gamma$ is defined as 
	\begin{equation}\label{rle}
	\ell_{\rho}(\gamma):=\ln \Bigg| \frac{\lambda_n}{\lambda_1} \Bigg|,
	\end{equation}
	where $\lambda_n$ and $\lambda_1$ are the largest and smallest eigenvalues in modulus of $\rho(\gamma)$.
 
 Now recall that $\rho(\gamma)$ is a product of matrices as  in \cref{eq:1}. Let $\rho(\gamma)$ be the weight matrix of the planar network $(\Gamma\, \omega)$ obtained by concatenating the planar networks corresponding to the building blocks as in \Cref{tEtiE}. Observe the crucial property that any weight in the resulting planar network is a product of some Fock-Goncharov coordinates (of $\rho$ with its framing) and their reciprocals. In particular, if each Fock-Goncharov coordinate is replaced by its modulus, then each weight of $(\Gamma\, \omega)$ is replaced by its modulus. Recall from \Cref{defn:rho0} that the new real and positive coordinates define the positive representation $\rho_0$. Thus, $\rho_0(\gamma)$ is the weight matrix of the resulting weighted planar network is $(\Gamma, \vert \omega \rvert)$. By \Cref{spctrl} we know that in this process the largest eigenvalue in modulus increases.
 
 Moreover, applying the same argument to $\gamma^{-1}$, we know that the largest eigenvalue of the (totally positive) matrix $\rho_0(\gamma)^{-1}$ is at least the largest eigenvalue in modulus of $\rho(\gamma)^{-1}$. Since the largest eigenvalue in modulus of a matrix $A^{-1}$ is the reciprocal of the smallest eigenvalue in modulus of $A$, we obtain that in the above process, the smallest  eigenvalue in modulus decreases.

 Together, this shows that $\ell_\rho(\gamma) \leq \ell_{\rho_0}(\gamma)$. 

 It remains to show that the length of any peripheral curve remains unchanged in the process. If $\gamma$ is a peripheral curve then from \cite[Theorem 23]{CTT} we know that in the equation \eqref{eq:1} $\delta_i=1$ (or $-1$) for all $i$. Hence $h(\gamma)$ is either upper triangular (for $\delta_i=1$) or lower triangular (for $\delta_i=-1$). Now from equations \eqref{eq:2} and \eqref{eq:3} we can clearly see that the diagonal entries of $\rho(\gamma)$ are in the form of the product of some Fock-Goncharov coordinates; recall that these diagonal entries are exactly the eigenvalues of a triangular matrix. Thus, afterwe replace each Fock-Goncharov coordinate with its  modulus to define $\rho_0$, the eigenvalues of $\rho_0(\gamma)$ are just replaced by their modulus. In other word, the largest and smallest eigenvalues in modulus remain the same, and  $\ell_{\rho}(\gamma)=\ell_{\rho_0}(\gamma)$ for the peripheral curve $\gamma$.
 \end{proof}

\begin{rmk}\label{strmk}
	Note from the above proof that if a curve $\gamma$ involves ideal triangles and edges of the monodromy graph with Fock-Goncharov coordinates already real and positive, then the Hilbert length of $\gamma$ remains unchanged, i.e.\  $\ell_{\rho}(\gamma)=\ell_{\rho_0}(\gamma)$.
\end{rmk}

\subsection{Domination for $\pslnc$}\label{genHil} 
  We shall now prove domination in the Hilbert length spectrum for general $n\geq 3$, generalizing the argument for $n=3$ that we saw in the previous subsection. We shall continue to assume that $S$ is a punctured surface, and let $\rho: \pi_1(S) \rightarrow \pslnc$ be a generic representation.  Since $\rho$ is generic, we can define a $\rho$-equivariant framing $\beta$ such that the Fock-Goncharov coordinates exists for the framed representation $(\rho, \beta)$, with respect to an ideal triangulation $\tau$ on $S$ (\textit{c.f.} the discussion in Section 2.3.)

  Let $\gamma \in \pi_1(S)$. Recall from Section 2.2 that we have 
  \begin{equation}\label{rhogammagen}
  \rho(\gamma)=T(t_k)^{\delta_k}E(e_k)\cdots T(t_1)^{\delta_1}E(e_1),
  \end{equation}
  and recall from Section 2.5 the finer product decomposition for each $T$ and $E$ matrices. In particular, from \eqref{Tt} we know that 
  \begin{equation}\label{tgen}
  T(t)=M(t)\cdot S,
  \end{equation}where $S$ is the constant matrix defined in \cref{S}, and from \cref{Mt} we have 
  \begin{equation}\label{Mgen}
  M(t)=\prod_{j=1}^{n-1}\bigg[F_{n-1}\prod_{i=1}^{n-j-1}\bigg(H_i(X_{i-1,n-i-j-1,j-1})F_{n-i-1}\bigg)\bigg],
  \end{equation}
  where $X_{a,b,c}$ are the triangle invariants of the triangle that contains $t,\ F_i$ and $H_i$ are as in \cref{F_i} and \cref{H_i} respectively. From \cref{Ee}, we also have
  \begin{equation}\label{Egen}
      E(e)= diag(1,z_{n-1},z_{n-2}z_{n-1},\cdots,z_1z_2\cdots z_{n-1})S,
  \end{equation}
  where $e$ is an $e$-type edge of the monodromy graph and $z_i$ are the edge invariants associated to $e$.
  
  Now our goal is to use this product decomposition to produce weighted planar networks whose weight matrices are the  building block matrices of $\rho(\gamma)$, i.e., the $T(t_i)^{\delta_i}E(e_i)$ matrices in \cref{rhogammagen}.

\vspace{.05in} 

  \noindent We first observe: 
  \begin{lem}\label{M}
  	The matrix $M(t)$ is the weight matrix of a weighted planar network with each  weight having the form of a  product of some  Fock-Goncharov coordinates  and their reciprocals.
  \end{lem}
  \begin{proof}
  	Notice that the matrix $M(t)$ is a product of the matrices $H_i(x)$ and $F_j$ for some $i,j$ and some triangle invariant $x$. \Cref{hf} shows the planar networks with the required property such that their weight matrices are $H_i(x)$ and $F_i$.
  	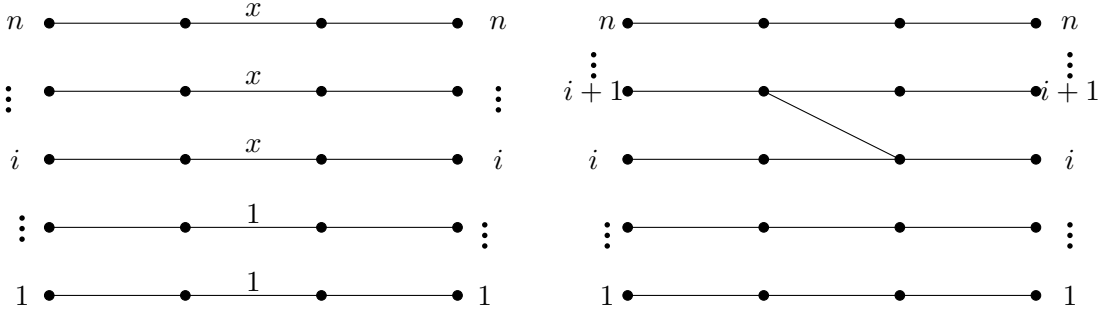
\begin{figure}[ht]
  		\centering
  		\begin{tikzpicture}[scale=0.9]
  		\filldraw (0,0) circle (2pt) (2,0) circle (2pt) (4,0) circle (2pt) (6,0) circle (2pt) (0,1) circle (2pt) (2,1) circle (2pt) (4,1) circle (2pt) (6,1) circle (2pt) (0,2) circle (2pt) (2,2) circle (2pt) (4,2) circle (2pt) (6,2) circle (2pt) (0,3) circle (2pt) (2,3) circle (2pt) (4,3) circle (2pt) (6,3) circle (2pt) (0,4) circle (2pt) (2,4) circle (2pt) (4,4) circle (2pt) (6,4) circle (2pt);
  		\draw (0,0) -- (6,0);
  		\draw (0,1) -- (6,1);
  		\draw (0,2) -- (6,2);
  		\draw (0,3) -- (6,3);
  		\draw (0,4)-- (6,4);
  		\node at (-0.4,0) {1};
  		\node at (-0.4,1.1) {\Huge\vdots};
  		\node at (-0.5,2) {$i$};
  		\node at (-0.6,3) {\Huge\vdots};
  		\node at (6.4,0) {1};
  		\node at (6.4,1) {\Huge\vdots};
  		\node at (6.6,2) {$i$};
  		\node at (6.6,3) {\Huge\vdots};
  		\node at (3,3.2) {$x$};
  		\node at (3,2.2) {$x$};
  		\node at (3,0.2) {$1$};
  		\node at (3,1.2) {$1$};
  		\node at (6.6,4) {$n$};
  		\node at (-0.5,4) {$n$};
  		\node at (3,4.2) {$x$};
  		\filldraw (8.5,0) circle (2pt) (10.5,0) circle (2pt) (12.5,0) circle (2pt) (14.5,0) circle (2pt) (8.5,1) circle (2pt) (10.5,1) circle (2pt) (12.5,1) circle (2pt) (14.5,1) circle (2pt) (8.5,2) circle (2pt) (10.5,2) circle (2pt) (12.5,2) circle (2pt) (14.5,2) circle (2pt) (8.5,3) circle (2pt) (10.5,3) circle (2pt) (12.5,3) circle (2pt) (14.5,3) circle (2pt) (8.5,4) circle (2pt) (10.5,4) circle (2pt) (12.5,4) circle (2pt) (14.5,4) circle (2pt);
  		\draw (8.5,0) -- (14.5,0);
  		\draw (8.5,1) -- (14.5,1);
  		\draw (8.5,2) -- (14.5,2);
  		\draw (8.5,3) -- (14.5,3);
  		\draw (8.5,4)-- (14.5,4);
  		\draw (10.5,3)-- (12.5,2);
  		\node at (15,1) {\Huge\vdots};
  		\node at (8.2,1) {\Huge\vdots};
  		\node at (8.2,4) {$n$};
  		\node at (15,4) {$n$};
  		\node at (8,3.5) {\Huge\vdots};
  		\node at (8,3) {$i+1$};
  		\node at (8,2) {$i$};
  		\node at (15,2) {$i$};
  		\node at (15,3) {$i+1$};
  		\node at (8.2,0) {1};
  		\node at (15,0) {1};
  		\node at (15,3.5) {\Huge\vdots};
  		\end{tikzpicture}
  		\caption[Planar networks with $H_i(x)$ and $F_i$ as weight matrices]{Planar networks with $H_i(x)$ (left) and $F_i$ (right) as weight matrices.}
  		\label{hf}
  	\end{figure} 
  	Concatenating these planar networks in proper order, we obtain the weighted planar network with the required properties such that its weight matrix is $M(t)$.	
  \end{proof}

  \vspace{.05in} 

  \noindent We also note: 
  
  \begin{lem}\label{se}
  	The matrix $S\cdot E(e)$ (upto projectivization) is the weight matrix of a planar network with weights in the form of some products of the coordinates and their reciprocals.
  \end{lem}
  \begin{proof}
  	Let's say \[ E(e)=diag(d_1,d_2,\cdots, d_n)\cdot S=(a_{ij})_{n\times n}, \]
  	and $S=(s_{ij})_{n\times n}$, where 
  	\begin{equation*}
  	s_{ij} =
  	\begin{cases}
  	(-1)^{i+1} & \text{if $\ i+j=n+1$,} \\
  	0 & \text{otherwise.}
  	\end{cases}
  	\end{equation*}	
  	Then by direct multiplication, we can check that 
  	\begin{equation*}
  	a_{ij}=
  	\begin{cases}
  	(-1)^{i+1}d_i & \text{if $\ i+j=n+1$,}\\
  	0 & \text{otherwise.}
  	\end{cases}
  	\end{equation*}
  	Let $S\cdot E(e)=(b_{ij})$, where \[ b_{ij}=\sum_{k=1}^{n}s_{ik}\cdot a_{kj}=s_{i,n+1-i}\cdot a_{n+1-i,j}.\]
  	So when $i=j$,
  	\[ b_{ij}=b_{ii}=(-1)^{i+1}\cdot (-1)^{n+1-i+1}d_{n+1-i}=(-1)^{n+1}d_{n+1-i}. \]
  	If $i\neq j$ we have that $ n+1-i+j\ne n+1\implies a_{n+1-i,j}=0\implies b_{ij}=0$.
  	So we get,
  	\begin{equation}\label{SEgen}
  	\begin{split}
  	S\cdot E(e) &=diag((-1)^{n+1}d_n,(-1)^{n+1}d_{n-1},\cdots,(-1)^{n+1}d_1) \\
  	&=diag(d_n,d_{n-1},\cdots,d_1) \ \text{(upto projectivization)}\\
   &=diag(z_1z_2\cdots z_{n-1},\cdots,z_{n-2}z_{n-1},z_{n-1},1)\ \text{(from \cref{Egen})}.
  	\end{split}
  	\end{equation}
  	In \Cref{M}, we have seen that a diagonal matrix is the weight matrix of a planar network with weights its diagonal entries. Now the diagonal entries of $S\cdot E(e)$ are all in the form of a product of some Fock-Goncharov coordinates, as the same is true for the entries of $E(e)$. This proves the lemma.  
  \end{proof}

  \noindent From the above Lemmas, we have the following result for one of the building blocks:
  
  \begin{prop}\label{et}
  	The matrix $T(t)E(e)$ is the weight matrix of a weighted planar network with  each  weight being a  product of some  Fock-Goncharov coordinates  and their reciprocals. 
  \end{prop}
  \begin{proof}
  	From \cref{tgen}, we obtain
  	\begin{equation}\label{TtEe}
  	T(t)E(e)=M(t)\cdot S\cdot E(e).
  	\end{equation} 
  	
 Hence concatenating the weighted planar networks (see \Cref{prp:1}) of \Cref{M} and \Cref{se}, we obtain a weighted planar network  with weight matrix $T(t)E(e)$.
  \end{proof}
  
  \noindent We shall now describe the planar network of $T(t)E(e)$ more explicitly. To do so, we decompose the matrix $M(t)$ as 
  \begin{equation}\label{Stn}
  M(t)=\prod_{j=1}^{n-1}\bigg[St(n-j)(t) \bigg]=St(n-1)(t)\cdot St(n-2)(t)\cdots\cdots St(1)(t),
  \end{equation}
  where $St(k)(t)$ is the part of the product in \eqref{Mgen} corresponding to the outer index $j=n-k$. For instance, $St(1)(t)= F_{n-1}$ and $St(2)(t)=F_{n-1}H_1(t)F_{n-2}$. From the construction of the matrix $M(t)$, we know that each of these $St(k)(t)$ matrices essentially maps the snake $AC$ a step closer to the snake $AB$.
  \begin{figure}[ht]
  	\centering
  	\begin{tikzpicture}[scale=0.6]
  	\draw (0,0)-- (2.31,4)-- (4.62,0)-- (0,0);
  	\draw (1.15,0)--(2.89,3) --(1.73,3) --(3.46,0) --(4.04,1)--(0.58,1) --(1.15,0);
  	\draw (1.15,2)-- (3.46,2)-- (2.31,0)-- (1.15,2);
  	\draw[green, line width=0.4mm] (0,0)-- (2.31,4);
  	\draw (5,2.5) node {$\longrightarrow$} (5,3) node {$St(1)$} (-0.4,0) node {$A$} (2.31,4.4) node {$C$} (5,0) node {$B$};
  	\draw (6,0)-- (8.31,4)-- (10.62,0)-- (6,0);
  	\draw (7.15,0)--(8.89,3) --(7.73,3) --(9.46,0) --(10.04,1)--(6.58,1) --(7.15,0);
  	\draw (7.15,2)-- (9.46,2)-- (8.31,0)-- (7.15,2);
  	\draw[green, line width=0.4mm] (6,0)-- (7.73,3)-- (8.89,3);
  	\draw (11,2.5) node {$\longrightarrow$} (11,3) node {$St(2)$};
  	\draw (12,0)-- (14.31,4)-- (16.62,0)-- (12,0);
  	\draw (13.15,0)--(14.89,3) --(13.73,3) --(15.46,0) --(16.04,1)--(12.58,1) --(13.15,0);
  	\draw (13.15,2)-- (15.46,2)-- (14.31,0)-- (13.15,2);
  	\draw[green, line width=0.4mm] (12,0)-- (13.15,2)-- (15.46,2);
  	\end{tikzpicture}   	 	
  	\caption{$St(1)$ and $St(2)$ respectively in $M(t)$.}
  	\label{St12}
  \end{figure}
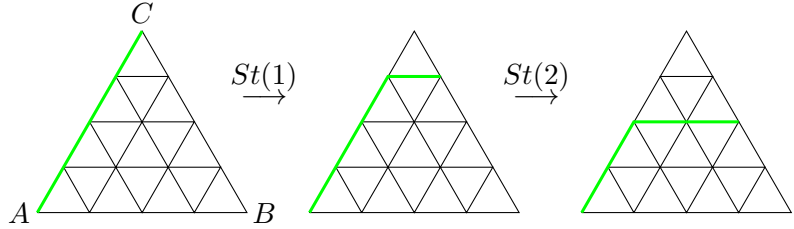
  In \Cref{St12}, the change of snakes for $St(1)$ and $St(2)$ have been shown. Also notice that, only those triangle invariants are involved in $St(k)(t)$ which are associated in that particular snake-change. For example, there are no triangle invariants associated in $St(1)$, so it is a constant matrix. The following proposition talks about these $St(k)(t)$ matrices, and in the proof, we also describe how these arise as weight matrices of certain weighted planar networks.
  \begin{prop}\label{stk}
  	The matrix $St(k)(t)$ can be written as 
  	\[ St(k)(t)=
  	\begin{pmatrix}
  	I_{n-k \times n-k} & 0_{n-k \times k} \\
  	A_{k \times n-k} & B_{k \times k}
  	\end{pmatrix},
  	\]
  	where $A=(a_{ij})_{k \times n-k}$ is given by
  	\begin{equation*}
  	a_{ij}=
  	\begin{cases}
  	1 & \text{if $j=n-k$,} \\
  	0 & \text{otherwise}
  	\end{cases}
  	\end{equation*}
  	and 
  	\[ B_{k\times k}=
  	\begin{pmatrix}
  	1 & 0&0&0&\cdots &0\\
  	1 & x_{k-1} & 0&0&\cdots &0\\
  	0&x_{k-1}&x_{k-1}x_{k-2} & 0&\cdots &0\\
  	\cdots &\cdots &\cdots &\cdots &\cdots &\cdots \\
  	0&\cdots &\cdots &0& x_{k-1}x_{k-2}\cdots x_2&x_{k-1}x_{k-2}\cdots x_1
  	\end{pmatrix}.
  	\]
  \end{prop}
  \begin{proof}
  	From \cref{Mgen}, we get 
  	\begin{align}\label{stkMat}
  	St(k)(t) &= F_{n-1}\prod_{i=1}^{k-1}\bigg(H_i(X_{i-1,k-i-1,n-k-1})F_{n-i-1}\bigg) \nonumber \\
  	&= F_{n-1}H_1(x_1)F_{n-2}H_2(x_2)\cdots\cdots H_{k-1}(x_{k-1})F_{n-k}
  	\end{align}
  	where $x_i=X_{i-1,k-i-1,n-k-1}$. In \Cref{St1Con}, we see planar networks such that their weight matrices are the matrices in \cref{stkMat} in that order.
  	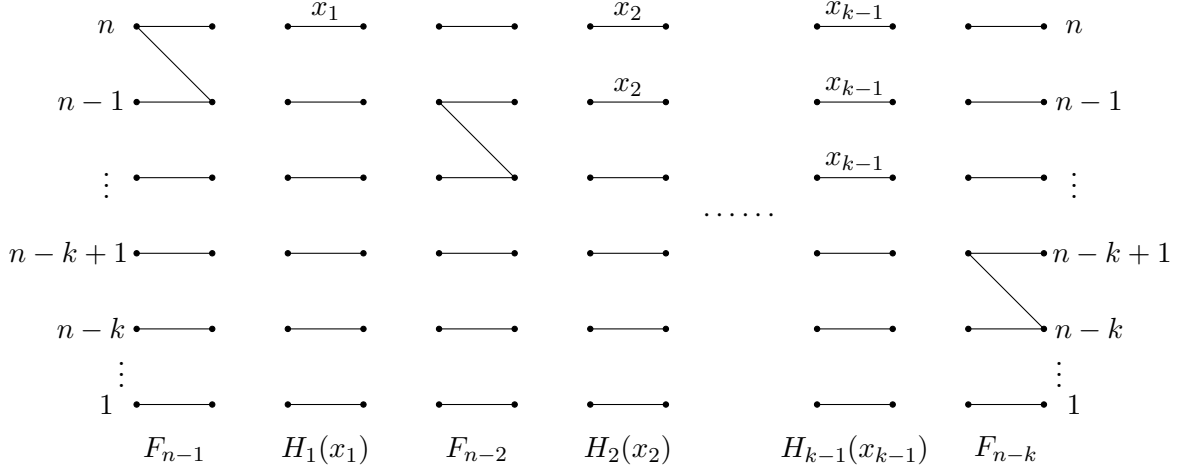
\begin{figure}[ht]
  		\centering
  		\begin{tikzpicture}
  		\draw (0,0)--(1,0);
  		\draw (0,1)--(1,1);
  		\draw (0,2)--(1,2);
  		\draw (0,3)--(1,3);
  		\draw (0,4)--(1,4)--(0,5);
  		\draw (0,5)--(1,5);
  		\filldraw (0,0) circle (1pt) (1,0) circle (1pt) (0,1) circle (1pt) (1,1) circle (1pt) (0,2) circle (1pt) (1,2) circle (1pt) (0,3) circle (1pt) (1,3) circle (1pt) (0,4) circle (1pt) (1,4) circle (1pt) (0,5) circle (1pt) (1,5) circle (1pt);
  		\draw (-.4,0) node {$1$};
  		\draw (-.4,5) node {$n$};
  		\draw (-.6,4) node {$n-1$};
  		\draw (-.6,1) node {$n-k$};
  		\draw (-.9,2) node {$n-k+1$};
  		\draw (-.2,0.5) node {$\vdots$};
  		\draw (-.4,3) node {$\vdots$};
  		\draw (12.4,0) node {$1$};
  		\draw (12.4,5) node {$n$};
  		\draw (12.6,4) node {$n-1$};
  		\draw (12.6,1) node {$n-k$};
  		\draw (12.9,2) node {$n-k+1$};
  		\draw (12.2,0.5) node {$\vdots$};
  		\draw (12.4,3) node {$\vdots$};
  		\draw (0.5,-0.6) node {$F_{n-1}$} (2.5,-0.6) node {$H_1(x_1)$} (4.5,-0.6) node {$F_{n-2}$} (6.5,-0.6) node {$H_2(x_2)$} (9.5,-0.6) node {$H_{k-1}(x_{k-1})$} (11.5,-0.6) node {$F_{n-k}$};
  		\draw (2,0)--(3,0);
  		\draw (2,1)--(3,1);
  		\draw (2,2)--(3,2);
  		\draw (2,3)--(3,3);
  		\draw (2,4)--(3,4);
  		\draw (2,5)--(3,5);
  		\filldraw (2,0) circle (1pt) (3,0) circle (1pt) (2,1) circle (1pt) (3,1) circle (1pt) (2,2) circle (1pt) (3,2) circle (1pt) (2,3) circle (1pt) (3,3) circle (1pt) (2,4) circle (1pt) (3,4) circle (1pt) (2,5) circle (1pt) (3,5) circle (1pt);
  		\draw (2.5,5.2) node {$x_1$};
  		\draw (4,0)--(5,0);
  		\draw (4,1)--(5,1);
  		\draw (4,2)--(5,2);
  		\draw (4,3)--(5,3)--(4,4);
  		\draw (4,4)--(5,4);
  		\draw (4,5)--(5,5);
  		\filldraw (4,0) circle (1pt) (5,0) circle (1pt) (4,1) circle (1pt) (5,1) circle (1pt) (4,2) circle (1pt) (5,2) circle (1pt) (4,3) circle (1pt) (5,3) circle (1pt) (4,4) circle (1pt) (5,4) circle (1pt) (4,5) circle (1pt) (5,5) circle (1pt);
  		\draw (6,0)--(7,0);
  		\draw (6,1)--(7,1);
  		\draw (6,2)--(7,2);
  		\draw (6,3)--(7,3);
  		\draw (6,4)--(7,4);
  		\draw (6,5)--(7,5);
  		\filldraw (6,0) circle (1pt) (7,0) circle (1pt) (6,1) circle (1pt) (7,1) circle (1pt) (6,2) circle (1pt) (7,2) circle (1pt) (6,3) circle (1pt) (7,3) circle (1pt) (6,4) circle (1pt) (7,4) circle (1pt) (6,5) circle (1pt) (7,5) circle (1pt);
  		\draw (6.5,5.2) node {$x_2$} (6.5,4.2) node {$x_2$};
  		\draw (8,2.5) node {$\cdots\cdots$};
  		\draw (9,0)--(10,0);
  		\draw (9,1)--(10,1);
  		\draw (9,2)--(10,2);
  		\draw (9,3)--(10,3);
  		\draw (9,4)--(10,4);
  		\draw (9,5)--(10,5);
  		\filldraw (9,0) circle (1pt) (10,0) circle (1pt) (9,1) circle (1pt) (10,1) circle (1pt) (9,2) circle (1pt) (10,2) circle (1pt) (9,3) circle (1pt) (10,3) circle (1pt) (9,4) circle (1pt) (10,4) circle (1pt) (9,5) circle (1pt) (10,5) circle (1pt);
  		\draw (9.5,5.2) node {$x_{k-1}$} (9.5,4.2) node {$x_{k-1}$} (9.5,3.2) node {$x_{k-1}$};
  		\draw (11,0)--(12,0);
  		\draw (11,1)--(12,1)--(11,2);
  		\draw (11,2)--(12,2);
  		\draw (11,3)--(12,3);
  		\draw (11,4)--(12,4);
  		\draw (11,5)--(12,5);
  		\filldraw (11,0) circle (1pt) (12,0) circle (1pt) (11,1) circle (1pt) (12,1) circle (1pt) (11,2) circle (1pt) (12,2) circle (1pt) (11,3) circle (1pt) (12,3) circle (1pt) (11,4) circle (1pt) (12,4) circle (1pt) (11,5) circle (1pt) (12,5) circle (1pt);
  		\end{tikzpicture}
  		\caption{Planar networks of the matrices in \cref{stkMat} in their respective order.}
  		\label{St1Con}
  	\end{figure}
  	It can be noted that, concatenating each network in \Cref{St1Con}, we obtain the weighted planar network in \Cref{concatSt}.

  	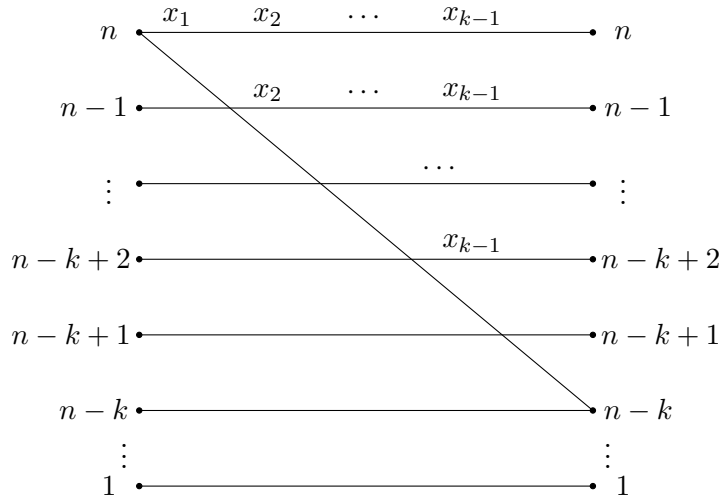
\begin{figure}[ht]
  		\centering
  		\begin{tikzpicture}
  		\draw (0,0)--(6,0) -- (0,5);
  		\draw (0,-1)--(6,-1);
  		\draw (0,1)--(6,1);
  		\draw (0,2)--(6,2);
  		\draw (0,3)--(6,3);
  		\draw (0,4)--(6,4);
  		\draw (0,5)--(6,5);
  		\filldraw (0,0) circle (1pt) (6,0) circle (1pt) (0,-1) circle (1pt) (6,-1) circle (1pt) (0,1) circle (1pt) (6,1) circle (1pt)(0,2) circle (1pt) (6,2) circle (1pt) (0,3) circle (1pt) (6,3) circle (1pt) (0,4) circle (1pt) (6,4) circle (1pt) (0,5) circle (1pt) (6,5) circle (1pt);
  		\draw (-.4,-1) node {$1$};
  		\draw (-.4,5) node {$n$};
  		\draw (-.6,4) node {$n-1$};
  		\draw (-.6,0) node {$n-k$};
  		\draw (-.9,1) node {$n-k+1$};
  		\draw (-.9,2) node {$n-k+2$};
  		\draw (-.2,-0.5) node {$\vdots$};
  		\draw (-.4,3) node {$\vdots$};
  		\draw (6.4,-1) node {$1$};
  		\draw (6.4,5) node {$n$};
  		\draw (6.6,4) node {$n-1$};
  		\draw (6.6,0) node {$n-k$};
  		\draw (6.9,1) node {$n-k+1$};
  		\draw (6.9,2) node {$n-k+2$};
  		\draw (6.2,-0.5) node {$\vdots$};
  		\draw (6.4,3) node {$\vdots$};
  		\draw (0.5,5.2) node {$x_1$} (1.7,5.2) node {$x_2$} (1.7,4.2) node {$x_2$} (3,5.2) node {$\cdots$} (3,4.2) node {$\cdots$} (4,3.2) node {$\cdots$} (4.4,2.2) node {$x_{k-1}$} (4.4,4.2) node {$x_{k-1}$} (4.4,5.2) node {$x_{k-1}$};
  		\end{tikzpicture}
  		\caption{Concatenated planar network of $St(k)(t)$ with weight matrix  \cref{stkMat}.}
  		\label{concatSt}
  	\end{figure}
  	By \Cref{prp:1} the weight matrix of the planar network in \Cref{concatSt} is exactly as claimed in the proposition.
  \end{proof}
  Now using \cref{SEgen}, proof of \Cref{stk} and the \cref{stkMat}, we construct the planar network of $T(t)E(e)$ from \cref{TtEe} in \Cref{TE},
  \begin{figure}[ht]
  	\centering
  	\begin{tikzpicture}
  	\draw (0,0)-- (12,0);
  	\draw (0,1)-- (12,1);
  	\draw (0,2)-- (12,2);
  	\draw (0,3)-- (12,3);
  	\draw (0,4)-- (12,4);
  	\draw (0,5)-- (12,5);
  	\draw (0,5)-- (3,0);
  	\draw (10,5)-- (10.6,4);
  	\draw (8,5)-- (9.2,3);
  	\draw (6,5)-- (7.8,2);
  	\draw (3,5)-- (5.4,1);
  	\fill [color=black] (0,0) circle (1.5pt) (12,0) circle (1.5pt) (0,1) circle (1.5pt) (12,1) circle (1.5pt) (0,2) circle (1.5pt) (12,2) circle (1.5pt) (0,3) circle (1.5pt) (12,3) circle (1.5pt) (0,4) circle (1.5pt) (12,5) circle (1.5pt) (0,5) circle (1.5pt) (12,4) circle (1.5pt) (8,5) circle (1.5pt) (6,5) circle (1.5pt) (3,5) circle (1.5pt) (5.4,1) circle (1.5pt) (7.8,2) circle (1.5pt) (9.2,3) circle (1.5pt) (10.6,4) circle (1.5pt) (10,5) circle (1.5pt) (3,0) circle (1.5pt);
  	\draw (-.4,0) node {$1$};
  	\draw (-.4,5) node {$n$};
  	\draw (-.6,4) node {$n-1$};
  	\draw (-.4,1) node {$2$};
  	\draw (-.4,2) node {$\vdots$};
  	\draw (-.4,3) node {$\vdots$};
  	\draw (1.5,5.2) node {$x_{1_1} \cdots x_{1_{n-2}}$};
  	\draw (1.8,4.2) node {$x_{1_2}\cdots x_{1_{n-2}}$};
  	\draw (2.5,3.2) node {$\cdots$};
  	\draw (3,2.2) node {$x_{1_{n-2}}$};
  	\draw (4.5,5.2) node {$x_{2_1} \cdots x_{2_{n-3}}$};
  	\draw (4.8,4.2) node {$x_{2_2} \cdots x_{2_{n-3}}$};
  	\draw (5.5,3.2) node {$\cdots$};
  	\draw (7,5.2) node {$\cdots$};
  	\draw (7.5,4.2) node {$\cdots$};
  	\draw (9,5.2) node {$x_{n-2}$} (11.2,5.2) node {$1$} (11.2,4.2) node {$z_{n-1}$} (11,0.2) node {$z_1\cdots z_{n-1}$} (11,1.2) node {$z_2\cdots z_{n-1}$} (11.1,2.2) node {$\cdots$} (11.2,3.2) node {$\cdots$};
  	\draw (12.4,0) node {$1$};
  	\draw (12.4,5) node {$n$};
  	\draw (12.6,4) node {$n-1$};
  	\draw (12.4,1) node {$2$};
  	\draw (12.4,2) node {$\vdots$};
  	\draw (12.4,3) node {$\vdots$};
  	\end{tikzpicture}
  	\caption{Planar network of $T(t)E(e)$.}
  	\label{TE}
  \end{figure}
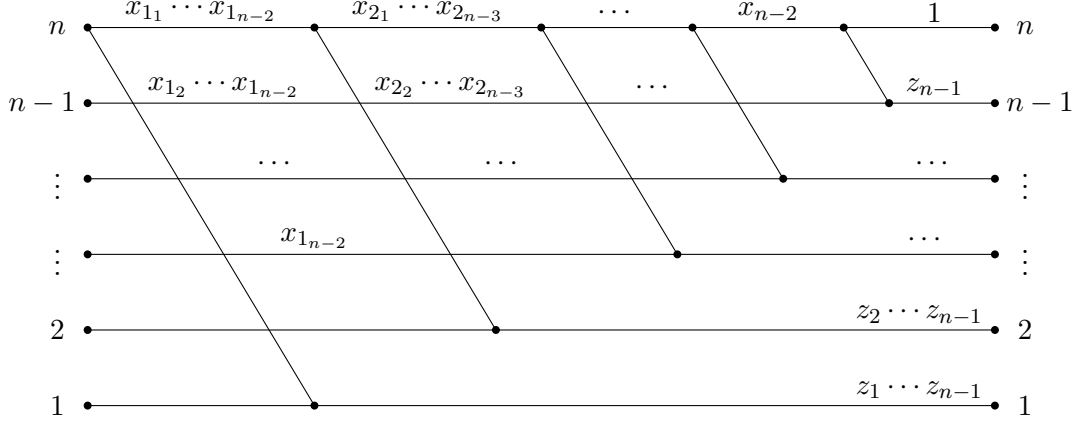
  where $x_{i_j}$ are triangle invariants for all $i,j$ and $z_k$ are edge invariants for all $k$. Note that for all $j,\ x_{i_j}$ are the triangle invariants associated to $St(n-i)$.

  \vspace{.1in}
  
  We shall now produce a weighted planar network corresponding to  the other building block $T(t)^{-1}E(e)$. 
  Since \[ T(t)^{-1}E(e)=S^{-1}M(t)^{-1}E(e) \]
  this is different from the previous case we have to construct a weighted planar network whose weight matrix is the \textit{inverse} matrix $M(t)^{-1}$. 
  This shall require more work since this could have \textit{negative} entries, in particular, entries that are $-1$ (see \eqref{Minv} and \eqref{fk}).  Recall that in our proof, we would like the weights of the planar networks to have the form of a product of the Fock-Goncharov coordinates and their reciprocals, but $-1$ is not in that form. In what follows we shall adapt the idea above of breaking $M(t)$ into ``steps" and use an inductive  procedure to construct a weighted planar network for $T(t)^{-1}E(e)$.

  \vspace{.05in}

  \noindent We start with the observation:

  \begin{clm}\label{SS}
  		$S\cdot S=I_n$, the identity matrix when $n$ is odd and $-I_n$ when $n$ is even.
  	\end{clm} 	 
  	\begin{proof}
  		Let $S=(s_{ij})_{n\times n}$ where 
  		\begin{equation*}
  		s_{ij} =
  		\begin{cases}
  		(-1)^{i+1} & \text{iff $\ i+j=n+1$,} \\
  		0 & \text{otherwise.}
  		\end{cases}
  		\end{equation*}	
  		Let $S\cdot S=(a_{ij})_{n\times n}\cdot (b_{ij})_{n\times n}=(c_{ij})_{n\times n}$. Then \[ c_{ij}=\sum_{p=1}^{n}a_{ip}b_{pj}=a_{i,n+1-i}\cdot b_{n+1-i,j}. \]
  		So when $i=j$, \[ c_{ij}=(-1)^{i+1}\cdot (-1)^{n+1-i+1}=(-1)^{n+3}=(-1)^{n+1}, \]
  		and for $i\neq j$, \[ c_{ij}=0 \ \ \text{   since } b_{n+1-i,j}=0. \]
  	\end{proof}	

	\noindent Hence upto projectivization, we can write
  	\begin{equation}\label{tinve}
  	T(t)^{-1}E(e)=S\cdot M(t)^{-1}\cdot S\cdot S\cdot E(e).
  	\end{equation}
  	From \Cref{se}, we know that $S\cdot E(e)$ has a planar network with the required properties. We shall show that $S\cdot M(t)^{-1}\cdot S$ also has a planar network with the required properties.

   \vspace{.05in}
   
  	Let $\Delta ABC$ be the lift of the triangle that contains the lift of the base point of $\gamma$ (see \Cref{stp1}). From the construction of the matrix $M(t)$ we know that $M(t)^{-1}$ maps the snake $AB$ to $AC$. We can easily check that 
  	\begin{equation}\label{Minv}
  	M(t)^{-1}= \prod_{j=n-1}^{1}\bigg[\prod_{i=n-j-1}^{1}\bigg(f_{n-i-1}h_i(X_{i-1,n-i-j-1,j-1})\bigg)f_{n-1}\bigg],
  	\end{equation}
  	where $h_k(x)=H_k(x)^{-1}=diag(1,1,\cdots,1,\underbrace{x^{-1},\cdots, x^{-1}}_{k})$ and 
  	\begin{equation}\label{fk} f_k=F_k^{-1}=\begin{pmatrix}
  	1 & 0 & & \cdots &  & 0 \\
  	0 & \ddots   \\
  	\vdots & \ddots & 1 & 0 & & \vdots \\
  	& & -1 & 1 & \\
  	\vdots & & & \ddots & \ddots & 0 \\
  	0 & \cdots & & & 0 & 1 
  	\end{pmatrix} 
  	\ \
  	\begin{matrix}\text{(i.e. 1 on the diagonals and the only}\\ \text{ non-diagonal -1 at $(k+1,k)$-th position).} \end{matrix}
  	\end{equation}

    \noindent We decompose the matrix $M(t)^{-1}$ as a product of $(n-1)$ matrices in the following way:
  	\begin{equation}\label{Mtinv-decomp} M(t)^{-1}= Step(n-1)(t)\cdot Step(n-2)(t)\cdots Step1(t)
   \end{equation}
  	where $Step(k)(t)$ is the part of the product in the expression \eqref{Minv} corresponding to the index $j=n-k$, i.e.,
  	\begin{equation}\label{stpk}
  	Step(k)(t)= \prod_{i=k-1}^{1}\bigg(f_{n-i-1}h_i(X_{i-1,k-i-1,n-k-1})\bigg)f_{n-1}.
  	\end{equation}
  	For instance, \[ Step(n-1)(t)=f_{n-1};\ Step(n-2)(t)=f_{n-2}h_1(x)f_{n-1}\] for some $x$ etc. In \Cref{stp1}, the change of snakes for $Step1$ and $Step2$ matrices have been shown.
  	\begin{figure}[ht]
  		\centering
  		\begin{tikzpicture}[scale=0.6]
  		\draw[smooth,samples=100,domain=2.3094010767585034:4.618802153517007] plot(\x,{0-sqrt(3)*(\x)+8});
  		\draw[smooth,samples=100,domain=0.0:2.3094010767585034] plot(\x,{sqrt(3)*(\x)});
  		\draw[smooth,samples=100,domain=0.0:4.618802153517007] plot(\x,{0});
  		\draw[smooth,samples=100,domain=0.5773502691896258:4.041451884327381] plot(\x,{1});
  		\draw[smooth,samples=100,domain=1.1547005383792517:3.4641016151377544] plot(\x,{2});
  		\draw[smooth,samples=100,domain=1.7320508075688772:2.886751345948129] plot(\x,{3});
  		\draw[smooth,samples=100,domain=1.1547005383792517:2.88675] plot(\x,{sqrt(3)*(\x)-2});
  		\draw[smooth,samples=100,domain=2.3094:3.4641016151377544] plot(\x,{sqrt(3)*(\x)-4});
  		\draw[smooth,samples=100,domain=3.4641016151377544:4.041451884327381] plot(\x,{sqrt(3)*(\x)-6});
  		\draw[smooth,samples=100,domain=1.73205:3.4641016151377544] plot(\x,{0-sqrt(3)*(\x)+6});
  		\draw[smooth,samples=100,domain=1.1547005383792517:2.3094010767585034] plot(\x,{0-sqrt(3)*(\x)+4});
  		\draw[smooth,samples=100,domain=0.57735:1.1547005383792517] plot(\x,{0-sqrt(3)*(\x)+2});
  		
  		\draw (-0.5,0) node {$A$};
  		\draw (2.5,4.5) node {$C$};
  		\draw (5,0) node {$B$};
  		\draw (0,0)-- (0.58,1);
  		\draw (0.58,1)-- (4.04,1);
  		\draw [line width=1.6pt,color=green] (6.58,1)-- (10.04,1);
  		\draw (6,0)-- (8.31,4);
  		\draw (8.31,4)-- (10.62,0);
  		\draw (6,0)-- (10.62,0);
  		\draw (6.58,1)-- (7.15,0);
  		\draw (7.15,2)-- (8.31,0);
  		\draw (7.73,3)-- (9.46,0);
  		\draw (8.89,3)-- (7.15,0);
  		\draw (9.46,2)-- (8.31,0);
  		\draw (10.04,1)-- (9.46,0);
  		\draw (7.15,2)-- (9.46,2);
  		\draw (7.73,3)-- (8.89,3);
  		\draw (12,0)-- (14.31,4);
  		\draw (12,0)-- (16.62,0);
  		\draw (14.31,4)-- (16.62,0);
  		\draw (12.58,1)-- (16.04,1);
  		\draw (13.15,2)-- (15.46,2);
  		\draw (13.73,3)-- (14.89,3);
  		\draw (12.58,1)-- (13.15,0);
  		\draw (13.15,2)-- (14.31,0);
  		\draw (13.73,3)-- (15.46,0);
  		\draw (14.89,3)-- (13.15,0);
  		\draw (15.46,2)-- (14.31,0);
  		\draw (16.04,1)-- (15.46,0);
  		\draw [line width=1.6pt,color=green] (6,0)-- (6.58,1);
  		\draw [line width=1.6pt,color=green] (12,0)-- (13.15,2);
  		\draw [line width=1.6pt,color=green] (13.15,2)-- (15.46,2);
  		\draw [line width=1.6pt,color=green] (0,0)-- (4.62,0);
  		\draw (5.4,2.6) node {$Step1$};
  		\draw (5.32,2) node {$\longrightarrow$};
  		\draw (11.2,2.6) node {$Step2$};
  		\draw (11.16,2) node {$\longrightarrow$};
  		\end{tikzpicture}   	 	
  		\caption{$Step1$ and $Step2$ in $M(t)^{-1}$.}
  		\label{stp1}
  	\end{figure}
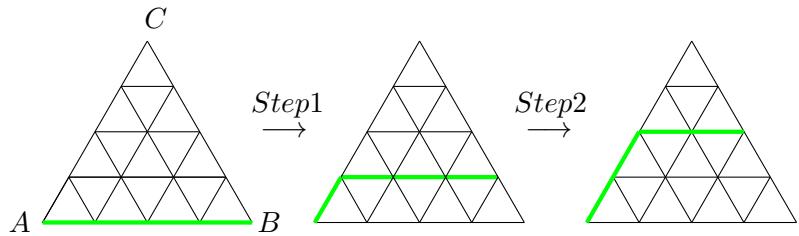 
  	\begin{rmk}
  		Note that the snake moves in opposite direction for the matrix $Step(i)$ defined in \cref{stpk} to that of the matrix $St(j)$ defined in \Cref{Stn}. So, 
  		\[ Step(n-k)(t)\cdot St(k)(t) = Id_n \] for all $1\le k\le (n-1)$.
  	\end{rmk}
  	\noindent Using the Claim \ref{SS}, we can write 
  	\begin{equation}\label{eq:sms}
  	S\cdot M(t)^{-1} \cdot S= S\cdot Step(n-1)(t)\cdot S \cdot S \cdot Step(n-2)(t)\cdot S \cdots S \cdot Step1(t) \cdot S,
  	\end{equation}
  	upto projectivization. Now we shall produce planar networks for all the matrices $S\cdot Step(k)(t) \cdot S$ with the required properties and concatenate them in order to get a planar network for the matrix $S\cdot M(t)^{-1} \cdot S$.

  	\begin{lem}\label{step1}
  		The matrix $\text{Step}1(t)$ is given as \[ \text{Step}1(t)= \begin{pmatrix}
  		1 & 0&0&0& \cdots&0 \\
  		-1 & 1&0&0& \cdots & 0 \\
  		0& -x_1^{-1} & x_1^{-1} &0& \cdots &0 \\
  		0&0& -x_1^{-1} x_2^{-1} & x_1^{-1}x_2^{-1} & \cdots &0 \\
  		\cdots& \cdots& \cdots& \cdots& \cdots& \cdots& \\
  		0&\cdots & \cdots &0&-x_1^{-1}x_2^{-1}\cdots x_{n-2}^{-1} &x_1^{-1}x_2^{-1}\cdots x_{n-2}^{-1}
  		\end{pmatrix} \]
  		for some triangle invariants $x_1,\ x_2\ \cdots x_{n-2}$.
  	\end{lem}
  	\begin{proof}
  		Let us introduce the notation of pre-superscript to denote the dimension of a square matrix from now on, namely $^k\!M$ will denote that $M$ is an $k\times k$ matrix. We shall use induction on $n$ to prove the claim. For $n=3$, 
  		$$^3\text{Step1}(t)=^3\!\!f_1 \cdot ^3\!\!h_1(x_1)\cdot ^3\!\!f_2=\begin{pmatrix}
  		1 & 0& 0\\
  		-1 & 1& 0 \\
  		0& -x_1^{-1}& x_1^{-1}
  		\end{pmatrix}$$
        
        which is already in the desired form. Now suppose it is true for $n=k$. Notice that, we can write $h_i$ and $f_i$ as block matrices in the following way:
  		\[ ^{(k+1)}\!h_{i}(x)= \begin{pmatrix}
  		^k\!h_{i-1}(x) & 0_{k\times 1} \\
  		0_{1\times k} & x^{-1}
  		\end{pmatrix} \text{  and  } ^{(k+1)}\!\!f_j = \begin{pmatrix}
  		^k\!f_j & 0_{k\times 1} \\
  		0_{1\times k} & 1
  		\end{pmatrix} \] for $j<k$, where $^{(k+1)}\!h_{i}(x)$ and $^{(k+1)}\!\!f_j$ are $(k+1) \times (k+1)$ matrix as in \cref{Minv}. \\
        
    \noindent Then from \cref{stpk}, we obtain, for $n=k+1$,
        \begin{equation*}
            \begin{split}
                &^{(k+1)}\!\text{Step}1(t) = ^{(k+1)}\!\!f_1 \cdot ^{(k+1)}\!\!h_{k-1}(x_1) \cdot ^{(k+1)}\!\!f_2 \cdot ^{(k+1)}\!\!h_{k-2}(x_2) \cdots   ^{(k+1)}\!\!f_{k-1}\cdot ^{(k+1)}\!\!h_1(x_{k-1})\cdot ^{(k+1)}\!\!f_k \\
                &=\begin{pmatrix}
  		^k\!f_1 & 0_{k\times 1} \\
  		0_{1\times k} & 1
  		\end{pmatrix}
  		\begin{pmatrix}
  		^k\!h_{k-2}(x_1) & 0_{k\times 1} \\
  		0_{1\times k} & x_1^{-1}
  		\end{pmatrix}
  		\begin{pmatrix}
  		^k\!f_2 & 0_{k\times 1} \\
  		0_{1\times k} & 1
  		\end{pmatrix}
  		\begin{pmatrix}
  		^k\!h_{k-3}(x_2) & 0_{k\times 1} \\
  		0_{1\times k} & x_2^{-1}
  		\end{pmatrix}
  		\cdots \cdots \\
        &\cdots \cdots
  		\begin{pmatrix}
  		^k\!f_{k-1} & 0_{k\times 1} \\
  		0_{1\times k} & 1
  		\end{pmatrix}
  		\begin{pmatrix}
  		^k\!h_0(x_{k-1}) & 0_{k\times 1} \\
  		0_{1\times k} & x_{k-1}^{-1}
  		\end{pmatrix}
  		\begin{pmatrix}
  		^k\!I & 0_{k\times 1} \\
  		(0,0,\cdots , 0,-1)_{1\times k} & 1
  		\end{pmatrix} \\
            \end{split}
        \end{equation*}
  		\begin{equation*}
  		\begin{split}
  		&=\begin{pmatrix}
  		^k\!f_1 \cdot ^k\!h_{k-2}(x_1) & 0_{k\times 1} \\
  		0_{1 \times k} & x_1^{-1}
  		\end{pmatrix}
  		\begin{pmatrix}
  		^k\!f_2 & 0_{k\times 1} \\
  		0_{1\times k} & 1
  		\end{pmatrix}
  		\begin{pmatrix}
  		^k\!h_{k-3}(x_2) & 0_{k\times 1} \\
  		0_{1\times k} & x_2^{-1}
  		\end{pmatrix}
  		\cdots \cdots \\
  		&\cdots \cdots
  		\begin{pmatrix}
  		^k\!f_{k-1} & 0_{k\times 1} \\
  		0_{1\times k} & 1
  		\end{pmatrix}
  		\begin{pmatrix}
  		^k\!I & 0_{k\times 1} \\
  		0_{1\times k} & x_{k-1}^{-1}
  		\end{pmatrix}
  		\begin{pmatrix}
  		^k\!I & 0_{k\times 1} \\
  		(0,0,\cdots , 0,-1)_{1\times k} & 1
  		\end{pmatrix} \\
  		&=\begin{pmatrix}
  		^k\!f_1 \cdot ^k\!h_{k-2}(x_1)\cdot ^k\!f_2\cdot ^k\!h_{k-3}(x_2) & 0_{k\times 1} \\
  		0_{1\times k} & x_1^{-1}x_2^{-1}
  		\end{pmatrix} \cdots 
  		\begin{pmatrix}
  		^k\!f_{k-1} & 0_{k\times 1} \\
  		0_{1\times k} & x_{k-1}^{-1}
  		\end{pmatrix}
  		\begin{pmatrix}
  		^k\!I & 0_{k\times 1} \\
  		(0,0,\cdots , 0,-1)_{1\times k} & 1
  		\end{pmatrix} \\
  		&= \begin{pmatrix}
  		^k\!f_1 \cdot ^k\!h_{k-2}(x_1)\cdots ^k\!f_{k-1} & 0_{k\times 1} \\
  		(0,0,\cdots , 0,-x_1^{-1}x_2^{-1}\cdots x_{k-1}^{-1})_{1\times k} & x_1^{-1}x_2^{-1}\cdots x_{k-1}^{-1}
  		\end{pmatrix}\\
  		&=\begin{pmatrix}
  		1 & 0&0&0& \cdots&0 \\
  		-1 & 1&0&0& \cdots & 0 \\
  		0& -x_1^{-1} & x_1^{-1} &0& \cdots &0 \\
  		0&0& -x_1^{-1} x_2^{-1} & x_1^{-1}x_2^{-1} & \cdots &0 \\
  		\cdots& \cdots& \cdots& \cdots& \cdots& \cdots& \\
  		0&\cdots & \cdots &0&-x_1^{-1}x_2^{-1}\cdots x_{k-1}^{-1} &x_1^{-1}x_2^{-1}\cdots x_{k-1}^{-1}
  		\end{pmatrix}
  		\begin{pmatrix}
  		\text{by} \\
  		\text{induction}\\
  		\text{hypothesis}
  		\end{pmatrix}.
  		\end{split}
  		\end{equation*} 
  	\end{proof}
  	\begin{lem} \label{sss}
  		The matrix $S \cdot \text{Step}1(t)\cdot S$ is given as 
  		\[S \cdot \text{Step}1(t)\cdot S=\begin{pmatrix}
  		x_1^{-1}x_2^{-1}\cdots x_{k-1}^{-1} & x_1^{-1}x_2^{-1}\cdots x_{k-1}^{-1} & 0 & \cdots & \cdots & 0 \\
  		\cdots & \cdots & \cdots & \cdots & \cdots & \cdots \\
  		0&\cdots &x_1^{-1}x_2^{-1}&x_1^{-1}x_2^{-1}&0&0\\
  		0& \cdots & 0 & x_1^{-1} & x_1^{-1}& 0\\
  		0 & \cdots & \cdots &0&1& 1\\
  		0 & \cdots & \cdots &0&0& 1
  		\end{pmatrix}\] 
  		upto projectivization. It is the weight matrix of the planar network in \Cref{sstep1s}.
  		
  		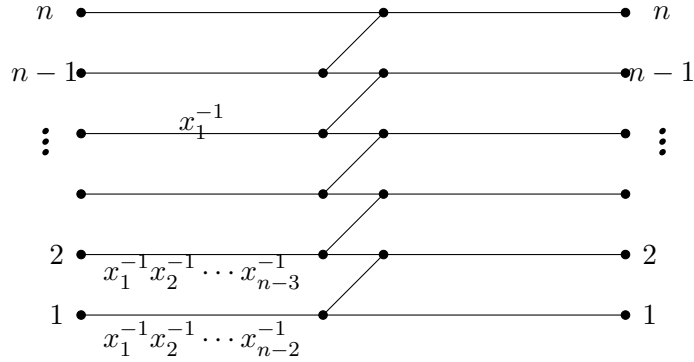
\begin{figure}[ht]
  			\centering
  			\begin{tikzpicture}[scale=0.8]
  			\filldraw (0,0) circle (2pt) (4,0) circle (2pt) (9,0) circle (2pt) (0,1) circle (2pt)  (4,1) circle (2pt) (5,1) circle (2pt) (4,2) circle (2pt) (5,2) circle (2pt) (4,3) circle (2pt) (5,3) circle (2pt) (4,4) circle (2pt) (5,4) circle (2pt) (5,5) circle (2pt) (0,2) circle (2pt) (0,3) circle (2pt) (0,4) circle (2pt) (0,5) circle (2pt) (9,1) circle (2pt) (9,2) circle (2pt) (9,3) circle (2pt) (9,4) circle (2pt) (9,5) circle (2pt);
  			\draw (0,0)-- (9,0);
  			\draw (0,1)-- (9,1);
  			\draw (0,2)-- (9,2);
  			\draw (0,3)-- (9,3);
  			\draw (0,4)-- (9,4);
  			\draw (0,5)-- (9,5);
  			\draw (4,0)-- (5,1);
  			\draw (4,1)-- (5,2);
  			\draw (4,3)-- (5,4);
  			\draw (4,2)-- (5,3);
  			\draw (4,4)-- (5,5);
  			\node at (-.4,0) {$1$};
  			\node at (9.4,0) {$1$};
  			\node at (-.4,1) {$2$};
  			\node at (9.4,1) {$2$};
  			\node at (-.6,5) {$n$};
  			\node at (9.6,5) {$n$};
  			\node at (-.6,4) {$n-1$};
  			\node at (9.6,4) {$n-1$};
  			\node at (9.6,3) {\Huge\vdots};
  			\node at (-.6,3) {\Huge\vdots};
  			\node at (2,-.4) {$x_1^{-1}x_2^{-1}\cdots x_{n-2}^{-1}$};
  			\node at (2,.7) {$x_1^{-1}x_2^{-1}\cdots x_{n-3}^{-1}$};
  			\node at (2,3.2) {$x_1^{-1}$};
  			\end{tikzpicture}
  			\caption{The planar network with weight matrix $S \cdot \text{Step}1(t)\cdot S$.}
  			\label{sstep1s}
  		\end{figure}
  	\end{lem}
  	It can be observed that the rows and columns of this matrix are in reversed order as that of $\text{Step}1(t)$ without the minus signs.
  	\begin{proof}
  		To prove the lemma, we observe that
  		\begin{itemize}
  			\item If \[S_0= \begin{pmatrix}
  			0 & \cdots & \cdots & 0 & 1 \\
  			\vdots & & \reflectbox{$\ddots$} & 1 & 0\\
  			\vdots & \reflectbox{$\ddots$} & \reflectbox{$\ddots$} & \reflectbox{$\ddots$} & \vdots \\
  			0 & \reflectbox{$\ddots$} & \reflectbox{$\ddots$} & & \vdots \\
  			1 & 0 & \cdots & \cdots & 0
  			\end{pmatrix}, \]
  			and $S_1= diag(1,-1,1, \cdots,(-1)^{n+1})$ are two square matrices of order $n$. Then \[S_1\cdot S_0=S,\] i.e., \vspace*{-3pt}
  		\end{itemize}
  		\[ \begin{bmatrix}
  		(-1)^2 & 0 & \cdots & \cdots & 0 \\
  		0 & (-1)^3 & \ddots & & \vdots \\
  		\vdots & \ddots & \ddots & \ddots & \vdots \\
  		\vdots & & \ddots & \ddots & 0 \\
  		0 & \cdots & \cdots & 0 & (-1)^{n+1}
  		\end{bmatrix}
  		\begin{bmatrix}
  		0 & \cdots & \cdots & 0 & 1 \\
  		\vdots & & \reflectbox{$\ddots$} & 1 & 0\\
  		\vdots & \reflectbox{$\ddots$} & \reflectbox{$\ddots$} & \reflectbox{$\ddots$} & \vdots \\
  		0 & \reflectbox{$\ddots$} & \reflectbox{$\ddots$} & & \vdots \\
  		1 & 0 & \cdots & \cdots & 0
  		\end{bmatrix} =\begin{bmatrix}
  		0 & \cdots & \cdots & 0 & 1 \\
  		\vdots & & \reflectbox{$\ddots$} & -1 & 0\\
  		\vdots & \reflectbox{$\ddots$} & \reflectbox{$\ddots$} & \reflectbox{$\ddots$} & \vdots \\
  		0 & \reflectbox{$\ddots$} & \reflectbox{$\ddots$} & & \vdots \\
  		(-1)^{n+1} & 0 & \cdots & \cdots & 0
  		\end{bmatrix}.
  		\]
  		
  		\begin{itemize}
  			\item Let $S_2=diag((-1)^{n+1},(-1)^{(n)}, \cdots , (-1)^3,(-1)^2$ be a diagonal matrix of order $n$. Then \[S_0\cdot S_2=S,\] i.e.,
  		\end{itemize} 
  		\[ \begin{bmatrix}
  		0 & \cdots & \cdots & 0 & 1 \\
  		\vdots & & \reflectbox{$\ddots$} & 1 & 0\\
  		\vdots & \reflectbox{$\ddots$} & \reflectbox{$\ddots$} & \reflectbox{$\ddots$} & \vdots \\
  		0 & \reflectbox{$\ddots$} & \reflectbox{$\ddots$} & & \vdots \\
  		1 & 0 & \cdots & \cdots & 0
  		\end{bmatrix} \cdot 
  		\begin{bmatrix}
  		(-1)^{n+1} & 0 & \cdots & \cdots & 0 \\
  		0 & (-1)^{n} & \ddots & & \vdots \\
  		\vdots & \ddots & \ddots & \ddots & \vdots \\
  		\vdots & & \ddots & \ddots & 0 \\
  		0 & \cdots & \cdots & 0 & (-1)^2
  		\end{bmatrix}
  		= \begin{bmatrix}
  		0 & \cdots & \cdots & 0 & 1 \\
  		\vdots & & \reflectbox{$\ddots$} & -1 & 0\\
  		\vdots & \reflectbox{$\ddots$} & \reflectbox{$\ddots$} & \reflectbox{$\ddots$} & \vdots \\
  		0 & \reflectbox{$\ddots$} & \reflectbox{$\ddots$} & & \vdots \\
  		(-1)^{n+1} & 0 & \cdots & \cdots & 0
  		\end{bmatrix}. \]
  		Using the above two facts, we obtain
  		\[ S \cdot \text{Step}1(t)\cdot S=S_0\cdot S_2 \cdot \text{Step}1(t)\cdot S_1\cdot S_0. \]
  		Now, in order to prove \Cref{sss}, we first prove that
  		\[ S_2 \cdot \text{Step}1(t)\cdot S_1=
  		\begin{pmatrix}
  		1 & 0&0&0& \cdots&0 \\
  		1 & 1&0&0& \cdots & 0 \\
  		0& x_1^{-1} & x_1^{-1} &0& \cdots &0 \\
  		0&0& x_1^{-1} x_2^{-1} & x_1^{-1}x_2^{-1} & \cdots &0 \\
  		\cdots& \cdots& \cdots& \cdots& \cdots& \cdots& \\
  		0&\cdots & \cdots &0& x_1^{-1}x_2^{-1}\cdots x_{n-2}^{-1} &x_1^{-1}x_2^{-1}\cdots x_{n-2}^{-1}
  		\end{pmatrix} \] upto projectivization. It can be done as follows : 
  		\begin{align*}
  		&\text{Step}1(t)\cdot S_1= \\
  		&\begin{pmatrix}
  		1 & 0&0&0& \cdots&0 \\
  		-1 & 1&0&0& \cdots & 0 \\
  		0& \frac{-1}{x_1} & \frac{1}{x_1} &0& \cdots &0 \\
  		0&0& \frac{-1}{x_1x_2} & \frac{-1}{x_1x_2} & \cdots &0 \\
  		\cdots& \cdots& \cdots& \cdots& \cdots& \cdots& \\
  		0&\cdots & \cdots &0&\frac{-1}{x_1x_2\cdots x_{n-2}} &\frac{-1}{x_1x_2\cdots x_{n-2}}
  		\end{pmatrix}
  		\begin{pmatrix}
  		(-1)^2 & 0 & \cdots & \cdots & 0 \\
  		0 & (-1)^3 & \ddots & & \vdots \\
  		\vdots & \ddots & \ddots & \ddots & \vdots \\
  		\vdots & & \ddots & \ddots & 0 \\
  		0 & \cdots & \cdots & 0 & (-1)^{n+1}
  		\end{pmatrix} \\
  		&= \begin{pmatrix}
  		1.(-1)^2 & 0&0& \cdots&0 \\
  		-1.(-1)^2 & 1.(-1)^3 &0& \cdots & 0 \\
  		0& \frac{-(-1)^3}{x_1} & \frac{(-1)^4}{x_1} & \cdots &0 \\
  		0&0& \frac{-(-1)^4}{x_1 x_2} &  \cdots &0 \\
  		\cdots& \cdots& \cdots& \cdots& \cdots& \\
  		0&\cdots & \cdots &\frac{-(-1)^n}{x_1x_2\cdots x_{n-2}} & \frac{(-1)^{n+1}}{x_1x_2\cdots x_{n-2}}
  		\end{pmatrix} \\
  		& =P, \text{ say.}
  		\end{align*}

  		\noindent Now, 	
  		\begin{align*}
  		&S_2 \cdot P = \begin{pmatrix}
  		1.(-1)^{n+3} & 0&0& \cdots&0 \\
  		-1.(-1)^{n+2} & 1.(-1)^{n+3} &0& \cdots & 0 \\
  		0& \frac{(-1)^{n+2}}{-x_1} & \frac{(-1)^{n+3}}{x_1} & \cdots &0 \\
  		0&0& \frac{(-1)^{n+2}}{-x_1 x_2} &  \cdots &0 \\
  		\cdots& \cdots& \cdots& \cdots& \cdots& \\
  		0&\cdots & \cdots & \frac{(-1)^{n+2}}{-x_1x_2\cdots x_{n-2}} & \frac{(-1)^{n+3}}{x_1x_2\cdots x_{n-2}}
  		\end{pmatrix} \\
  		&=(-1)^{n+3}.\begin{pmatrix}
  		1 & 0&0&0& \cdots&0 \\
  		1 & 1&0&0& \cdots & 0 \\
  		0& \frac{1}{x_1} & \frac{1}{x_1} &0& \cdots &0 \\
  		0&0& \frac{1}{x_1 x_2} & \frac{1}{x_1x_2} & \cdots &0 \\
  		\cdots& \cdots& \cdots& \cdots& \cdots& \cdots& \\
  		0&\cdots & \cdots &0& \frac{1}{x_1x_2\cdots x_{n-2}} &\frac{1}{x_1x_2\cdots x_{n-2}}
  		\end{pmatrix}.
  		\end{align*} 
  		Since $S_0$ is a permutation matrix that reverses the entries of a vector, pre-multiplication and post-multiplication by $S_0$ reverse the order of the rows and columns of the matrix $S_2 \cdot \text{Step}1(t)\cdot S_1 $ respectively. This completes the proof of the \Cref{sss}. We can easily check that $S \cdot \text{Step}1(t)\cdot S$ is the weight matrix of the planar network in \Cref{sstep1s}.	
  	\end{proof} 
\noindent  	Now observe that the matrix $^n\text{Step2}$ can be written as a block matrix in terms of $^{n-1}\text{Step1}$ (possibly with different invariants) in the following way: 
  	\[ ^n\text{Step2}=
  	\begin{pmatrix}
  	1 & 0_{1\times (n-1)} \\
  	0_{(n-1) \times 1} & ^{n-1}\text{Step1}
  	\end{pmatrix}.  \]
  	Similarly for all $1\leq k<n$, we can write
  	\begin{equation}\label{eq:stpk}
  	^n\text{Step(k)}=
  	\begin{pmatrix}
  	1 & 0_{1\times (n-1)} \\
  	0_{(n-1) \times 1} & ^{n-1}\text{Step(k-1)}
  	\end{pmatrix}=\cdots 
  	=\begin{pmatrix}
  	I_{k-1} & 0_{(k-1)\times(n-k+1)} \\
  	0_{(n-k+1)\times (k-1)} & ^{(n-k+1)}Step1
  	\end{pmatrix}.
  	\end{equation}
  	
\noindent and by \Cref{step1}  we have:

  		\[ ^n\!\text{Step(k)}=
  		\begin{pmatrix}
  		1&0&\cdots &&&&&\cdots &0 \\
  		0&1&0& \cdots &&&&\cdots & 0 \\
  		\vdots &\cdots & \ddots & \cdots &&&&\cdots & 0 \\
  		0 & \cdots &0& 1_{(k,k)} & 0&\cdots && \cdots&0 \\
  		0 & \cdots &0& -1 & 1&0&\cdots& \cdots & 0 \\
  		0 & \cdots &\cdots&0& \frac{-1}{x_1} & \frac{1}{x_1} &0& \cdots &0 \\
  		0 & \cdots &&\cdots&0& \frac{-1}{x_1 x_2} & \frac{1}{x_1x_2} & \cdots &0 \\
  		\cdots &\cdots& \cdots& \cdots& \cdots& \cdots& \cdots& \cdots& \cdots& \\
  		0&\cdots &&&& \cdots &0&\frac{-1}{x_1x_2\cdots x_{n-k-1}} &\frac{1}{x_1x_2\cdots x_{n-k-1}}
  		\end{pmatrix} \]
  		for some triangle invariants $x_1, x_2, \cdots, x_{n-k-1} $.
    
    \noindent As in the proof of \Cref{sss}, we can compute:
    
      \begin{align*}
  		& ^n\!S\cdot ^n\text{Step(k)} \cdot ^n\!S= \\
  		&\begin{pmatrix}
  		\frac{1}{x_1x_2\cdots x_{n-k-1}}& \frac{1}{x_1x_2\cdots x_{n-k-1}}&0&\cdots &&&&&0 \\
  		\cdots & \cdots & \cdots &&&& \cdots & \cdots & \cdots \\
  		0&\cdots & \frac{1}{x_1x_2} & \frac{1}{x_1x_2} &0& \cdots &&&0 \\
  		0&\cdots & 0 & \frac{1}{x_1} & \frac{1}{x_1}& 0& \cdots &&0\\
  		0& \cdots & & 0&1&1&0&\cdots & 0\\
  		0&\cdots & & \cdots &0&1_{(n+1-k,n+1-k)}&0&\cdots &0\\
  		0& \cdots & &&&\cdots & \ddots & \cdots & \vdots\\
  		0&\cdots&&&&\cdots&0&1&0\\
  		0&\cdots &&&&& \cdots & 0 & 1
  		\end{pmatrix}
  		\end{align*}
  	
   \noindent This is the weight matrix of the weighted planar network shown in \Cref{fig_stpk}.
  	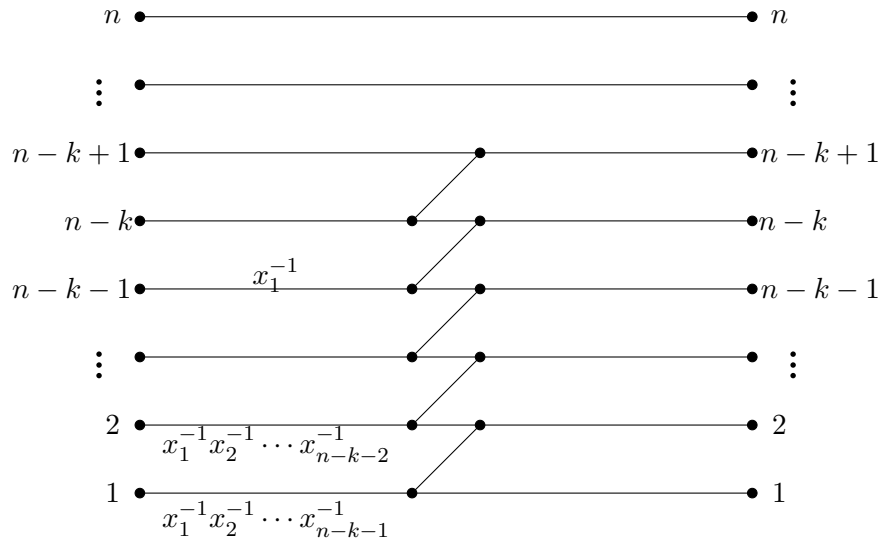
\begin{figure}[ht]
  		\centering
  		\begin{tikzpicture}[scale=0.9]
  		\filldraw (0,0) circle (2pt) (4,0) circle (2pt) (9,0) circle (2pt) (0,1) circle (2pt)  (4,1) circle (2pt) (5,1) circle (2pt) (4,2) circle (2pt) (5,2) circle (2pt) (4,3) circle (2pt) (5,3) circle (2pt) (4,4) circle (2pt) (5,4) circle (2pt) (5,5) circle (2pt) (0,2) circle (2pt) (0,3) circle (2pt) (0,4) circle (2pt) (0,5) circle (2pt) (9,1) circle (2pt) (9,2) circle (2pt) (9,3) circle (2pt) (9,4) circle (2pt) (9,5) circle (2pt) (9,6) circle (2pt) (9,7) circle (2pt) (0,6) circle (2pt) (0,7) circle (2pt);
  		\draw (0,0)-- (9,0);
  		\draw (0,1)-- (9,1);
  		\draw (0,2)-- (9,2);
  		\draw (0,3)-- (9,3);
  		\draw (0,4)-- (9,4);
  		\draw (0,5)-- (9,5);
  		\draw (0,6)-- (9,6);
  		\draw (0,7)-- (9,7);
  		\draw (4,0)-- (5,1);
  		\draw (4,1)-- (5,2);
  		\draw (4,3)-- (5,4);
  		\draw (4,2)-- (5,3);
  		\draw (4,4)-- (5,5);
  		\node at (-.4,0) {1};
  		\node at (9.4,0) {1};
  		\node at (-.4,1) {2};
  		\node at (9.4,1) {2};
  		\node at (-1,5) {$n-k+1$};
  		\node at (10,5) {$n-k+1$};
  		\node at (-.6,4) {$n-k$};
  		\node at (9.6,4) {$n-k$};
  		\node at (-1,3) {$n-k-1$};
  		\node at (10,3) {$n-k-1$};
  		\node at (9.6,2) {\Huge\vdots};
  		\node at (9.6,6) {\Huge\vdots};
  		\node at (-.6,2) {\Huge\vdots};
  		\node at (-.6,6) {\Huge\vdots};
  		\node at (9.4,7) {$n$};
  		\node at (-.4,7) {$n$};
  		\node at (2,-.4) {$x_1^{-1}x_2^{-1}\cdots x_{n-k-1}^{-1}$};
  		\node at (2,.7) {$x_1^{-1}x_2^{-1}\cdots x_{n-k-2}^{-1}$};
  		\node at (2,3.2) {$x_1^{-1}$};
  		\end{tikzpicture}
  		\caption{Planar network with weight matrix $^n\!S\cdot ^n\text{Step(k)} ^n\!S$ in $M(t)^{-1}$.}
  		\label{fig_stpk}
  	\end{figure}  

\vspace{.05in} 

   \noindent We can now complete: 

    \begin{prop}\label{mt}
  	The matrix $T(t)^{-1}E(e)$ is the weight matrix of a planar network with weights in the form of some products of the coordinates and their reciprocals.
  \end{prop}
  \begin{proof}
  	From the above discussion we have obtained weighted planar networks for all the matrices $S\cdot Step(k)(t) \cdot S$ of equation \eqref{eq:sms} such that the weights are in the form of a product of some Fock-Goncharov coordinates and their reciprocals. Hence by \eqref{eq:sms} we can concatenate them  to obtain a weighted planar network with weight matrix $S\cdot M(t)^{-1} \cdot S$. From \eqref{tinve} we see that one further concatenation with the weighted planar network  arising from \Cref{se} yields the desired network with weight matrix $T(t)^{-1}E(e)$. 
  \end{proof}

\vspace{.05in}

 \noindent  We can now establish the the generalization of \Cref{prop3} to any $n\geq 3$:

 \begin{prop}\label{propn} Let $\rho: \pi_1(S) \to \pslnc$ be a (generic) representation as in the beginning of this subsection. Then there exists a  positive representation $\rho_0:\pi_1(S) \to \pslnr$ that dominates $\rho$ in the Hilbert length spectrum. Moreover, the lengths of the peripheral curves are unchanged. 
\end{prop}
  \begin{proof}\label{pt1} 
  The definition of the positive representation $\rho_0$ will follows the same procedure as that in  \Cref{defn:rho0}: Since $\rho$ is generic, pick a framing $\beta$ such that the framed representation $(\rho,\beta)$ has well-defined Fock-Goncharov coordinates with respect to the ideal triangulation $\tau$. Replace each Fock-Goncharov coordinate by its modulus, and define $\rho_0$ to be the representation corresponding to the resulting real and positive coordinates via  \Cref{FGparam}.
  	
The proof that $\rho_0$ defined above dominates $\rho$ is the same as that for $n=3$, in the proof of \Cref{prop3}:
Let $\gamma \in \pi_1(S)$  and let $(\Gamma\,\omega)$ be the weighted planar network with weight matrix $\rho(\gamma)$, that exists by concatenating the weighted planar networks corresponding building blocks as obtained in Lemmas \ref{et} and \ref{mt} -- see \eqref{rhogammagen}.  It follows from the definition of $\rho_0$, and the properties of the weights in these planar networks, that if one replaces each weight by its modulus, one obtains a weighted planar network $(\Gamma, \lvert \omega \rvert)$ with weight matrix $\rho_0(\gamma)$. By \Cref{spctrl} we know that the largest eigenvalue of $\rho_0(\gamma)$ is at least the largest eigenvalue in modulus of $\rho(\gamma)$.  Similarly, by applying the preceding argument to $\gamma^{-1}$, we also have that the smallest eigenvalue of $\rho_0(\gamma)$ is at most the smallest eigenvalue in modulus of $\rho(\gamma)$. From the definition of the Hilbert length of $\gamma$ (see \eqref{rlen}) we conclude that $\ell_\rho(\gamma) \leq \ell_{\rho_0}(\gamma)$. 

As before, if $\gamma$ is peripheral then $\rho(\gamma)$ is triangular and its diagonal entries are in the form of a product of Fock-Goncharov coordinates. From the definition of $\rho_0$, the matrix $\rho_0(\gamma)$ is also triangular and each diagonal entry is the  modulus of the corresponding entry in $\rho(\gamma)$. In particular, the largest and smallest eigenvalues in modulus are exactly the same for both matrices, and we have $\ell_{\rho}(\gamma)=\ell_{\rho_0}(\gamma)$. 
  \end{proof}

  \begin{rmk}\label{benFib}
  	From the way $\rho_0$ has been defined from a given generic representation $\rho$, we can see that the absolute values of each of their Fock-Goncharov coordinates are equal, and they lie in the same bending fiber (as defined in \Cref{defn:bf}). 
  \end{rmk}
  
\section{Domination in the translation length spectrum}\label{syms}
As defined previously, let $\rho_0$ be the positive representation defined from the given generic representation $\rho$ by equipping with a framing and replacing each Fock-Goncharov coordinate with its modulus. In this section, we shall prove that $\rho_0$ also dominates $\rho$ in the translation length spectrum. We first introduce some basic notions in the theory of majorization, that we shall need, mostly referring to \cite{Bhatia}. 

\subsection{Majorization and doubly stochastic matrices}
Consider the partial order relation on $\mathbb{R}^n$ defined as follows. 
\begin{defn}\cite[\S II.1]{Bhatia}\label{def:wm}  Let   $x=(x_1,x_2,\cdots,x_n),\ y=(y_1,y_2,\cdots,y_n)\in \mathbb{R}^n$ such that $x_1\ge\cdots \ge x_n$ and $y_1\ge\cdots \ge y_n$.
	We say $x$ is \emph{majorised} by $y$ and denote it by $x \prec y$ if 
	\begin{equation}\label{pSum}
	\sum_{i=1}^{k}x_i \le \sum_{i=1}^{k}y_i
	\end{equation}
	for all $k$ with $1\le k \le n$ and 
	\begin{equation}\label{allsum}
	\sum_{i=1}^{n}x_i = \sum_{i=1}^{n}y_i.
	\end{equation}
 We say $x$ is \emph{weakly submajorised} by $y$ and denote it by $x \prec_w y$ if \cref{pSum} is satisfied.
\end{defn}

\begin{defn}
A real and non-negative matrix $A=(a_{ij})_{n\times n}$ is called \emph{doubly stochastic} if
\begin{equation*}
\sum_{i=1}^{n}a_{ij}=1 \text{ for all } j \text{ and } \sum_{j=1}^{n}a_{ij}=1 \text{ for all } i.
\end{equation*}
The matrix $A$ is called \emph{doubly substochastic} if
\begin{equation*}
\sum_{i=1}^{n}a_{ij}\le 1 \text{ for all } j \text{ and } \sum_{j=1}^{n}a_{ij}\le 1 \text{ for all } i.
\end{equation*}
\end{defn}

The following are two basic facts: 

\begin{lem}\cite[Lemma 1]{vonNeumann}, \cite[C.1. Theorem]{MOA}\label{vn}
	For every doubly substochastic matrix $P$, there exists a doubly stochastic matrix $Q$ such that $P\le Q$.
\end{lem}
\begin{lem}\cite[Theorem II.2.8]{Bhatia}\label{dsmA}
	Let $x,y \in \mathbb{R}^n_+$. Then $x\prec_w y$ if and only if there exists a doubly substochastic matrix $A$ such that $x=Ay$.
\end{lem}

\noindent In what follows, for a function $\phi: D (\subseteq \mathbb{R}) \rightarrow \mathbb{R}$ and $x=(x_1,\cdots,x_n) \in \mathbb{R}^n$, let $\phi(x)$ denote $(\phi(x_1), \cdots,\phi(x_n))$. 
\begin{lem}\label{cmiphi}
	Let $x=(x_1,\cdots,x_n),\ y=(y_1,\cdots,y_n) \in \mathbb{R}^n_+$ such that $x\prec_w y$. Then for any convex and monotonically increasing function $\phi:\mathbb{R}\rightarrow \mathbb{R}$, we have \[ \sum_{i=1}^{n}\phi(x_i)\le \sum_{i=1}^{n}\phi(y_i). \] 
\end{lem}
\begin{proof}
	The proof is similar to that of Theorem II.3.1 in \cite{Bhatia}. Let $x\prec_w y$. Then using \Cref{dsmA}, we get a doubly substochastic matrix $B=(b_{ij})_{n\times n}$ such that $x=By$. Also, using \Cref{vn}, we get a doubly stochastic matrix $A=(a_{ij})_{n\times n}$ such that $B\le A$. Now,\[ x_i=\sum_{j=1}^{n}b_{ij}y_j \le \sum_{j=1}^{n}a_{ij}y_j. \] Since $\phi$ is monotonically increasing and convex, we have 
	\[ \phi(x_i)\le \phi \bigg( \sum_{j=1}^{n}a_{ij}y_j \bigg) \le \sum_{j=1}^{n}a_{ij} \phi(y_j). \] Hence,
	\[ \sum_{i=1}^{n}\phi(x_i) \le \sum_{i=1}^{n} \sum_{j=1}^{n} a_{ij} \phi(y_j)= \sum_{j=1}^{n} \sum_{i=1}^{n} a_{ij} \phi(y_j)= \sum_{j=1}^{n} \phi(y_j), \] where the last equality follows from the fact that $A$ is a doubly stochastic matrix.
\end{proof}

\noindent We also have: 

\begin{lem}\label{log2phi}
	Let $\phi : \mathbb{R}_+ \rightarrow \mathbb{R}$ such that $\phi(e^t)$ is convex and monotone increasing in $t$. Let $x,y \in \mathbb{R}^n_+$ such that $\log x \prec_w \log y$. Then $\phi(x) \prec_w \phi(y)$.
\end{lem}
\begin{proof}
	Let $X_i=(x_1,\cdots, x_i)$ and $Y_i=(y_1,\cdots, y_i)$. Clearly, $X_n=x$ and $Y_n=y$. Now $\log X_n \prec_w \log Y_n \implies \log X_k \prec_w \log Y_k$ for all $k$ with $1\le k \le n$. Since $\phi(e^t)$ is convex and monotone increasing in $t$, using \Cref{cmiphi}, we get \[ \log X_k \prec_w \log Y_k \implies \sum_{i=1}^{k} \phi(e^{\log x_i}) \le \sum_{i=1}^{k} \phi(e^{\log y_i}) \implies \sum_{i=1}^{k} \phi(x_i) \le \sum_{i=1}^{k} \phi(y_i) \] for all $k$. Hence, $\phi(x) \prec_w \phi(y)$.
\end{proof}

\noindent We note the first application of these lemmas in the Theorem below -- see  \cite[Theorem II.3.6]{Bhatia}, and we refer the reader there for the proof. 

\begin{thm}(Weyl's Majorant Theorem)\label{wmt}
	Let $A$ be a $n\times n$ square matrix with singular values $\sigma_1\le \cdots\le \sigma_n$ and eigenvalues $\lambda_i$ such that $|\lambda_1|\le \cdots \le |\lambda_n|$. Let $\phi:\mathbb{R}_+ \rightarrow \mathbb{R}_+$ such that $\phi(e^x)$ is a monotone increasing convex function of $x$. Then \[ (\phi(|\lambda_n|),\cdots, \phi(|\lambda_1|)) \prec_w (\phi(\sigma_n),\cdots,\phi(\sigma_1)). \]
\end{thm}

\begin{cor}\label{corwmt}
	Assuming $\phi(x)=(\log x)^2$ in \Cref{wmt}, we obtain \[ \sum_{i=1}^{n}(\log|\lambda_i(A)|)^2 \le \sum_{i=1}^{n}(\log\sigma_i(A))^2, \] where $\lambda_i(A)$ are $\sigma_i(A)$ are the eigenvalues and singular values of $A$. (Here, recall that the \emph{singular values} of a matrix $A\in M_n$ are the positive square roots of the eigenvalues of the Hermitian positive semidefinite matrix $AA^*$.)
\end{cor}

\subsection{The symmetric space $\mathbb{X}_n$}\label{symspace}
Let $\mathbb{X}_n$ be the set of Hermitian positive definite $n \times n$ matrices of determinant 1. Let us consider the action of $\slnc$ on $\mathbb{X}_n$ given as 
\begin{equation}\label{grac}
A\cdot M=AMA^*
\end{equation} for all $A\in \slnc$ and $M\in \mathbb{X}_n$, where $A^*$ is the conjugate transpose of $A$. We can easily check that\[ (A\cdot M)^*=(AMA^*)^*=AM^*A^*=AMA^*=A\cdot M;\ \det(AMA^*)=1 \]  and $\langle AMA^*(v),v \rangle=\langle MA^*(v),A^*(v)\rangle >0$ since $M$ is positive definite. Also \[ (AB)\cdot M=ABM(AB)^*=ABMB^*A^*=A\cdot (B\cdot M). \] Hence, it is a group action. Now $g$ is a stabilizer of $I$ if and only if $gg^*=I$. Hence, $g$ is a unitary matrix. So we get, $Stab(I)=\sun$. Now $Orb(I)=\{ AA^* \ |\ A\in \slnc \}$. Clearly, $Orb(I)\subseteq \mathbb{X}_n$. Also, let $M\in \mathbb{X}_n$. We know that every positive definite Hermitian matrix $M$ can be written as $M=BDB^{-1}$ for some $B\in \sun$, where $D$ is a diagonal matrix with all real and positive entries and $\det(D)=1$. Then 
\[ B\sqrt{D} (B\sqrt{D})^*=B\sqrt{D}\sqrt{D}^*B^*=BDB^{-1}=M. \]
Hence $M\in Orb(I)$. So we get, $Orb(I)= \mathbb{X}_n$, and hence 
\[ \mathbb{X}_n \cong \slnc/\sun \cong (\slnc/\{\pm I_n\})/(\sun/\{\pm I_n\}) \cong \pslnc/\mathrm{PSU}(n) \] 
as a topological group. Now the fact that $\slnc$ is semisimple and $\sun$ is a maximal compact subgroup makes $\mathbb{X}_n$ a symmetric space (see \cite{Helgason}). Assuming the Killing form to be \[ \langle A,B \rangle = \frac{1}{4}\cdot trace(AB),  \] the distance function on $\mathbb{X}_n $ is estimated by
\begin{equation}\label{trln}
\tau_I(A):=d_{\mathbb{X}_n}(I_n,A\cdot I_n)=d_{\mathbb{X}_n}(I_n,AA^*)=\sqrt{(\log \sigma_1(A))^2+\cdots + (\log \sigma_n(A))^2},
\end{equation} 
where $\sigma_i(B)$'s are the singular values of $B\in \mathbb{X}_n$. It is the translation length of $A$ at the origin. Readers are referred to \cite[Proposition 26.1]{Canary} for the details. It is important to note here that in \cite{Canary}, Canary has discussed about $\slnr/\son$, but here we are dealing with $\slnc/\sun$, as the same computation works in this case also. (For another exposition, readers are referred to \cite{Schwartz}.)

\vspace{.05in} 

\noindent The following seems to be well-known (see, for example, \cite{Parreau12}), but we provide a proof for the sake of completeness.

\begin{prop}\label{Tlequal}
	The translation length of an element $A \in \slnc$ in $\mathbb{X}_n$ is given by
	\[ \ell_{\mathbb{X}_n}(A)=\sqrt{\sum_{i=1}^{n}(\log|\lambda_i(A)|)^2}, \]
	where $\lambda_i(A)$ are the eigenvalues of $A$. Here, The {translation length} $ \ell_{\mathbb{X}_n}(A):= \inf_{x\in \mathbb{X}_n} d_{\mathbb{X}_n}(x,A\cdot x)$. 
\end{prop}
\begin{proof}
	Let $A \in \slnc$ and $x\in \mathbb{X}_n$. We have already seen that $Orb(I_n)=\mathbb{X}_n$. Hence $x=BB^*$ for some $B\in \slnc$. Now,
	\begin{equation*}
	\begin{split}
	d_{\mathbb{X}_n}(x,A\cdot x)&= d_{\mathbb{X}_n}(B^{-1}\cdot x,B^{-1}\cdot(A\cdot x))\\
	&= d_{\mathbb{X}_n}(I_n,B^{-1}ABB^*A^*(B^{-1})^*)\\
	&= d_{\mathbb{X}_n}(I_n,B^{-1}AB(B^{-1}AB)^*)\\
	&= \sqrt{\sum_{i=1}^{n}(\log\sigma_i(B^{-1}AB))^2} \text{ from \cref{trln}} \\
	&\ge \sqrt{\sum_{i=1}^{n}(\log|\lambda_i(B^{-1}AB)|)^2}\\
	&= \sqrt{\sum_{i=1}^{n}(\log|\lambda_i(A)|)^2}
	\end{split}
	\end{equation*} where the inequality follows from \Cref{corwmt}. Hence we have,
	\begin{equation}\label{tauineq}
	\ell_{\mathbb{X}_n}(A)=\inf_{x\in \mathbb{X}_n} d_{\mathbb{X}_n}(x,A\cdot x)\ge \sqrt{\sum_{i=1}^{n}(\log|\lambda_i(A)|)^2}.
	\end{equation}
	Now, suppose $A$ is a diagonalizable matrix. Then there exists $P\in \slnc$ such that $P^{-1}AP=D$ for some diagonal matrix $D$. Let us choose $x_0=PP^*\in \mathbb{X}_n$. Then
	\begin{equation}\label{eqachv}
	\begin{split}
	d_{\mathbb{X}_n}(x_0,A\cdot x_0)= \sqrt{\sum_{i=1}^{n}(\log\sigma_i(P^{-1}AP))^2}= \sqrt{\sum_{i=1}^{n}(\log\sigma_i(D))^2} &= \sqrt{\sum_{i=1}^{n}(\log|\lambda_i(D)|)^2}\\
	&= \sqrt{\sum_{i=1}^{n}(\log|\lambda_i(A)|)^2}
	\end{split}
	\end{equation}
	as, we can easily check that, the singular values of a diagonal matrix are the absolute values of its eigenvalues. So from \cref{tauineq} and \cref{eqachv} we obtain, for a diagonalizable matrix $A\in \slnc$, \[ \ell_{\mathbb{X}_n}(A)= \sqrt{\sum_{i=1}^{n}(\log|\lambda_i(A)|)^2}. \] Now suppose $A$ is non-diagonalizable. Then we know that, there is a sequence of diaonalizable matrices $\{A_k \}_{k=1}^n \subseteq \slnc$ such that $\lim\limits_{k\rightarrow \infty}A_k=A$. Now for all $k$, we have,
	\begin{equation*}
	\begin{split}
	\ell_{\mathbb{X}_n}(A_k) &= \sqrt{\sum_{i=1}^{n}(\log|\lambda_i(A_k)|)^2}\\
	\implies \lim\limits_{k\rightarrow \infty} \ell_{\mathbb{X}_n}(A_k) &= \lim\limits_{k\rightarrow \infty} \sqrt{\sum_{i=1}^{n}(\log|\lambda_i(A_k)|)^2}\\
	\implies \ell_{\mathbb{X}_n}(A) &= \sqrt{\sum_{i=1}^{n}(\log|\lambda_i(A)|)^2},
	\end{split}
	\end{equation*} since $\ell_{\mathbb{X}_n}(A)$ and each $\lambda_i(A)$ is a continuous function of the entries of $A$, i.e., these are continuous functions in the space of matrices.
\end{proof}

\begin{rmk}
    The {translation length} of an element $A\in \pslnc$ is defined by considering a lift $\widetilde{A} \in \slnc$ and applying \Cref{Tlequal}. This is well-defined since if  $A_1, A_2 \in \slnc$ are two lifts of $A\in \pslnc$, then $A_1=c A_2$ for some $c\in \mathbb{C}^*$ with $|c|=1$ and consequently $|\lambda_i(A_1)|= |\lambda_i(cA_2)|=|c||\lambda_i(A_2)|=|\lambda_i(A_2)|$ for each $1\leq i\leq n$.
\end{rmk}

\noindent We also note that inequalities of translation lengths of matrices in $\pslnc$ do not depend on the choice of representative in any projective class:

\begin{lem}
	Let $A,B \in \slnc$ and $c\in \mathbb{C}^\ast$ such that $cA,cB \in \slnc$. Then \[ \ell_{\mathbb{X}_n}(A)\le \ell_{\mathbb{X}_n}(B) \iff \ell_{\mathbb{X}_n}(cA)\le \ell_{\mathbb{X}_n}(cB). \]
\end{lem}
\begin{proof}
	We know that if $\lambda_1(A),\lambda_2(A),\cdots, \lambda_n(A)$ are the eigenvalues of $A$, then the corresponding eigenvalues of $cA$ are $c \lambda_1(A),c \lambda_2(A),\cdots, c \lambda_n(A)$. Now
	\begin{align*}
	\ell_{\mathbb{X}_n}^2(cA) &= \sum_{i=1}^{n}[\log(|c \lambda_i(A)|)]^2=\sum_{i=1}^{n}[\log|c|+\log |\lambda_i(A)|]^2 \\
	&=n(\log|c|)^2 + 2\log|c| \sum_{i=1}^{n}\log |\lambda_i(A)|+ \sum_{i=1}^{n}(\log |\lambda_i(A)|^2)\\
	&=n(\log|c| )^2 + \ell_{\mathbb{X}_n}^2(A),
	\end{align*}
	since $\sum \log \lambda_i(A)=\log \prod \lambda_i(A)=\log 1=0$. This proves the lemma.
\end{proof}

\subsection{Proof of the second domination result} 
We shall now prove that, for the given generic representation $\rho:\pi_1(S)\rightarrow \pslnc$, the positive representation $\rho_0:\pi_1(S)\rightarrow \pslnr$ defined earlier dominates $\rho$ in the translation length spectrum. 

\vspace{.05in} 

\noindent \textit{Notation.} For two matrices $A=(a_{ij})_{n\times n}$ and $B=(b_{ij})_{n\times n}$, we denote $A \le B$ if $a_{ij}\le b_{ij}$ for all $i,j$.  Moreover, of $A$ is a complex matrix, let $|A|=(|a_{ij}|)_{n\times n}$. 

\vspace{.05in} 

\noindent The following is the key observation, that uses the Lindstr\"{o}m's Lemma for weighted planar networks:

\begin{lem}\label{pEgnIneq}
	Let $\gamma \in \pi_1(S)$ and $\lambda_1, \lambda_2,\cdots,\lambda_n$ be the eigenvalues of $\rho(\gamma)$ such that $|\lambda_1|\le |\lambda_2|\le \cdots \le |\lambda_n|$. Let $\lambda'_1, \lambda'_2,\cdots,\lambda'_n$ be the eigenvalues of $\rho_0(\gamma)$ with $|\lambda'_1|\le |\lambda'_2|\le \cdots \le |\lambda'_n|$. Then \[ \prod_{i=1}^{k}|\lambda_{n-i+1}| \le \prod_{i=1}^{k}|\lambda'_{n-i+1}| \] for all $k$ with $1\le k \le n$.
\end{lem}
\begin{proof}
	Using \Cref{spctrl}, we already know that $|\lambda_n|\le |\lambda'_n|$. Let us denote $\rho(\gamma)$ and $\rho_0(\gamma)$ by $A$ and $A'$ respectively in this proof. Let us consider the $k$-th antisymmetric power or the $k$-th Grassmann power of $A$, denoted by $\wedge^k A$ and defined on antisymmetric tensors $x_1\wedge \cdots \wedge x_k$ as 
	\begin{equation}
	\wedge^k A (x_1\wedge \cdots \wedge x_k) := Ax_1 \wedge \cdots \wedge Ax_k.
	\end{equation}
	Clear $\lambda_{i_1} \lambda_{i_2}\cdots \lambda_{i_k}$ are the eigenvalues of $\wedge^k A$ with $|\lambda_{n-k+1}\cdots \lambda_n|$ being its spectral radius. The entries of $\wedge^k A$ are the minors of $A$ of order $k$ (see \cite[\S I.5]{Bhatia} for details). Let $\Delta_{I,J}$ and $\Delta'_{I,J}$ be the $\{I,J\}$ minors of $A$ and $A'$ respectively, where $|I|=|J|=k\le n$. 	Let $f_1,f_2,\cdots f_j$ be the family of vertex disjoint paths from sources $I$ to sinks $J$ in the planar network of $A$ with weights $w_1,w_2,\cdots w_j$ respectively. Then $|w_1|,|w_2|,\cdots,|w_j|$ are the weights of those families in the planar network of $A'$. Then from Lindstr\"om's Lemma (see \Cref{linds}), we obtain  $\Delta_{I,J}=w_1+w_2+\cdots+w_j$, and $\Delta'_{I,J}=|w_1|+|w_2|+\cdots+|w_j|$. From the triangle inequality of complex numbers, we can conclude that $|\Delta_{I,J}|\le \Delta'_{I,J}$. Here we recall that $A'$ is a totally positive matrix unless $\gamma$ is a peripheral curve, in which case $A'$ is totally non-negative. That is, $\Delta'_{I,J}$ is real and positive for all $I,J$ if $\gamma$ is non-peripheral. Hence we have $|\wedge^k A|\le \wedge^k A'$. Now using \Cref{spctrl} once again, we conclude that $|\lambda_{n-k+1}\cdots \lambda_n| \le |\lambda'_{n-k+1}\cdots \lambda'_n|$.
\end{proof}   

\vspace{.05in}

\begin{proof}[Proof of \Cref{mainthm}]

    By \Cref{propn} it remains to prove the domination statement for the translation length spectrum, in particular, that $\rho_0$ dominates $\rho$, where $\rho_0$ is obtained by replacing each Fock-Goncharov coordinate by its modulus.

	The idea of this proof is similar to that of Weyl's Majorant Theorem (see \cite[Theorem II.3.6]{Bhatia}). Let $\gamma \in \pi_1(S)$ and $\lambda_1, \lambda_2,\cdots,\lambda_n$ be the eigenvalues of $\rho(\gamma)$ such that $|\lambda_1|\le |\lambda_2|\le \cdots \le |\lambda_n|$. Let $\lambda'_1, \lambda'_2,\cdots,\lambda'_n$ be the eigenvalues of $\rho_0(\gamma)$ with $|\lambda'_1|\le |\lambda'_2|\le \cdots \le |\lambda'_n|$. Let $\lambda = ( \lambda_n, \lambda_{n-1},\cdots,\lambda_1)$ and $\lambda'=( \lambda'_n, \lambda'_{n-1},\cdots,\lambda'_1)$. From \Cref{pEgnIneq}, we get \[ \prod_{i=1}^{k}|\lambda_{n-i+1}| \le \prod_{i=1}^{k}|\lambda'_{n-i+1}| \] for all $k$ with $1\le k \le n$. Applying $\log$ on both sides, we obtain
	\[ \sum_{i=1}^{k}\log|\lambda_{n-i+1}| \le \sum_{i=1}^{k}\log|\lambda'_{n-i+1}| \] for all $k$ with $1\le k \le n$. That is, $\log |\lambda| \prec_w \log |\lambda'|$ (see \Cref{def:wm}). 
 
 Now we shall apply \Cref{log2phi}. Consider the function $\phi: \mathbb{R}^n_+ \rightarrow \mathbb{R}$ such that $\phi(x) = (\log x)^2$. Clearly, $\phi(e^x)=x^2$ is monotone increasing and convex in $x$. Hence we get, $\phi(|\lambda|) \prec_w \phi(|\lambda'|)$, i.e., $(\log|\lambda|)^2 \prec_w (\log|\lambda'|)^2 $. So, we have \[ \sum_{i=1}^{k}(\log|\lambda_{n-i+1}|)^2 \le \sum_{i=1}^{k}(\log|\lambda'_{n-i+1}|)^2 \] for all $k$ with $1\le k \le n$. In particular, taking $k=n$, we obtain
	\[ \ell_{\mathbb{X}_n}(\rho(\gamma))=\sqrt{\sum_{i=1}^{n}(\log|\lambda_{i}|)^2} \le \sqrt{\sum_{i=1}^{n}(\log|\lambda'_{i}|)^2}=\ell_{\mathbb{X}_n}(\rho_0(\gamma)) \]
 where $\ell_{\mathbb{X}_n}(A)$ is the translation length in $\mathbb{X}_n$ (see \Cref{Tlequal}). 
 
	We can note here that if $\gamma$ is a peripheral curve in $\pi_1(S)$, then $\rho(\gamma)$ is a triangular matrix and its diagonal entries are in the form of product of some Fock-Goncharov coordinates. Hence the absolute values of the diagonal entries of $\rho(\gamma)$ and $\rho_0(\gamma)$ are identical. Since the eigenvalues of a triangular matrix are precisely its diagonal entries, we obtain $\ell_{\mathbb{X}_n}(\rho(\gamma))= \ell_{\mathbb{X}_n}(\rho_0(\gamma))$ for any peripheral curve $\gamma$.
\end{proof}
\begin{rmk}
    We can also prove a domination result for the length functions 
    \[ \ell_k(\rho(\gamma)):=\log\bigg| \frac{\lambda_{n-k+1}\cdots \lambda_n}{\lambda_1\cdots \lambda_k} \bigg| \] for $1\le k < n$, analogous to the length function for the $k$-th fundamental weight, as defined in \cite{Burg} (see also \cite{MZ}) in the context of $\pslnr$-representations. 
    Note that for $k=1$ we recover the Hilbert length \eqref{rlen}. Now, from \Cref{pEgnIneq}, we obtain $$|\lambda_{n-k+1}\cdots \lambda_n|\le |\lambda'_{n-k+1}\cdots \lambda'_n| \text{ and } |\lambda_{k+1}\cdots \lambda_n| \le |\lambda'_{k+1}\cdots \lambda'_n|.$$ Moreover, since we can always choose a representative of $\rho(\gamma)$ in $\slnc$, we have $$\prod_{i=1}^{n}|\lambda_i| = \prod_{i=1}^{n}|\lambda'_i| =1 \implies |\lambda_1 \cdots \lambda_k| \ge |\lambda'_1 \cdots \lambda'_k|$$ which readily yields \[ \ell_k(\rho(\gamma))=\log\bigg| \frac{\lambda_{n-k+1}\cdots \lambda_n}{\lambda_1\cdots \lambda_k} \bigg| \le \log\bigg| \frac{\lambda'_{n-k+1}\cdots \lambda'_n}{\lambda'_1\cdots \lambda'_k} \bigg|= \ell_k(\rho_0(\gamma)) \] for all $1\le k<n$. 
\end{rmk}

 \bibliographystyle{alpha}
 \bibliography{references}
 
\end{document}